\documentclass{amsart}
\usepackage{amssymb}
\usepackage{amsfonts}
\usepackage{amsmath}
\usepackage{psfig}
\usepackage{psfrag}
\usepackage{graphicx}

\newtheorem{theorem}{Theorem}[section]
\newtheorem{lemma}[theorem]{Lemma}
\newtheorem{proposition}[theorem]{Proposition}
\newtheorem{definition}[theorem]{Definition}
\newtheorem{example}[theorem]{Example}

\newtheorem{remark}[theorem]{Remark}

\numberwithin{equation}{section}

\begin{document}

\title{Pattern generation problems arising in multiplicative integer systems}


\author{Jung-Chao Ban$^{\star}$}
\address{Department of Applied Mathematics, National Dong Hwa University, Hualien 97401, Taiwan}
\email{jcban@mail.ndhu.edu.tw}
\thanks{$^{\star}$The first author would like to thank the National Science Council, R.O.C. (Contract No. NSC 100-2115-M-259-009-MY2) and the National Center for Theoretical Sciences for partially supporting this research.}

\author{Wen-Guei Hu$^{\star\star}$}
\address{Department of Applied Mathematics, National Chiao Tung University, Hsinchu 300, Taiwan}
\email{wghu@mail.nctu.edu.tw}
\thanks{$^{\star\star}$The second author would like to thank the National Science Council, R.O.C. and
the ST Yau Center for partially supporting this research.}

\author{Song-Sun Lin$^{\dagger}$}
\address{Department of Applied Mathematics, National Chiao Tung University, Hsinchu 300, Taiwan}
\email{sslin@math.nctu.edu.tw}
\thanks{$^{\dagger}$The third author would like to thank the National Science Council, R.O.C. (Contract No. NSC 98-2115-M-009) and
the ST Yau Center for partially supporting this research.}

%

\begin{abstract}
This study investigates a multiplicative integer system by using a method that was developed for studying pattern generation problems. The entropy and the Minkowski dimensions of general multiplicative systems can thus be computed. A multi-dimensional decoupled system is investigated in three main steps. (I) Identify the admissible lattices of the system; (II) compute the density of copies of admissible lattices of the same length, and (III)
compute the number of admissible patterns on the admissible lattices.

 A coupled system can be decoupled by removing the multiplicative relation set and then performing procedures similar to those applied to a decoupled system. The admissible lattices are chosen to be the maximum graphs of different degrees which are mutually independent. The entropy can be obtained after the remaining error term is shown to approach zero as the degree of the admissible lattice tends to infinity.

\end{abstract}

\maketitle

\section{Introduction}

\hspace{0.4cm}Multiplicative integer systems have been intensively studied in recent years; see \cite{14-1,17,18,26,32,33,42,43,44} and the references therein. One of the main related issues is to compute Minkowski (box) dimension and Hausdorff dimension of such systems and to compare them. These two dimensions are equal in an additive shift. However, for most known examples of multiplicative integer systems, they are different. Since the computations of these two dimensions are difficult, effective methods for computing these dimensions for general multiplicative systems must be developed.

This investigation is motivated directly by the work of Kenyon \emph{et al.} \cite{33}, who utilized a variational method to obtain the results on

\begin{equation}\label{eqn:1.1}
\mathbb{X}_{2}^{0}=\left\{
(x_{1},x_{2},x_{3},\cdots )\in\{0,1\}^{\mathbb{N}} \mid x_{k}x_{2k}=0 \text{ for all } k\geq 1
 \right\},
\end{equation}
and also pointed out that the method fails for the system

\begin{equation}\label{eqn:1.0-2}
\mathbb{X}_{2,3}^{0}=\left\{
(x_{1},x_{2},x_{3},\cdots )\in\{0,1\}^{\mathbb{N}} \mid x_{k}x_{2k}x_{3k}=0 \text{ for all } k\geq 1
 \right\}.
\end{equation}
This work provides an approach to general multiplicative systems, including (\ref{eqn:1.1}) and (\ref{eqn:1.0-2}).

This study emphasizes the computation of entropy $h(\mathbb{X})$ of the multiplicative system $\mathbb{X}$. Let $N$ be the number of the symbols of system $\mathbb{X}$. The Minkowski dimension $\dim_{M}(\mathbb{X})$ is given by

\begin{equation}\label{eqn:1.0-3}
\dim_{M}(\mathbb{X})= \frac{1}{\log N}h(\mathbb{X}),
\end{equation}
where

\begin{equation}\label{eqn:1.0-4}
h(\mathbb{X})=\underset{m\rightarrow \infty}{\lim}\hspace{0.1cm}\frac{1}{m}\log|X_{m}|,
\end{equation}
$X_{m}$ is the set of $m$-blocks in $\mathbb{X}$ and $|X_{m}|$ is the number of $X_{m}$.

Definition (\ref{eqn:1.0-4}) is standard for symbolic dynamical systems \cite{38}, and it specifies the growth rate of $|X_{m}|$. However, $\mathbb{X}$ is not invariant under the shift map $\sigma $.
Bowen \cite{7-1} used $(n,\varepsilon )$-spanning and $(n,\varepsilon )$%
-separation sets to define the topological entropy $h_{top}^{B}(T,Z)$ for
an arbitrary set $Z$ in a topological dynamical system $(X,T)$. If $X$ is a
shift space with the shift map $\sigma _{X}$, then the topological entropy $%
h_{top}^{B}(\sigma _{X},X)$ equals the entropy $h(X)$ \cite{38}.
Hence, $h_{top}^{B}(\sigma _{\mathbb{X}},\mathbb{X})=h(\mathbb{X})$. Recently,
Feng and Huang \cite{18-1} defined several topological entropies of
subsets $Z$ in a topological dynamical system  $(X,T)$. In particular, they defined the upper capacity topological
entropy $h_{top}^{UC}(T,Z)$ and packing topological entropy
$h_{top}^{P}(T,Z)$. The variational principle for those entropies are proved
therein. For any $Z\subseteq X$, $h_{top}^{B}(T,Z)\leq
h_{top}^{P}(T,Z)\leq h_{top}^{UC}(T,Z)$ \cite{18-1}. If
$Z$ is $T$-invariant and compact, then they are coincident.

The multi-dimensional shifts of finite type have been studied intensively; see \cite{11,12,13,14,20,21,22,23,35,36,37-1,38,39,45,47,53,53-1} and the references therein. Among them, the authors studied pattern generation problems on multi-dimensional shifts of finite type and developed some efficient means of studying the generation of admissible patterns, and  then computing the topological entropy; see \cite{1,2,3,4,5,6,7,24,25,27,37}. This study shows that these methods can be used to study the multi-dimensional decoupled systems and one-dimensional coupled systems of multiplicative integers, including $\mathbb{X}_{2,3}^{0}$.

To illustrate our method, (\ref{eqn:1.1}) is investigated first.
The topological entropy $h(\mathbb{X}_{2}^{0})$ has been shown to be

\begin{equation}\label{eqn:1.2}
h(\mathbb{X}_{2}^{0})=\underset{k=1}{\overset{\infty}{\sum}}\frac{1}{2^{k+1}}\log a_{k},
\end{equation}
where $a_{k}$ is a Fibonacci number with $a_{1}=2$, $a_{2}=3$ and $a_{k+1}=a_{k}+a_{k-1}$ for all $k\geq 2$  \cite{17}. The derivation of (\ref{eqn:1.2}) is as follows. Denote by the multiplicative relation set $\mathbb{M}_{2}$ of the integers of 2-power, i.e.,

\begin{equation}\label{eqn:1.3}
\mathbb{M}_{2}=\left\{1,2,4,8, 16,32, \cdots , 2^{n},\cdots\right\}.
\end{equation}
Denote by the complementary index set $\mathcal{I}_{2}$ of $\mathbb{X}_{2}^{0}$ that contains all positive odd integers:

\begin{equation}\label{eqn:1.4}
\mathcal{I}_{2}=\left\{n\in\mathbb{N} \hspace{0.2cm} |\hspace{0.2cm} 2\nmid n\right\}=\left\{2k+1\right\}_{k=0}^{\infty}.
\end{equation}
The set of all natural number $\mathbb{N}$ can now be rearranged into

\begin{equation}\label{eqn:1.5}
\mathbb{N}=\underset{i\in\mathcal{I}_{2}}{\bigcup}i\mathbb{M}_{2},
\end{equation}
where $i\mathbb{M}_{2}=\left\{i,2i,2^{2}i,\cdots,2^{n}i,\cdots  \right\}$.
Clearly,
\begin{equation}\label{eqn:1.5-1}
i\mathbb{M}_{2}\cap j\mathbb{M}_{2}=\emptyset
\end{equation}
if $i,j\in\mathcal{I}_{2}$ and $i\neq j$.
More precisely, the right-hand side of (\ref{eqn:1.5}) can be expressed as an $\mathbb{N}\times \mathbb{N}$ table, Table 1.1.

\begin{equation*}
\begin{array}{c}
\psfrag{y}{$\mathbb{M}_{2}$}
\psfrag{z}{$\mathcal{I}_{2}$}
\includegraphics[scale=0.9]{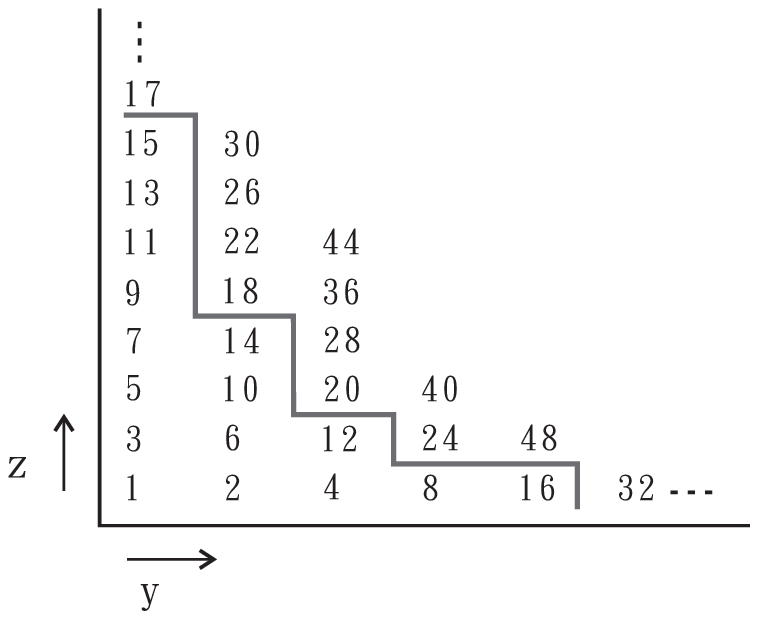}
\end{array}
\end{equation*}
\begin{equation*}
\text{Table 1.1.}
\end{equation*}

On integer lattice $\mathbb{Z}^{1}$, for $k\geq 1$, a $k$-lattice $\mathbb{Z}_{k}$ can be represented by $k$-cells as Fig. 1.1 (a) for drawing numbers or $k$-vertices as Fig. 1.1 (b) for drawing graph latter.


\begin{equation*}
\begin{array}{cccc}
\psfrag{a}{$\cdots$}
\psfrag{c}{$Z_{k}$}
\includegraphics[scale=1.0]{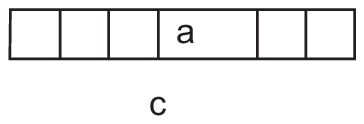}
& &&
\psfrag{c}{$Z_{k}$}
\includegraphics[scale=0.7]{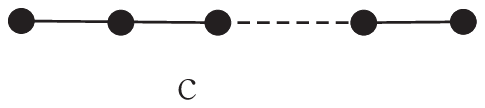}
\\
\text{Figure 1.1 (a).} & & &\text{Figure 1.1 (b).}
\end{array}
\end{equation*}
Let $M_{k}$ and $iM_{k}$ be the numbered lattices of the first $k$ elements in $\mathbb{M}_{2}$ and in $i\mathbb{M}_{2}$ on the $Z_{k}$, respectively.

\begin{equation*}
\begin{array}{c}
\psfrag{a}{$\cdots$}
\psfrag{b}{1}
\psfrag{c}{2}
\psfrag{d}{4}
\psfrag{e}{\footnotesize{$2^{k-1}$}}
\psfrag{f}{$i$}
\psfrag{g}{$2i$}
\psfrag{h}{$4i$}
\psfrag{j}{\tiny{$2^{k-1}i$}}
\psfrag{k}{$M_{k}$}
\psfrag{l}{$iM_{k}$}
\includegraphics[scale=1.6]{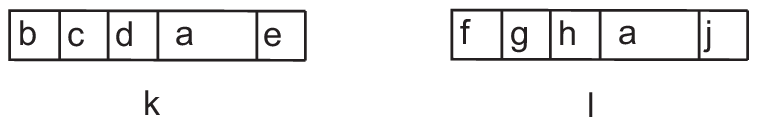}
\end{array}
\end{equation*}
\begin{equation*}
\text{Figure 1.2.}
\end{equation*}

Let

\begin{equation}\label{eqn:1.6}
\mathcal{N}(m)=\left\{k\in\mathbb{N} \hspace{0.2cm}| \hspace{0.2cm} 1\leq k\leq m\right\},
\end{equation}
be the set of natural numbers that are less than or equal to $m$. For each $n\geq 1$ and $1\leq i\leq 2^{n}$, let

\begin{equation}\label{eqn:1.6-1}
k_{n}(i)=\max\left\{k \hspace{0.2cm}| \hspace{0.2cm} i2^{k}\leq 2^{n}\right\}=\log_{2}\left\lfloor\frac{2^{n}}{i}\right\rfloor,
\end{equation}
where $\lfloor x \rfloor$ is the largest integer that is less than or equal to $x$.

Then, from Table 1.1, it is clear that

\begin{equation}\label{eqn:1.7}
\mathcal{N}(2^{n})=\underset{i\in\mathcal{I}_{2},1\leq i\leq 2^{n}}{\bigcup}iM_{k_{n}(i)}.
\end{equation}
For example, for $n=4$, from Table 1.1,

\begin{equation}\label{eqn:1.8}
\mathcal{N}(2^{4})=M_{5}\cup (3M_{3})\cup (5M_{2})\cup (7M_{2})\cup (9M_{1})\cup (11M_{1})\cup (13M_{1})\cup (15M_{1}).
\end{equation}

In terms of blank lattices, the numbers in $\mathcal{N}(2^{4})$ lie on

\begin{equation}\label{eqn:1.9}
 \text{one copy of }Z_{5},  \text{one copy of }Z_{3}, \text{two copies of }Z_{2} \text{ and } 2^{2} \text{ copies of }Z_{1}.
\end{equation}
The result of general $\mathcal{N}(2^{n})$ follows from the following proposition, which can be easily proven by mathematical induction.

\begin{proposition}
\label{proposition:1.1}
 For integers $Q\geq 2$ and $n\geq 1$,

\begin{equation}\label{eqn:1.10}
Q^{n}=(n+1)+n(Q-2)+(Q-1)^{2}\underset{k=1}{\overset{n-1}{\sum}}kQ^{n-1-k}.
\end{equation}
In particular,

\begin{equation}\label{eqn:1.11}
2^{n}=(n+1)+\underset{k=1}{\overset{n-1}{\sum}}k2^{n-1-k}.
\end{equation}

\end{proposition}

Therefore, (\ref{eqn:1.11}) states that the numbers in $\mathcal{N}(2^{n})$ are spread out on blank lattices with one copy of $Z_{n+1}$ and $2^{n-1-k}$ copies of $Z_{k}$, $1\leq k\leq n-1$. In particular, setting $n=4$ in (\ref{eqn:1.11}) yields (\ref{eqn:1.9}).

Now, consider again system $\mathbb{X}_{2}^{0}$ and the target formula (\ref{eqn:1.2}). For any $n\geq 1$, let $X_{n}$ be the set of all admissible $n$-sequences in $\mathbb{X}_{2}^{0}$:

\begin{equation}\label{eqn:1.12}
X_{n}=\left\{
(x_{1},x_{2},\cdots, x_{n})\in \{0,1\}^{\mathbb{Z}_{n}} \hspace{0.2cm}| \hspace{0.2cm} x_{k}x_{2k}=0 \text{ for all }k\geq 1 \text{ and } 2k\leq n
\right\}.
\end{equation}
Our purpose is to compute $|X_{n}|$, which is the number of elements in (\ref{eqn:1.12}).
The entropy $h(\mathbb{X}_{2}^{0})$ follows from

\begin{equation}\label{eqn:1.12-0}
h(\mathbb{X}_{2}^{0})=\underset{n\rightarrow \infty}{\lim}\hspace{0.1cm}\frac{1}{n}\log|X_{n}|.
\end{equation}

The constraint

\begin{equation}\label{eqn:1.12-1}
x_{k}x_{2k}=0
\end{equation}
in (\ref{eqn:1.12}) is the admissible condition of golden-mean shift on $\mathbb{Z}^{1}$, and it states that symbol 1 is not allowed to follow symbol 1 immediately. Then, the forbidden set on $Z_{2}$ is
$\left\{
\includegraphics[scale=0.5]{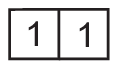}
\right\}$.
The transition matrix is

\begin{equation}\label{eqn:1.13}
G=
\left[
\begin{array}{cc}
1 & 1 \\
1 & 0
\end{array}
\right].
\end{equation}
Le $\Sigma_{k}$ be the set of all admissible patterns on $\mathbb{Z}_{k}$ with respect to (\ref{eqn:1.13}); then

\begin{equation}\label{eqn:1.14}
|\Sigma_{k}|=a_{k},
\end{equation}
which is the $k$-th Fibonacci number. Since the constraint (\ref{eqn:1.12-1}) applies to each $i\mathbb{M}_{2}$ independently for $i\in\mathcal{I}_{2}$,

\begin{equation}\label{eqn:1.15}
|X_{2^{n}}|=|\Sigma_{n+1}|\underset{k=1}{\overset{n-1}{\prod}}|\Sigma_{k}|^{2^{n-1-k}},
\end{equation}
which implies

\begin{equation}\label{eqn:1.16}
\frac{1}{2^{n}}\log|X_{2^{n}}|=\frac{1}{2^{n}}a_{n+1}+\underset{k=1}{\overset{n-1}{\sum}}\frac{1}{2^{k+1}}\log a_{k}.
\end{equation}
Hence (\ref{eqn:1.2}) follows easily from (\ref{eqn:1.16}).

By the similar argument, (\ref{eqn:1.10}) of Proposition \ref{proposition:1.1} also recovers the following results \cite{17}.

\begin{theorem}
\label{theorem:1.2}
 For any $Q\geq 2$, denote the multiplicative integers system

\begin{equation}\label{eqn:1.17}
\mathbb{X}_{Q}^{0}=\left\{
(x_{1},x_{2},\cdots )\in\{0,1\}^{\mathbb{N}} \mid x_{k}x_{Qk}=0 \text{ for all } k\geq 1
 \right\} ,
\end{equation}
then

\begin{equation}\label{eqn:1.18}
h(\mathbb{X}_{Q}^{0})=(Q-1)^{2}\underset{k=1}{\overset{\infty}{\sum}}\frac{1}{Q^{k+1}}\log a_{k}.
\end{equation}

 \end{theorem}

Consideration of the above reveals the following three main parts of our study of $\mathbb{X}_{2}^{0}$.

\begin{enumerate}
\item[(I)] Identify the numbered lattice $M_{k}$ and the admissible blank lattice $Z_{k}$ from the given system; see Figs. 1.1 and 1.2.

\item[(II)] Compute the numbers of copies of independent admissible lattices of the same length; see formulae (\ref{eqn:1.10}) and (\ref{eqn:1.11}).

\item[(III)] Determine the set of all admissible patterns $\Sigma_{k}$, which can be generated on $Z_{k}$, and compute the number of $|\Sigma_{k}|$.

\end{enumerate}

Notably, step (III) in the  study of $\mathbb{X}_{2}^{0}$ is the classical one-dimensional pattern generation problem; see \cite{38}.

Based on the above observations, the rest of this paper will consider the following two classes of systems.

\begin{enumerate}
\item[(i)] Multi-dimensional decoupled systems like

\begin{equation}\label{eqn:1.19}
\mathbb{X}_{2,3}^{0}=\left\{
(x_{1},x_{2},x_{3},\cdots )\in\{0,1\}^{\mathbb{N}} \mid x_{k}x_{2k}x_{3k}=0 \text{ for all } k\geq 1
 \right\}.
\end{equation}

\item[(ii)] One-dimensional coupled systems like

\begin{equation}\label{eqn:1.20}
\mathbb{X}_{2}^{A}=\left\{
(x_{1},x_{2},\cdots )\in\Sigma_{A} \mid x_{k}x_{2k}=0 \text{ for all } k\geq 1
 \right\},
\end{equation}
i.e.,
\begin{equation}\label{eqn:1.21}
\mathbb{X}_{2}^{A}=\mathbb{X}_{2}^{0}\cap \Sigma_{A},
\end{equation}
where $\Sigma_{A}$ is a proper additive shift of finite type.
\end{enumerate}

First, consider multi-dimensional decoupled systems. Let

\begin{equation}\label{eqn:1.20-1}
1<\gamma_{1}<\gamma_{2}<\cdots <\gamma_{d}
\end{equation}
be natural numbers, $d\geq 2$, such that $\gamma_{i}$ and $\gamma_{j}$ are relatively prime for all $i<j$, i.e.,

\begin{equation}\label{eqn:1.20-2}
(\gamma_{i},\gamma_{j})=1
\end{equation}
for all $1\leq i < j \leq d$, where $(a,b)$ is the greatest common divisor of natural numbers $a$ and $b$.

Denote by

\begin{equation}\label{eqn:1.20-3}
\Gamma \equiv \Gamma_{d}=\{\gamma_{1},\gamma_{2},\cdots , \gamma_{d}\}
\end{equation}
and

\begin{equation}\label{eqn:1.20-4}
\begin{array}{rl}
\mathbb{X}_{\Gamma}^{0} \equiv  & \mathbb{X}_{\gamma_{1},\gamma_{2},\cdots, \gamma_{d}}^{0} \\
&\\
 =& \left\{
(x_{1},x_{2},x_{3},\cdots )\in\{0,1\}^{\mathbb{N}} \mid x_{k}x_{\gamma_{1}k}x_{\gamma_{2}k}\cdots x_{\gamma_{d}k}=0 \text{ for all } k\geq 1
 \right\}.
 \end{array}
\end{equation}
Let
\begin{equation}\label{eqn:2.37}
\begin{array}{rl}
\mathbb{M}_{\Gamma}\equiv & \left\{ \gamma_{1}^{m_{1}}\gamma_{2}^{m_{2}}\cdots\gamma_{d}^{m_{d}} \hspace{0.2cm}|\hspace{0.2cm} m_{j}\geq 0    \right\} \\
& \\
\equiv & \left\{ q_{k}\right\}_{k=1}^{\infty}
\end{array}
\end{equation}
with $q_{k}<q_{j}$ if $k< j$.

Then, (\ref{eqn:2.37}) defines a sequence of $d$-dimensional numbered lattices $M_{k}$ of $k$ cells. The blank lattices $L_{k}$ are defined analogously; see (\ref{eqn:2.2}) and (\ref{eqn:2.39-10}). The complementary index set $\mathcal{I}_{\Gamma}$ of $\mathbb{M}_{\Gamma}$ is defined by

\begin{equation}\label{eqn:2.38}
\mathcal{I}_{\Gamma}=\left\{ n\in \mathbb{N}\hspace{0.2cm}|\hspace{0.2cm} \gamma_{j}\nmid n, 1\leq j\leq d    \right\}.
\end{equation}
Hence,

\begin{equation}\label{eqn:2.39}
\mathbb{N}= \underset{i\in \mathcal{I}_{\Gamma}}{\bigcup} i \mathbb{M}_{\Gamma}.
\end{equation}

The following theorem for multi-dimensional decoupled system is proven in Theorem \ref{theorem:2.10}.
\begin{theorem}
\label{theorem:1.10}
Let $\Gamma=\left\{ \gamma_{1},\gamma_{2},\cdots, \gamma_{d} \right\}$ satisfy (\ref{eqn:1.20-1}) and (\ref{eqn:1.20-2}). Then the entropy of $\mathbb{X}_{\Gamma}^{0}$ is given by

\begin{equation}\label{eqn:1.20-5}
h(\mathbb{X}_{\Gamma}^{0})=\underset{k=1}{\overset{\infty}{\sum}}\beta_{\Gamma} \left( \frac{1}{q_{k}}-\frac{1}{q_{k+1}}\right)\log|\Sigma_{k}|,
\end{equation}
where

\begin{equation}\label{eqn:1.20-6}
\beta_{\Gamma}=\frac{\sharp\left(\mathcal{I}_{\Gamma}\cap \left[1,\gamma_{1}\gamma_{2}\cdots \gamma_{d} \right]\right)}{\gamma_{1}\gamma_{2}\cdots \gamma_{d}}.
\end{equation}

\end{theorem}

Finally, consider one-dimensional coupled system:
\begin{equation}\label{eqn:1.20-7}
\mathbb{X}_{Q}^{A}\equiv \mathbb{X}_{Q}^{0} \cap \Sigma_{A},
\end{equation}
where $\Sigma_{A}$ is a shift of finite type generated by transition matrix $A$.

Denote by $L_{Q;k}$  the degree $k$ blank lattice of the admissible numbered lattice $M_{k}(l)$; see Fig. 3.2 for $Q=2$ and Figs. 3.7 and 3.8 for $Q=3$.
The following theorem are proven by Theorems 3.7 and \ref{theorem:3.7}.

\begin{theorem}
\label{theorem:1.11}
For any $Q\geq 2$ and $k\geq 2$,

\begin{equation}\label{eqn:1.20-8}
\frac{Q-1}{Q\left(Q^{k}-1\right)}\log|\Sigma_{Q;k}|\leq h(\mathbb{X}_{Q}^{A})\leq\frac{Q-1}{Q\left(Q^{k}-1\right)}\left(\log|\Sigma_{Q;k}|+k\log 2\right),
\end{equation}
and

\begin{equation}\label{eqn:1.20-9}
h(\mathbb{X}_{Q}^{A}) = \underset{k\rightarrow \infty}{\lim}\frac{Q-1}{Q\left(Q^{k}-1\right)}\log|\Sigma_{Q;k}|,
\end{equation}
where $\Sigma_{Q;k}$ is the set of all admissible patterns on $L_{Q;k}$.
\end{theorem}

After we completed our study of (i) and (ii), we became aware of the work of Peres \emph{et al.} \cite{43} on (\ref{eqn:1.19}). These authors obtained the same results as ours for multi-dimensional system (i). Our methods for studying (i) differ from theirs by using the results from an investigation of pattern generation problems, and involving the three specified steps (I), (II) and (III). Moreover, a modification of these procedures enables us to study the one-dimensional coupled system (ii).

For the multi-dimensional coupled system like

\begin{equation}\label{eqn:1.21-1}
\begin{array}{rl}
\mathbb{X}_{2,3}^{A}= & \left\{
(x_{1},x_{2},\cdots )\in\Sigma_{A} \mid x_{k}x_{2k}x_{3k}=0 \text{ for all } k\geq 1
 \right\} \\
 & \\
 = &  \mathbb{X}_{2,3}^{0}\cap \Sigma_{A},
 \end{array}
\end{equation}
which is much more delicate. The dimension of $\Sigma_{A}$ is one and $\mathbb{X}_{2,3}^{0}$ is two. Difference of the dimensions induces an intrinsic difficulty
to decouple the system effectively. The problem of (\ref{eqn:1.21-1}) is still not solved by using our method which works well in studying one-dimensional coupled system (\ref{eqn:1.21}). See Section 4 for further discussion.

The rest of this paper is arranged as follows. Section 2 studies multi-dimensional systems. Section 3 studies one-dimensional coupled systems.
Section 4 studies multi-dimensional coupled systems.

\section{Multi-dimensional systems}

\hspace{0.4cm}This section concerns multi-dimensional systems. For simplicity, $\mathbb{X}_{2,3}^{0}$ is considered first. Recall that

\begin{equation}\label{eqn:2.1}
\mathbb{X}_{2,3}^{0}=\left\{
(x_{1},x_{2},x_{3},\cdots )\in\{0,1\}^{\mathbb{N}} \mid x_{k}x_{2k}x_{3k}=0 \text{ for all } k\geq 1
 \right\}.
\end{equation}

Firstly, the admissible numbered and blank lattices, determined by the constraint

\begin{equation}\label{eqn:2.2}
x_{k}x_{2k}x_{3k}=0,
\end{equation}
must be identified in $\mathbb{X}_{2,3}^{0}$.

Let $\mathbb{M}_{2,3}$ be the set of the numbers that are multiples of 2-power and 3-power:

 \begin{equation}\label{eqn:2.3}
\mathbb{M}_{2,3}=\left\{2^{k}3^{l} \hspace{0.2cm}|\hspace{0.2cm} k,l\geq 0\right\}.
\end{equation}
$\mathbb{M}_{2,3}$ can be expressed as follows.

\begin{equation*}
\begin{array}{cccc}
\includegraphics[scale=0.8]{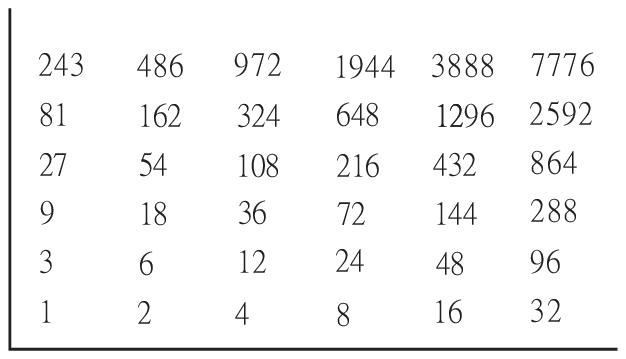} && &
\psfrag{1}{$q_{1}$}
\psfrag{2}{$q_{2}$}
\psfrag{3}{$q_{4}$}
\psfrag{4}{$q_{6}$}
\psfrag{5}{$q_{9}$}
\psfrag{7}{$q_{3}$}
\psfrag{8}{$q_{5}$}
\psfrag{9}{$q_{8}$}
\psfrag{a}{$q_{11}$}
\psfrag{b}{$q_{15}$}
\psfrag{d}{$q_{7}$}
\psfrag{e}{$q_{10}$}
\psfrag{f}{$q_{14}$}
\psfrag{g}{$q_{18}$}
\psfrag{h}{$q_{23}$}
\psfrag{k}{$q_{12}$}
\psfrag{l}{$q_{16}$}
\psfrag{m}{$q_{21}$}
\psfrag{n}{$q_{26}$}
\psfrag{o}{$q_{32}$}
\psfrag{q}{$q_{19}$}
\psfrag{r}{$q_{24}$}
\psfrag{s}{$q_{30}$}
\psfrag{t}{$q_{36}$}
\psfrag{u}{$q_{43}$}
\psfrag{v}{$q_{27}$}
\psfrag{w}{$q_{33}$}
\psfrag{x}{$q_{40}$}
\psfrag{y}{$q_{47}$}
\psfrag{z}{$q_{55}$}
\psfrag{A}{$q_{13}$}
\psfrag{B}{$q_{20}$}
\psfrag{C}{$q_{29}$}
\psfrag{D}{$q_{39}$}
\psfrag{E}{$q_{51}$}
\psfrag{F}{$q_{64}$}
\includegraphics[scale=0.8]{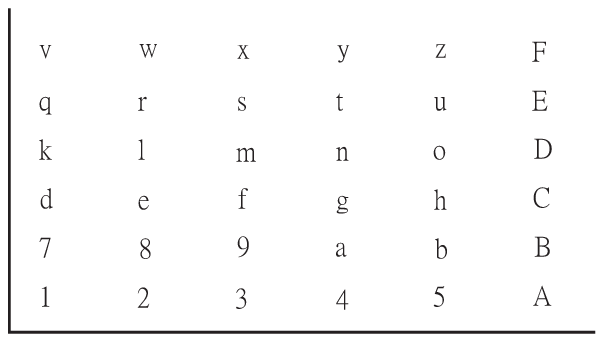}
\end{array}
\end{equation*}
\begin{equation*}
\mathbb{M}_{2,3}
\end{equation*}
\begin{equation*}
\text{Table 2.1.}
\end{equation*}

Since $\mathbb{M}_{2,3}$ inherits the natural ordering of natural numbers, the $k$-th number in $\mathbb{M}_{2,3}$ can be denoted by $q_{k}$, $k\in\mathbb{N}$. It seems that no known formula relates $2^{m}3^{n}$ to $q_{i}$ explicitly \cite{50}.

From Table 2.1, for any
$k\geq 1$, the L-shaped $k$-cell numbered lattice $M_{k}$ that contains $\{q_{1},q_{2},\cdots, q_{k}\}$ can be defined, and $k$-cell blank lattice $L_{k}$ can also be defined by deleting the numbers of $M_{k}$.
Indeed, for $k\geq 1$,

\begin{equation}\label{eqn:2.2}
L_{k}=\left\{(i,j)\in \mathbb{Z}^{2} \hspace{0.1cm} \mid \hspace{0.1cm} 2^{i}3^{j}\leq q_{k} \text{ for } i,j\geq 0\right\}.
\end{equation}
Fig. 2.1 presents the first eight numbered lattices $M_{k}$. Notably, the number $q_{k}$ in the bold square is the largest number in $M_{k}$.

\begin{equation*}
\begin{array}{llll}
\psfrag{a}{$M_{1}$ }
\includegraphics[scale=0.8]{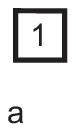} & \hspace{0.3cm}\psfrag{a}{$M_{2}$ }
\includegraphics[scale=0.8]{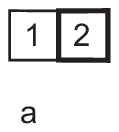}  & \hspace{0.3cm}\psfrag{a}{$M_{3}$ }
\includegraphics[scale=0.8]{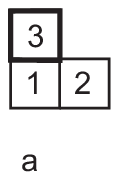}  &\hspace{0.3cm} \psfrag{a}{$M_{4}$ }
\includegraphics[scale=0.8]{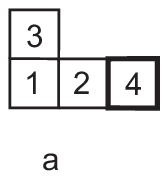}\\
& &　& \\
\psfrag{a}{$M_{5}$ }
\includegraphics[scale=0.8]{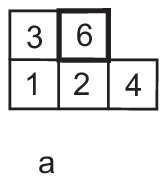} & \hspace{0.3cm}\psfrag{a}{$M_{6}$ }
\includegraphics[scale=0.8]{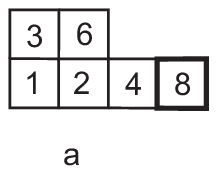}  &\hspace{0.3cm} \psfrag{a}{$M_{7}$ }
\includegraphics[scale=0.8]{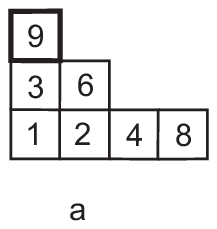}  &\hspace{0.3cm} \psfrag{a}{$M_{8}$ }
\includegraphics[scale=0.8]{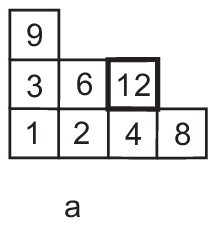}\\
\end{array}
\end{equation*}
\begin{equation*}
\text{Figure 2.1.}
\end{equation*}
In contrast to $\mathbb{X}_{2}^{0}$, all lattices $M_{k}$ and $L_{k}$ are now two-dimensional. Therefore, the system $\mathbb{X}_{2,3}^{0}$ can be regarded as a two-dimensional system; see \cite{3,4,5,6,37}.

Now, the complementary index set $\mathcal{I}_{2,3}$ of $\mathbb{X}_{2,3}^{0}$ is the set of all natural numbers that cannot be divided by two and three:

 \begin{equation}\label{eqn:2.4}
\mathcal{I}_{2,3} =  \left\{ n\in\mathbb{N} \hspace{0.2cm}|\hspace{0.2cm} 2\nmid n \text{ and } 3\nmid n \right\}
= \left\{6k+1,6k+5\right\}_{k=0}^{\infty}.
\end{equation}

The set $\mathbb{N}$ of natural numbers can be rearranged into the first octant of three-dimensional space as

 \begin{equation}\label{eqn:2.5}
\mathbb{N}=\underset{i\in\mathcal{I}_{2,3}}{\bigcup}i\mathbb{M}_{2,3},
\end{equation}
and for $i,j\in\mathcal{I}_{2,3}$ with $i\neq j$,

 \begin{equation}\label{eqn:2.6}
i\mathbb{M}_{2,3}\cap j\mathbb{M}_{2,3}=\emptyset;
\end{equation}
see Table 2.2.

\begin{equation*}
\begin{array}{c}
\psfrag{a}{$2^{m}$ }
\psfrag{b}{$\mathcal{I}_{2,3}$ }
\psfrag{c}{$3^{n}$ }
\includegraphics[scale=0.7]{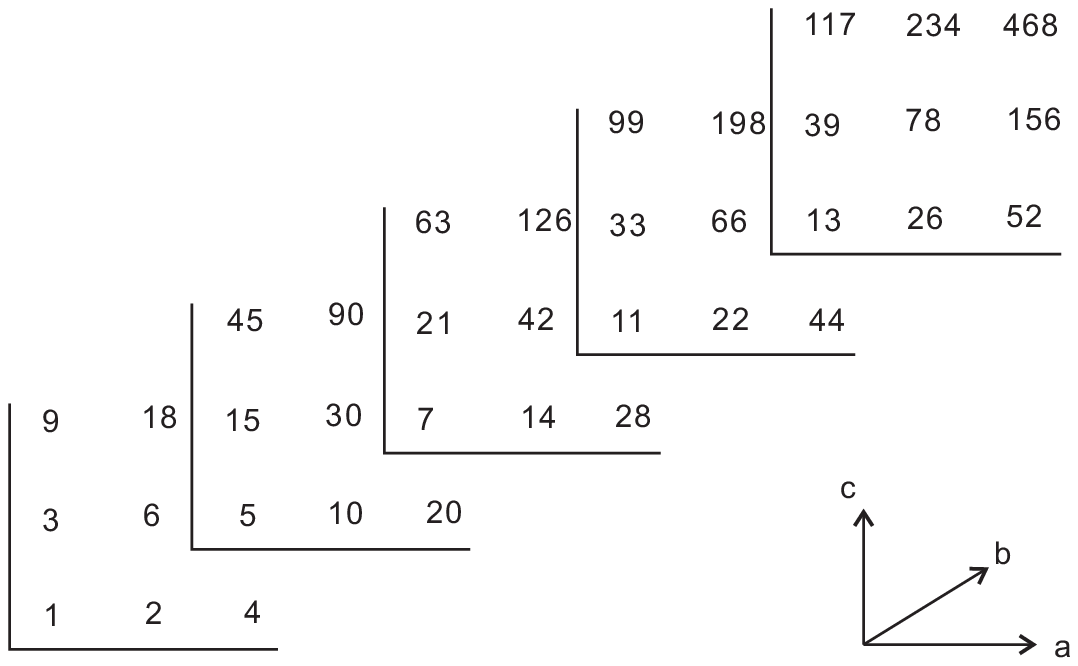}
\end{array}
\end{equation*}
\begin{equation*}
\text{Table 2.2.}
\end{equation*}

Therefore, for any $n\in\mathbb{N}$, there exists a unique $i=i(n)$ such that $n\in i\mathbb{M}_{2,3}$.

Next, proceed to step (II): compute the numbers of copies of $M_{k}$ for a given $\mathcal{N}(m)$, defined in (\ref{eqn:1.6}).

Let
\begin{equation}\label{eqn:2.7}
q_{K}=2^{m}3^{n}\in\mathbb{M}_{2,3}
\end{equation}
for some $K\geq 1$. Denote by

\begin{equation}\label{eqn:2.8}
J_{K}=\left\{ i\in\mathcal{I}_{2,3} \hspace{0.2cm}|\hspace{0.2cm} 1\leq i\leq q_{K}\right\}.
\end{equation}
Define

\begin{equation}\label{eqn:2.9}
q_{K}(i)=\max\left\{q\in\mathbb{M}_{2,3} \hspace{0.2cm}|\hspace{0.2cm} iq\leq q_{K}\right\},
\end{equation}
then $q_{K}(i)\geq 1$ for any $i\in J_{K}$.

Denote by the set

\begin{equation}\label{eqn:2.11}
I_{k}(K)=\left\{  i\in J_{K}  \hspace{0.2cm}|\hspace{0.2cm} q_{K}(i)=q_{k} \right\}
\end{equation}
for $1\leq k\leq K$. Clearly,
if $q_{K}(i)=q_{k}$, then
\begin{equation*}
iq_{k}\leq q_{K} < iq_{k+1}.
\end{equation*}
Therefore, a parallel result of (\ref{eqn:1.7}) and (\ref{eqn:1.10}) is as follows.

\begin{equation}\label{eqn:2.10}
\mathcal{N}(q_{K})=\underset{i\in J_{K}}{\bigcup}iM_{q_{K}(i)}=\underset{i\in I_{K}(k)}{\bigcup} i M_{k}
\end{equation}
with
\begin{equation}\label{eqn:2.12}
I_{k}(K)=\left( \frac{q_{K}}{q_{k+1}}, \frac{q_{K}}{q_{k}}\right]\cap \mathcal{I}_{2,3}.
\end{equation}

The following example illustrates (\ref{eqn:2.10}) and (\ref{eqn:2.12}).

\begin{example}
\label{example:2.4}
For $q_{14}=36$, it is easy to verify

\begin{equation*}
\left\{
\begin{array}{l}
I_{1}(14)=\{19,23,25,29,31,35\} \\ \\
I_{2}(14)=\{13,17\} \\ \\
I_{3}(14)=\{11\}, \hspace{0.2cm} I_{4}(14)=\{7\}, \hspace{0.2cm}I_{5}(14)=\{5\}, \hspace{0.2cm}I_{14}(14)=\{1\}.
\end{array}
\right.
\end{equation*}
The others are empty. Therefore,
\begin{equation*}
\mathcal{N}(36)=M_{14}\cup 5M_{5}\cup 7M_{4}\cup 11M_{3}\cup 13M_{2}\cup 17M_{2} \underset{i\in\{19,23,25,29,31,35\}}{\cup} iM_{1}.
\end{equation*}
\end{example}

For $1\leq k\leq K$, denote by

\begin{equation}\label{eqn:2.14}
\alpha_{k}(K)=|I_{k}(K)|,
\end{equation}
the number of copies of $M_{k}$ in $\mathcal{N}(q_{K})$.

Formula (\ref{eqn:2.12}) gives the density of copies of $M_{k}$ in $[1,q_{K}]$, which is crucial in computing the entropy. Indeed, following proposition is offered.

\begin{proposition}
\label{proposition:2.5}
On $\mathbb{X}_{2,3}^{0}$, for any $k\geq 1$,

\begin{equation}\label{eqn:2.15}
\underset{K\rightarrow\infty}{\lim}\frac{\alpha_{k}(K)}{q_{K}}=\beta_{2,3}\left(\frac{1}{q_{k}}-\frac{1}{q_{k+1}}\right),
\end{equation}
where
\begin{equation}\label{eqn:2.16}
\beta_{2,3}=\frac{\sharp\left\{\mathcal{I}_{2,3}\cap \left[1,2\cdot 3\right]\right\}}{2\cdot3}=\frac{1}{3}.
\end{equation}

\end{proposition}

\begin{proof}
For any fixed $k\geq 1$, from (\ref{eqn:2.12}) and (\ref{eqn:2.14}), (\ref{eqn:2.15}) follows. Clearly, (\ref{eqn:2.16}) follows from (\ref{eqn:2.4}).
\end{proof}

Furthermore, for $n\geq 1$, denote by
\begin{equation}\label{eqn:2.17}
\left\{
\begin{array}{l}
J_{n}=\left\{i\in \mathcal{I}_{2,3}\hspace{0.2cm}|\hspace{0.2cm} i\leq n\right\} , \\ \\
q(n;i)=\max\left\{q\in\mathbb{M}_{2,3} \hspace{0.2cm}|\hspace{0.2cm} iq\leq n\right\}, \\ \\
I(k;n)=\left\{i\in J_{n}  \hspace{0.2cm}|\hspace{0.2cm} q(n;i)=q_{k}\right\} \text{ and } \\ \\
\alpha(k;n)=|I(k;n)|.
\end{array}
\right.
\end{equation}
Now, the following proposition can be verified.
\begin{proposition}
\label{proposition:2.6}
For $k\geq 1$,

\begin{equation}\label{eqn:2.18}
\underset{n\rightarrow\infty}{\lim}\frac{\alpha(k;n)}{n}=\beta_{2,3}\left(\frac{1}{q_{k}}-\frac{1}{q_{k+1}}\right),
\end{equation}
where $\beta_{2,3}$ is given by (\ref{eqn:2.16}).

\end{proposition}
\begin{proof}
It is easy to see that if $q(n;i)=q_{k}$, then
\begin{equation*}
iq_{k}\leq n < iq_{k+1}.
\end{equation*}
As (\ref{eqn:2.12}), we have
\begin{equation}\label{eqn:2.19-0}
I(k;n)=\left( \frac{n}{q_{k+1}}, \frac{n}{q_{k}}\right]\cap \mathcal{I}_{2,3}.
\end{equation}
Therefore, (\ref{eqn:2.18}) follows immediately.

\end{proof}

After the density of $M_{k}$ (\ref{eqn:2.18}) is obtained, step (II) is completed. Now, the final step (III) is to compute the admissible patterns on $L_{k}$ for all $k\geq 1$.

Previously, two-dimensional pattern generation problems on L-shaped lattices has been studied by Lin and Yang in \cite{37}. The basic lattice $\mathbb{L}_{2,3}$ is defined by

\begin{equation}\label{eqn:2.19-2}
\mathbb{L}_{2,3}=\left\{ (i,j)\in\mathbb{Z}^{2} \hspace{0.1cm} \mid\hspace{0.1cm} 0\leq i+j\leq 1 \text{ for } i,j\geq 0 \right\}=\{(0,0), (1,0),(0,1)\},
\end{equation}
i.e.,
\begin{equation*}
\psfrag{a}{$\mathbb{L}_{2,3}=$}
\psfrag{b}{or}
\psfrag{c}{,}
\includegraphics[scale=0.5]{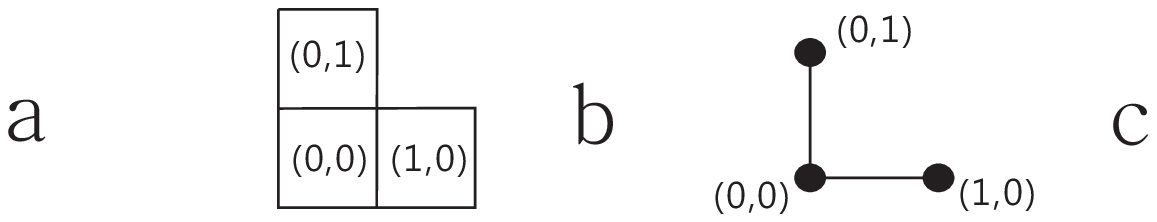}
\end{equation*}
the L-shaped lattice with origin $(0,0)$ as the corner vertex.

The constraint (2.2) implies that the forbidden local pattern on $\mathbb{L}_{2,3}$ is

\begin{equation}\label{eqn:2.19}
\mathcal{F}_{2,3}=\left\{
\begin{array}{c}
\includegraphics[scale=0.7]{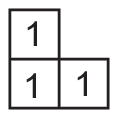}
\end{array}
 \right\}.
\end{equation}
Therefore, the basic set of admissible patterns is

\begin{equation}\label{eqn:2.20}
\mathcal{B}_{2,3}=\left\{
\begin{array}{ccccccc}
\includegraphics[scale=0.7]{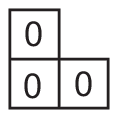}, & \includegraphics[scale=0.7]{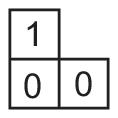}, &\includegraphics[scale=0.7]{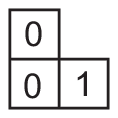}, &\includegraphics[scale=0.7]{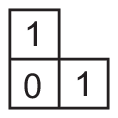}, &\includegraphics[scale=0.7]{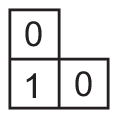}, &\includegraphics[scale=0.7]{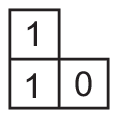}, &\includegraphics[scale=0.7]{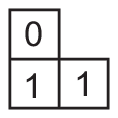}
\end{array}
 \right\};
\end{equation}
see \cite{5,6}.
For $k\geq1$, the set of all admissible patterns on $L_{k}$ that are determined by $\mathcal{B}_{2,3}$ is defined by
\begin{equation}\label{eqn:2.20-0}
\begin{array}{rl}
 & \Sigma_{k}= \Sigma_{k}(\mathcal{B}_{2,3}) \\
 & \\
 = & \left\{U\in \{0,1\}^{L_{k}} \hspace{0.1cm} : \hspace{0.1cm} U\mid_{L}\in\mathcal{B}_{2,3}\text{ for all }L=\mathbb{L}_{2,3}+\mathbf{v}\subseteq L_{k} \text{ with some }\mathbf{v}\in\mathbb{Z}^{2}  \right\}.
\end{array}
\end{equation}
Clearly, $\Sigma_{1}=\left\{
\begin{array}{cc}
\includegraphics[scale=0.7]{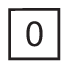}, & \includegraphics[scale=0.7]{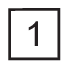}
\end{array}
 \right\}$ and $\Sigma_{2}=\left\{
\begin{array}{cccc}
\includegraphics[scale=0.7]{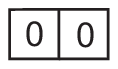}, & \includegraphics[scale=0.7]{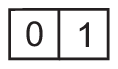},& \includegraphics[scale=0.7]{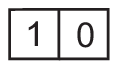},& \includegraphics[scale=0.7]{sigma2_4.eps}
\end{array}
 \right\}$, and no constraint applies on $L_{1}$ and $L_{2}$.

Recall that
\begin{equation*}
X_{2,3;n}=\left\{
(x_{1},x_{2},\cdots, x_{n})\in \{0,1\}^{\mathbb{Z}_{n}} \hspace{0.2cm}| \hspace{0.2cm} x_{k}x_{2k}x_{3k}=0 \text{ for all }k\geq 1 \text{ and } 3k\leq n
\right\}
\end{equation*}
and

\begin{equation*}
h(\mathbb{X}_{2,3}^{0})=\underset{n\rightarrow \infty}{\lim}\hspace{0.1cm}\frac{1}{n}\log|X_{2,3;n}|.
\end{equation*}

After the above procedures have been completed, the following result concerning the entropy $h(\mathbb{X}_{2,3}^{0})$ is obtained.

\begin{theorem}
\label{theorem:2.7}
The entropy of $\mathbb{X}_{2,3}^{0}$ is given by

\begin{equation}\label{eqn:2.23}
h(\mathbb{X}_{2,3}^{0})=\underset{k=1}{\overset{\infty}{\sum}}\beta_{2,3}\left(\frac{1}{q_{k}}-\frac{1}{q_{k+1}}\right)\log|\Sigma_{k}|,
\end{equation}
where $|\Sigma_{k}|$ is the number of all admissible patterns determined by $\mathcal{B}_{2,3}$ on $L_{k}$.
\end{theorem}

\begin{proof}
For any $n\geq 1$, let $X_{2,3;n}$ be the set of all admissible $n$-sequences in $\mathbb{X}_{2,3}^{0}$. From the condition (\ref{eqn:2.2}), it is easy to see that for any two $i_{1},i_{2}\in \mathcal{I}_{2,3}$, the admissible patterns on $i_{1}\mathbb{M}_{2,3}$ and the admissible patterns on $i_{2}\mathbb{M}_{2,3}$ are mutually independent. Then, we have that for any $n\geq 1$,

\begin{equation*}
|X_{2,3;n}|=\underset{k\in J(n)}{\prod}|\Sigma_{k}|^{\alpha(k;n)}.
\end{equation*}

Therefore, from Propositions 2.3,
\begin{equation*}
\begin{array}{rl}
h(\mathbb{X}_{2,3}^{0})= &  \underset{n\rightarrow \infty}{\lim} \frac{\log |X_{2,3;n}|}{n} \\
& \\
 = & \underset{n\rightarrow \infty}{\lim} \left(\underset{k\in J(n)}{\sum}\alpha(k;n) \log|\Sigma_{k}|\right)/ n \\
& \\
=& \underset{k=1}{\overset{\infty}{\sum}}\beta_{2,3}\left(\frac{1}{q_{k}}-\frac{1}{q_{k+1}}\right)\log|\Sigma_{k}|.
\end{array}
\end{equation*}
The proof is complete.
\end{proof}

\begin{remark}
\label{remark:2.7-1}

\item[(i)]
%
%

Denote by

\begin{equation}\label{eqn:2.22}
|\Sigma_{k}|=b_{k}
\end{equation}
the number of patterns in $\Sigma_{k}$. Since no exact formula relates $2^{m}3^{n}$ to $q_{i}$ for $\mathbb{M}_{2,3}$ in Table 2.1, unlike for a Fibonacci number, no recursive formula exists for $b_{k}$; see \cite{50}. This fact creates the most difficulty in computing entropy for a multi-dimensional system; see\cite{5,6,37}.

However, for relatively small $k$, $b_{k}$ can be computed using the transition matrices developed in \cite{5,6}. Table 2.3 presents few cases for $q_{k}=6^{l}$, $1\leq l\leq 4$.

\begin{equation*}
\begin{tabular}{|l|c|c|c|c|}
\hline
 $k$ & 5 & 14 & 26 & 43
\\\hline
$b_{k}$ & 14& 3722 & 5,434,757 &  172,749,984,030 \\
\hline
\end{tabular}
\end{equation*}

\begin{equation*}
\text{Table 2.3.}
\end{equation*}

\item[(ii)] Define the ratio of $|\Sigma_{k}|$ by

\begin{equation}\label{eqn:2.23-2}
r_{k}=|\Sigma_{k}|/ |\Sigma_{k-1}|= b_{k}/ b_{k-1}
\end{equation}
for $k\geq 2$. From Table 2.1 and Fig. 2.1, it is easy to verify that

\begin{equation}\label{eqn:2.23-3}
\begin{array}{ccc}
r_{k}=2 & \text{if} & q_{k}=2^{n},
\end{array}
\end{equation}
for some $n\geq 1$.

On the other hand, it can be shown that there exists $C\leq \frac{31}{16}$ such that

\begin{equation}\label{eqn:2.23-4}
\begin{array}{ccc}
r_{k}\leq C & \text{for} & q_{k}\neq 2^{n} \text{ for all }n\geq 2.
\end{array}
\end{equation}
Therefore, $\{r_{k}\}$ cannot have a limit as $k$ tends to $\infty$, unlike the Fibonacci sequence which has the limit $\frac{1+\sqrt{5}}{2}$. A further study of $\{r_{k}\}$ and $b_{k}$ is needed.

\end{remark}

In the following, an approximation of (\ref{eqn:2.23}) is given.
For $n \geq 1$, let
\begin{equation}\label{eqn:2.23-1}
h^{(n)}(\mathbb{X}_{2,3}^{0})=\underset{k=1}{\overset{n}{\sum}}\beta_{2,3}\left(\frac{1}{q_{k}}-\frac{1}{q_{k+1}}\right)\log|\Sigma_{k}|.
\end{equation}
Clearly, from Theorem 2.4, $h^{(n)}(\mathbb{X}_{2,3}^{0})$ is a lower bound of $h(\mathbb{X}_{2,3}^{0})$, and $h^{(n)}(\mathbb{X}_{2,3}^{0})$ increasingly approaches $h(\mathbb{X}_{2,3}^{0})$ as $n$ tends to infinity. Furthermore, let

\begin{equation}\label{eqn:2.23-2}
\begin{array}{rl}
E^{(n)}(\mathbb{X}_{2,3}^{0})\equiv &  \underset{k=n+1}{\overset{\infty}{\sum}}\beta_{2,3}\left(\frac{1}{q_{k}}-\frac{1}{q_{k+1}}\right)\log2^{k} \\
& \\
= & \beta_{2,3}  \left( \frac{n+1}{q_{n+1}} +3- \underset{k=1}{\overset{n+1}{\sum}}\frac{1}{q_{k}} \right)\log2,
\end{array}
\end{equation}
where $ \underset{k=1}{\overset{\infty}{\sum}}\frac{1}{q_{k}}=3$. Hence,

\begin{equation}\label{eqn:2.23-3}
h(\mathbb{X}_{2,3}^{0})-h^{(n)}(\mathbb{X}_{2,3}^{0}) \leq E^{(n)}(\mathbb{X}_{2,3}^{0}).
\end{equation}

Table 2.4 presents cases for $n$ with $q_{n}=6^{l}$ and $1\leq l\leq 4$.

\begin{equation*}
\begin{tabular}{|c|c|c|c|c|}
\hline
 $n$ &5 &14  &26 &43 \\
\hline
$h^{(n)}(\mathbb{X}_{2,3}^{0})$ & 0.319901 & 0.537229 & 0.620707 & 0.645733  \\
\hline
\end{tabular}
\end{equation*}

\begin{equation*}
\text{Table 2.4.}
\end{equation*}
Moreover, $h^{(153)}(\mathbb{X}_{2,3}^{0})\approx 0.654303$ and $E^{(153)}(\mathbb{X}_{2,3}^{0}) \approx 0.0000238741$.

Now, the theorem just established can be easily extended to general multiplicative systems. For simplicity, the result is stated and necessary modifications to its proof are only sketched.

First, consider $\mathbb{X}_{\Gamma}^{0}$ that satisfies (\ref{eqn:1.20-1}) and (\ref{eqn:1.20-2}).
$\mathbb{M}_{\Gamma}$ is defined in (\ref{eqn:2.37}). For $k\geq 1$, the blank $k$-cell lattice $L_{k}$ is defined by

\begin{equation}\label{eqn:2.39-10}
L_{k}=\left\{(i_{1},i_{2},\cdots, i_{d})\in \mathbb{Z}^{d} \hspace{0.1cm} \mid \hspace{0.1cm} \gamma_{1}^{i_{1}}\gamma_{2}^{i_{2}}\cdots \gamma_{d}^{i_{d}}\leq q_{k} \text{ for } i_{q}\geq 0, 1\leq q\leq d\right\}.
\end{equation}

After determining the $k$-cell lattice $L_{k}$ of $\mathbb{X}_{\Gamma}^{0}$, the basic lattice $\mathbb{L}_{\Gamma}$ of $\mathbb{X}_{\Gamma}^{0}$ is defined by
\begin{equation}\label{eqn:2.39-11}
\mathbb{L}_{\Gamma}=\left\{(i_{1},i_{2},\cdots, i_{d})\in \mathbb{Z}^{d} \hspace{0.1cm} \mid \hspace{0.1cm} 0\leq \underset{k=1}{\overset{d}{\sum}}\hspace{0.2cm}i_{k}\leq 1 \text{ for }i_{k}\geq 0, 1\leq k\leq d\right\},
\end{equation}
the $d$-dimensional L-shaped lattice with the origin $(0,0,\cdots,0)$ as the corner vertex; see (\ref{eqn:2.19-2}) for $\mathbb{X}_{2,3}^{0}$. From the constraint

\begin{equation}\label{eqn:2.40-99}
x_{k}x_{\gamma_{1}k}x_{\gamma_{2}k}\cdots x_{\gamma_{d}k}=0,
\end{equation}
we have the forbidden set
\begin{equation*}
\mathcal{F}_{\Gamma}=\left\{U=(u_{i_{1},i_{2},\cdots, i_{d}})\in \{0,1\}^{\mathbb{L}_{\Gamma}} \hspace{0.1cm} \mid \hspace{0.1cm} u_{i_{1},i_{2},\cdots, i_{d}}=1 \text{ for all }(i_{1},i_{2},\cdots, i_{d})\in\mathbb{L}_{\Gamma}   \right\}.
\end{equation*}
Then, the basic set of admissible patterns is defined by

\begin{equation}\label{eqn:2.40-10}
\mathcal{B}_{\Gamma}=\{0,1\}^{\mathbb{L}_{\Gamma}} \setminus \mathcal{F}_{\Gamma},
\end{equation}
which induces a $d$-dimensional shift of finite type $\Sigma(\mathcal{B}_{\Gamma})$ on the $d$-dimensional lattice space $\mathbb{Z}^{d}$.

Denote by $\Sigma_{k}(\mathcal{B}_{\Gamma})$ the set of all admissible patterns determined by $\mathcal{B}_{\Gamma}$ on $L_{k}$ and

\begin{equation*}
X_{\Gamma;n}=\left\{
(x_{1},x_{2},\cdots, x_{n})\in \{0,1\}^{\mathbb{Z}_{n}} \hspace{0.2cm}| \hspace{0.2cm} x_{k}x_{\gamma_{1}k}\cdots x_{\gamma_{d}k}=0 \text{ for all }k\geq 1 \text{ and } \gamma_{d}k\leq n
\right\}.
\end{equation*}

By a similar argument, the entropy of $\mathbb{X}_{\Gamma}^{0}$ can be obtained as follows.

\begin{theorem}
\label{theorem:2.10}
Let $\Gamma=\left\{ \gamma_{1},\gamma_{2},\cdots, \gamma_{d} \right\}$ satisfy (\ref{eqn:1.20-1}) and (\ref{eqn:1.20-2}). Then the entropy of $\mathbb{X}_{\Gamma}^{0}$ is given by

\begin{equation}\label{eqn:2.41}
h(\mathbb{X}_{\Gamma}^{0})=\underset{k=1}{\overset{\infty}{\sum}}\beta_{\Gamma} \left( \frac{1}{q_{k}}-\frac{1}{q_{k+1}}\right)\log|\Sigma_{k}|,
\end{equation}
where

\begin{equation}\label{eqn:2.41-0}
\beta_{\Gamma}=\frac{\sharp\left(\mathcal{I}_{\Gamma}\cap \left[1,\gamma_{1}\gamma_{2}\cdots \gamma_{d} \right]\right)}{\gamma_{1}\gamma_{2}\cdots \gamma_{d}}.
\end{equation}

\end{theorem}

\begin{proof}

By (\ref{eqn:2.37}) and (\ref{eqn:2.38}), let
\begin{equation}\label{eqn:2.41-1}
\left\{
\begin{array}{l}
J_{\Gamma;n}=\left\{i\in \mathcal{I}_{\Gamma}\hspace{0.2cm}|\hspace{0.2cm} i\leq n\right\} , \\ \\
q_{\Gamma}(n;i)=\max\left\{q\in\mathbb{M}_{\Gamma} \hspace{0.2cm}|\hspace{0.2cm} iq\leq n\right\}, \\ \\
I_{\Gamma}(k;n)=\left\{i\in J_{\Gamma;n}  \hspace{0.2cm}|\hspace{0.2cm} q_{\Gamma}(n;i)=q_{k}\right\} \text{ and } \\ \\
\alpha_{\Gamma}(k;n)=|I_{\Gamma}(k;n)|.
\end{array}
\right.
\end{equation}
As the proof of Proposition 2.3, it can be verified that
\begin{equation*}
I_{\Gamma}(k;n)=\left(\frac{n}{q_{k+1}},\frac{n}{q_{k}}\right]\bigcap \mathcal{I}_{\Gamma}
\end{equation*}
and

\begin{equation}\label{eqn:2.41-10}
\underset{n\rightarrow \infty}{\lim} \frac{\alpha_{\Gamma}(k;n)}{n}=\beta_{\Gamma} \left( \frac{1}{q_{k}}-\frac{1}{q_{k+1}}\right).
\end{equation}

From the constraint (\ref{eqn:2.40-99}), for any $i_{1}\neq i_{2}\in \mathcal{I}_{\Gamma}$, the admissible patterns on $i_{1}\mathbb{M}_{\Gamma}$ and the admissible patterns on $i_{2}\mathbb{M}_{\Gamma}$ are mutually independent. And, the admissible patterns on $i\mathbb{M}_{\Gamma}$ are completely determined by $\mathcal{B}_{\Gamma}$.

Then, it can be verified that
\begin{equation*}
|X_{\Gamma;n}|=\underset{k\in J_{\Gamma;n}}{\prod}|\Sigma_{k}(\Gamma)|^{\alpha_{\Gamma}(k;n)}.
\end{equation*}
Therefore, the result follows immediately.
%
%
%
%

\end{proof}

The following three-dimensional system illustrates the methods and results.

\begin{example}
\label{example:2.11}
For $d=3$, consider
\begin{equation*}
\mathbb{X}_{2,3,5}^{0}=\left\{
(x_{1},x_{2},x_{3},\cdots )\in\{0,1\}^{\mathbb{N}} \mid x_{k}x_{2k}x_{3k}x_{5k}=0 \text{ for all } k\geq 1
 \right\}.
\end{equation*}

Then,

\begin{equation*}
\begin{array}{c}
\psfrag{a}{$2^{m}$}
\psfrag{b}{$3^{n}$}
\psfrag{c}{$5^{r}$}
\includegraphics[scale=0.8]{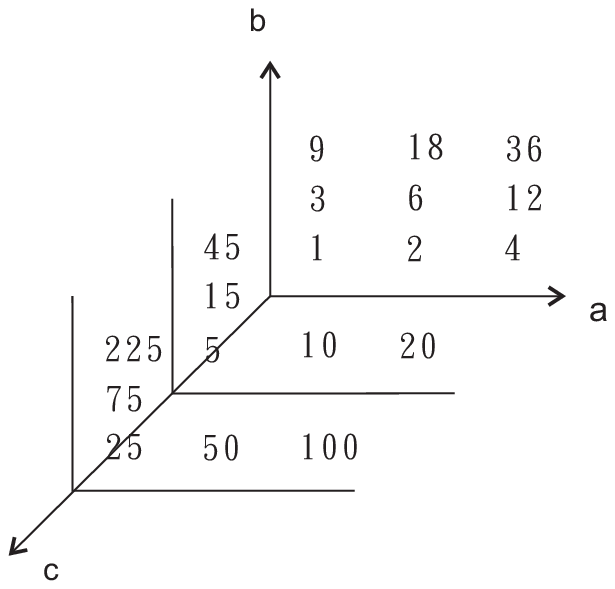}  \\
\mathbb{M}_{2,3,5}
\end{array}
\end{equation*}
\begin{equation*}
\text{Table 2.5.}
\end{equation*}
The first five numbered lattices are listed as follows.

\begin{equation*}
\begin{array}{lllll}
\psfrag{a}{$M_{1}$ }
\includegraphics[scale=0.4]{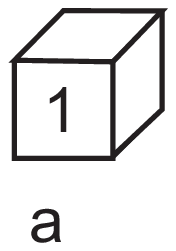} & \hspace{0.3cm}\psfrag{a}{$M_{2}$ }
\includegraphics[scale=0.4]{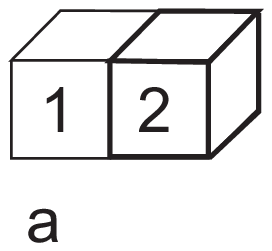}  & \hspace{0.3cm}\psfrag{a}{$M_{3}$ }
\includegraphics[scale=0.4]{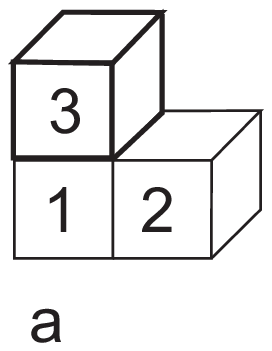}  & \hspace{0.3cm} \psfrag{a}{$M_{4}$ }
\includegraphics[scale=0.4]{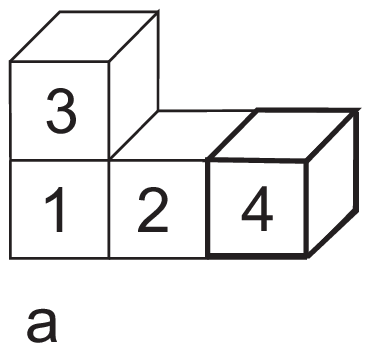} &  \psfrag{a}{$M_{5}$ }
\includegraphics[scale=0.4]{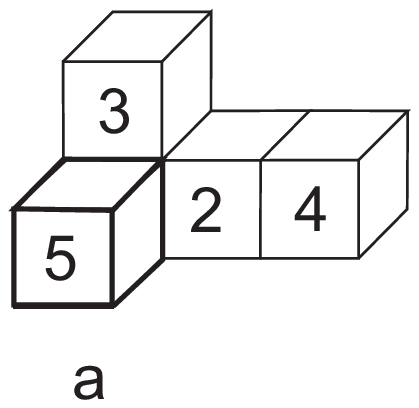}
\end{array}
\end{equation*}
\begin{equation*}
\text{Figure 2.2.}
\end{equation*}
The basic lattice is
\begin{equation*}
\psfrag{a}{$\mathbb{L}_{2,3,5}=$ }
\psfrag{b}{.}
\includegraphics[scale=0.4]{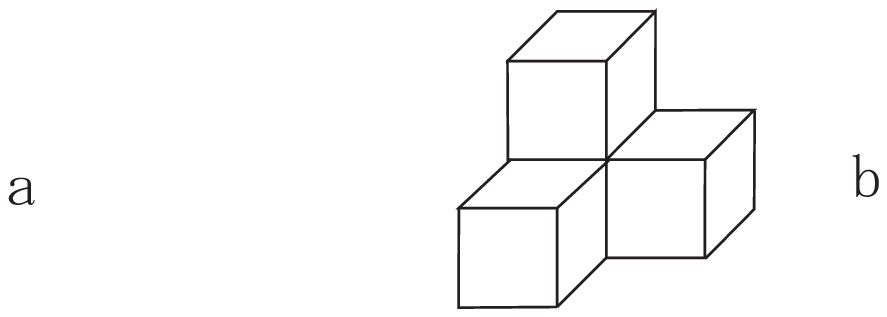}
\end{equation*}

Clearly,
\begin{equation*}
\mathcal{I}_{2,3,5}=\left\{30k+j\hspace{0.2cm}|\hspace{0.2cm} j\in\{1,7,11,13,17,19,23,29\} \text{ and }k\geq 0         \right\}.
\end{equation*}
Therefore, it can be verified that

\begin{equation}\label{eqn:2.41-1}
h(\mathbb{X}_{2,3,5}^{0})=\underset{k=1}{\overset{\infty}{\sum}}\beta_{2,3,5} \left( \frac{1}{q_{k}}-\frac{1}{q_{k+1}}\right)\log|\Sigma_{k}(2,3,5)|,
\end{equation}
where $\beta_{2,3,5}=\frac{4}{15}$ and the forbidden set of $\Sigma_{k}(2,3,5)$ is $\left\{\begin{array}{c}\includegraphics[scale=0.2]{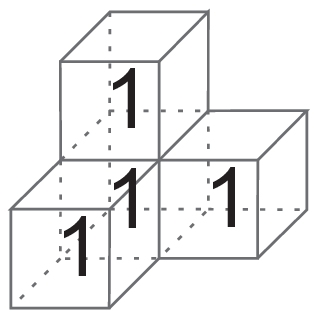}\end{array}\right\}$.

The $n$-th order approximation of (\ref{eqn:2.41-1}) is as follows. For $n\geq1$, let
\begin{equation*}
h^{(n)}(\mathbb{X}_{2,3,5}^{0})=\underset{k=1}{\overset{n}{\sum}}\beta_{2,3,5} \left( \frac{1}{q_{k}}-\frac{1}{q_{k+1}}\right)\log|\Sigma_{k}(2,3,5)|.
\end{equation*}
In Table 2.6, some cases for $|\Sigma_{n}(2,3,5)|$ in $1\leq n\leq 25$ are listed.

\begin{equation*}
\begin{array}{l}
\begin{tabular}{|c|c|c|c|c|c|}
\hline
 $n$ & 5 &10 & 15& 20& 25
\\\hline
$|\Sigma_{n}(2,3,5)|$ & 30& 904 & 25,720 & 738,816  & 19,959,552   \\
\hline
\end{tabular}
\end{array}
\end{equation*}

\begin{equation*}
\text{Table 2.6.}
\end{equation*}
Moreover, $h^{(25)}(\mathbb{X}_{2,3,5}^{0})\approx 0.548837$.
\end{example}

The previous idea also applies to system $\mathbb{X}_{\Gamma}^{0}$ that does not satisfy the conditions (\ref{eqn:1.20-1}) and (\ref{eqn:1.20-2}), where $\Gamma=\left\{\gamma_{1},\gamma_{2},\cdots, \gamma_{d}\right\}$.
Let
$C^{*}$ be the least common multiple of $\gamma_{1},\gamma_{2},\cdots, \gamma_{d}$.
Denote by

\begin{equation}\label{eqn:2.35-10}
\Gamma^{*}=\{p_{1},p_{2},\cdots, p_{Q}\}
\end{equation}
the set of prime factors of $C^{*}$ with $p_{1}<p_{2}<\cdots < p_{Q}$, $Q\geq 1$.
Clearly, $\Gamma^{*}$ satisfies both (\ref{eqn:1.20-1}) and (\ref{eqn:1.20-2}).

Then, $\mathbb{X}_{\Gamma}^{0}$ can be studied by using $\mathbb{M}_{\Gamma}^{*}$ and $\mathcal{I}_{\Gamma^{*}}$.
Denote by

\begin{equation}\label{eqn:2.35-11}
\begin{array}{rl}
\mathbb{M}_{\Gamma^{*}}= & \left\{ p_{1}^{m_{1}}p_{2}^{m_{2}}\cdots p_{Q}^{m_{Q}} \hspace{0.2cm}|\hspace{0.2cm} m_{j}\geq 0    \right\} \\
& \\
\equiv & \left\{ q^{*}_{k}\right\}_{k=1}^{\infty}
\end{array}
\end{equation}
with $q^{*}_{k}<q^{*}_{j}$ if $k< j$. The complementary index set $\mathcal{I}_{\Gamma^{*}}$ of $\mathbb{M}_{\Gamma^{*}}$ is defined by

\begin{equation}\label{eqn:2.35-12}
\mathcal{I}_{\Gamma^{*}}=\left\{ n\in \mathbb{N}\hspace{0.2cm}|\hspace{0.2cm} p_{j}\nmid n, 1\leq j\leq Q    \right\}.
\end{equation}
For $k\geq 1$, let the $k$-cell lattice $L_{k}^{*}$ of $\mathbb{X}_{\Gamma}^{0}$ be
\begin{equation}\label{eqn:2.39-20}
L_{k}^{*}=\left\{(i_{1},i_{2},\cdots, i_{Q})\in \mathbb{Z}^{Q} \hspace{0.1cm} \mid \hspace{0.1cm} p_{1}^{i_{1}}p_{2}^{i_{2}}\cdots p_{Q}^{i_{Q}}\leq q^{*}_{k} \text{ for } i_{q}\geq 0, 1\leq q\leq Q\right\}.
\end{equation}

Now, the constraint

\begin{equation}\label{eqn:2.40}
x_{k}x_{\gamma_{1}k}x_{\gamma_{2}k}\cdots x_{\gamma_{d}k}=0
\end{equation}
can be expressed in terms of $\Gamma^{*}$. Indeed,
define the basic lattice $\mathbb{L}_{\Gamma}$ of $\mathbb{X}_{\Gamma}^{0}$ by
\begin{equation}\label{eqn:2.39-11}
\mathbb{L}_{\Gamma}=\left\{(i_{1},i_{2},\cdots, i_{Q})\in \mathbb{Z}^{Q} \hspace{0.1cm} \mid \hspace{0.1cm}  p_{1}^{i_{1}}p_{2}^{i_{2}}\cdots p_{Q}^{i_{Q}}\in\{1,\gamma_{1},\gamma_{2},\cdots,\gamma_{d}\} \right\}.
\end{equation}
Then, the forbidden set $\mathcal{F}_{\Gamma}$ is given by
\begin{equation*}
\mathcal{F}_{\Gamma}=\left\{U=(u_{i_{1},i_{2},\cdots, i_{Q}})\in \{0,1\}^{\mathbb{L}_{\Gamma}} \hspace{0.1cm} \mid \hspace{0.1cm} u_{i_{1},i_{2},\cdots, i_{Q}}=1 \text{ for all }(i_{1},i_{2},\cdots, i_{Q})\in\mathbb{L}_{\Gamma}   \right\}.
\end{equation*}
Therefore, the basic set of admissible patterns can be written

\begin{equation}\label{eqn:2.40-10}
\mathcal{B}_{\Gamma}=\{0,1\}^{\mathbb{L}_{\Gamma}} \setminus \mathcal{F}_{\Gamma}.
\end{equation}
Notably, $\mathcal{B}_{\Gamma}$ induces a $Q$-dimensional shift of finite type $\Sigma(\mathcal{B}_{\Gamma})$.

Let $\Sigma_{k}(\mathcal{B}_{\Gamma})$ be the set of all admissible patterns that can be determined by  $\mathcal{B}_{\Gamma}$ on $L_{k}^{*}$, $k\geq 1$. In the following, Theorem 2.6 is generalized for $\mathbb{X}_{\Gamma}^{0}$ without the conditions (\ref{eqn:1.20-1}) and (\ref{eqn:1.20-2}).

\begin{theorem}
\label{theorem:2.10-1}
Let $\Gamma=\left\{ \gamma_{1},\gamma_{2},\cdots, \gamma_{d} \right\}$. Then, the entropy of $\mathbb{X}_{\Gamma}^{0}$ is given by

\begin{equation}\label{eqn:2.41}
h(\mathbb{X}_{\Gamma}^{0})=\underset{k=1}{\overset{\infty}{\sum}}\beta_{\Gamma^{*}} \left( \frac{1}{q^{*}_{k}}-\frac{1}{q^{*}_{k+1}}\right)\log|\Sigma_{k}(\mathcal{B}_{\Gamma})|,
\end{equation}
where

\begin{equation}\label{eqn:2.41-0}
\beta_{\Gamma^{*}}=\frac{\sharp\left(\mathcal{I}_{\Gamma^{*}}\cap \left[1,p_{1}p_{2}\cdots p_{Q} \right]\right)}{p_{1}p_{2}\cdots p_{Q}},
\end{equation}
$\Gamma^{*}$, $\mathbb{M}_{\Gamma^{*}}$ and $\mathcal{I}_{\Gamma^{*}}$ are given by (2.41), (2.42) and (2.43), respectively.

\end{theorem}

\begin{proof}

%

First, from the construction of (\ref{eqn:2.35-10}), (\ref{eqn:2.35-11}) and (\ref{eqn:2.35-12}), it is clear that
\begin{equation*}
\mathbb{N}=\underset{i\in \mathcal{I}_{\Gamma^{*}}}{\bigcup}\hspace{0.2cm}i\mathbb{M}_{\Gamma^{*}}
\end{equation*}
and for $i,j\in \mathcal{I}_{\Gamma^{*}}$ with $i\neq j$,
\begin{equation*}
i\mathbb{M}_{\Gamma^{*}}\cap j\mathbb{M}_{\Gamma^{*}}=\emptyset.
\end{equation*}
It is easy to see that$\gamma_{q}\in \mathbb{M}_{\Gamma^{*}}$ for $1\leq q\leq d$. Moreover, if $n\in i\mathbb{M}_{\Gamma^{*}}$ for some $i\in \mathcal{I}_{\Gamma^{*}}$, then $\gamma_{q}n\in i\mathbb{M}_{\Gamma^{*}}$ for all $1\leq q\leq d$.
Hence, from the constraint (\ref{eqn:2.40}), the admissible patterns on $i_{1}\mathbb{M}_{\Gamma^{*}}$ and the admissible patterns on $i_{2}\mathbb{M}_{\Gamma^{*}}$ are mutually independent for $i_{1}\neq i_{2}\in \mathcal{I}_{\Gamma^{*}}$.

As in the proof of Theorem 2.6, we can define $J_{\Gamma^{*};n}$, $q_{\Gamma^{*}}(n;i)$, $I_{\Gamma^{*}}(k;n)$ and $\alpha_{\Gamma^{*}}(k;n)$. It can be proven that

\begin{equation*}
I_{\Gamma^{*}}(k;n)=\left(\frac{n}{q^{*}_{k+1}},\frac{n}{q^{*}_{k}}\right]\bigcap \mathcal{I}_{\Gamma^{*}}
\end{equation*}
and

\begin{equation*}
\underset{n\rightarrow \infty}{\lim} \frac{\alpha^{*}_{\Gamma}(k;n)}{n}=\beta_{\Gamma^{*}} \left( \frac{1}{q^{*}_{k}}-\frac{1}{q^{*}_{k+1}}\right).
\end{equation*}

Next, the constraint (\ref{eqn:2.40}) and the construction of $\mathcal{B}_{\Gamma}$ imply that the admissible patterns on $i\mathbb{M}_{\Gamma^{*}}$, $i\in\mathcal{I}_{\Gamma^{*}}$, are completely determined by $\mathcal{B}_{\Gamma}$.
Hence,
\begin{equation*}
|X_{\Gamma;n}|=\underset{k\in J_{\Gamma^{*};n}}{\prod}|\Sigma_{k}(\mathcal{B}_{\Gamma})|^{\alpha_{\Gamma^{*}}(k;n)},
\end{equation*}
where $\Sigma_{k}(\mathcal{B}_{\Gamma})$ the set of all admissible patterns determined by $\mathcal{B}_{\Gamma}$ on $L^{*}_{k}$.
Therefore, (\ref{eqn:2.41}) follows. The proof is complete.
\end{proof}

 %

%
%
%
%
%
%
%
%
%
%
%

%
%
%
%
%
%
%
%
%
%
%
%
%
%

The following example illustrates Theorem 2.8.

\begin{example}
\label{example:2.7-1}
%
%
%
%
%

Consider $\mathbb{X}_{2,8}^{0}$. It is easy to see that $\Gamma^{*}=\{2\}$. From (2.46), the basic lattice $\mathbb{L}_{2,8}=\{0,1,3\}=\includegraphics[scale=0.5]{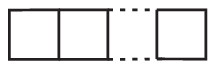}$, here the third cell is deleted. The forbidden set $\mathcal{F}_{2,8}$ is $\left\{\includegraphics[scale=0.5]{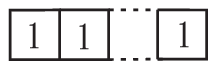}\right\}$ and
$\mathcal{B}_{2,8}=\{0,1\}^{\mathbb{L}_{2,8}}\setminus \mathcal{F}_{2,8}$. Define the associated transition matrix

\begin{equation*}
\psfrag{a}{$A(\mathcal{B}_{2,8})$=}
\psfrag{b}{.}
\includegraphics[scale=0.5]{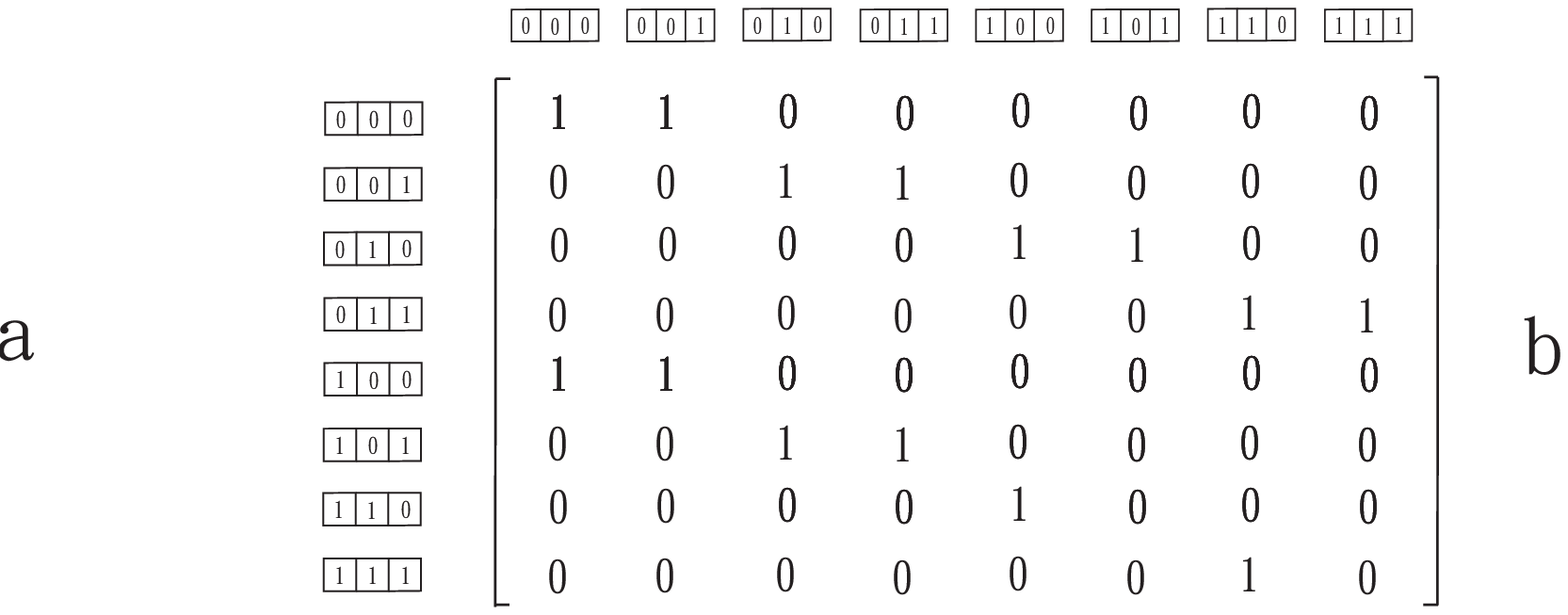}
\end{equation*}
Then,

\begin{equation*}
h(\mathbb{X}_{2,8}^{0})=\underset{k=1}{\overset{\infty}{\sum}}\frac{1}{2^{k+1}}\log |\Sigma_{k}(\mathcal{B}_{2,8})|,
\end{equation*}
where
$|\Sigma_{m}(\mathcal{B}_{2,8})|=2^{m}$ and $|\Sigma_{n}(\mathcal{B}_{2,8})|=\left|A(2,8)^{n-3}\right|$ for $1\leq m \leq 3$ and $n\geq 4$.

\end{example}

In the remaining of this section, the constraint (\ref{eqn:2.40}) is further relaxed. Therefore, we can study more general case than $\mathbb{X}_{\Gamma}^{0}$.
For simplicity, only $\Gamma$ that satisfies conditions (\ref{eqn:1.20-1}) and (\ref{eqn:1.20-2}) is studied.

For any $N\geq 2$, consider a multiplicative system is of $N$-symbols, $\{0,1,2,\cdots,N-1\}$.
For any $d\geq 1$, let the constraint set $\mathcal{C}$ be a subset of $\left\{0,1,\cdots,(N-1)^{d}\right\}$.
Denote by $\mathbb{X}_{\Gamma}(N,\mathcal{C})$ the multiplicative integer system with constraint set $\mathcal{C}$:

\begin{equation}\label{eqn:2.36}
\begin{array}{rl}
 & \mathbb{X}_{\Gamma}(N,\mathcal{C}) \\
 & \\
 = & \left\{
(x_{1},x_{2},\cdots )\in\{0,1,\cdots,N-1\}^{\mathbb{N}} \mid x_{k}x_{\gamma_{1}k}\cdots x_{\gamma_{d}k}\in \mathcal{C} \text{ for } k\geq 1
 \right\}.
 \end{array}
\end{equation}
Then, the basic set $\mathcal{B}_{\Gamma}(N,\mathcal{C})$ of admissible patterns on $\mathbb{L}_{\Gamma}$ is given by

\begin{equation}\label{eqn:2.36-2}
\begin{array}{rl}
& \mathcal{B}_{\Gamma}(N,\mathcal{C}) \\
& \\
= & \left\{U=(u_{i_{1},i_{2},\cdots, i_{d}})\in \{0,1,\cdots,N-1\}^{\mathbb{L}_{\Gamma}} \hspace{0.1cm} \mid \hspace{0.1cm} \underset{(i_{1},i_{2},\cdots, i_{d})\in \mathbb{L}_{\Gamma}}{\prod}u_{i_{1},i_{2},\cdots, i_{d}}\in\mathcal{C}    \right\}.
\end{array}
\end{equation}

The following theorem can be proven as Theorem 2.6.

\begin{theorem}
\label{theorem:2.11}
Let $\Gamma=\left\{ \gamma_{1},\gamma_{2},\cdots, \gamma_{d} \right\}$ satisfy (\ref{eqn:1.20-1}) and (\ref{eqn:1.20-2}) and  $\mathcal{C}\subseteq \left\{0,1,\cdots,(N-1)^{d}\right\}$. The entropy of $\mathbb{X}_{\Gamma}(N,\mathcal{C})$ is given by

\begin{equation}\label{eqn:2.41-20}
h(\mathbb{X}_{\Gamma}(N,\mathcal{C}))=\underset{k=1}{\overset{\infty}{\sum}}\beta_{\Gamma} \left( \frac{1}{q_{k}}-\frac{1}{q_{k+1}}\right)\log|\Sigma_{k}(\mathcal{B}_{\Gamma }(N,\mathcal{C}))|,
\end{equation}
where
$\Sigma_{k}(\mathcal{B}_{\Gamma }(N, \mathcal{C}))$ is the set of $d-$dimensional admissible local patterns that can be generated by $\mathcal{B}_{\Gamma }(N, \mathcal{C})$ on $L_{k}$.
\end{theorem}

\begin{proof}
This proof is similar to the proof of Theorem 2.6. The only difference between $\mathbb{X}_{\Gamma}^{0}$ and $\mathbb{X}_{\Gamma}(N,\mathcal{C})$ is their constraints.
By (\ref{eqn:2.36}), it is easy to see that the basic set $\mathcal{B}_{\Gamma }(N,\mathcal{C})$ can completely determine the patterns on $i\mathbb{M}_{\Gamma}$ for $i\in \mathcal{I}_{\Gamma}$.
Therefore, the result follows.

\end{proof}

The following example illustrates Theorem 2.10.

\begin{example}
\label{example:2.11-1}
Let $N=3$ and $\mathcal{C}=\{0,2\}$. Then

\begin{equation*}
\mathbb{X}_{2}(3,\mathcal{C})=\left\{
(x_{1},x_{2},x_{3},\cdots )\in\{0,1,2\}^{\mathbb{N}} \mid x_{k}x_{2k}\in \{0,2\} \text{ for all } k\geq 1
 \right\}.
\end{equation*}
The basic set of admissible local patterns is now given by
\begin{equation*}
\mathcal{B}_{2}(3, \mathcal{C})=\left\{
\begin{array}{ccccccc}
\includegraphics[scale=0.5]{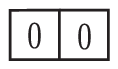}, & \includegraphics[scale=0.5]{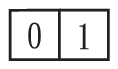}, & \includegraphics[scale=0.5]{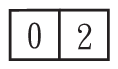}, & \includegraphics[scale=0.5]{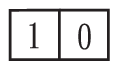}, &
\includegraphics[scale=0.5]{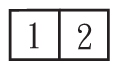}, & \includegraphics[scale=0.5]{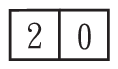}, & \includegraphics[scale=0.5]{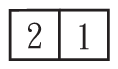}
\end{array}
\right\}.
\end{equation*}
The associated transition matrix is

\begin{equation*}
\psfrag{a}{$A(2;3,\mathcal{C})$=}
\psfrag{b}{.}
\includegraphics[scale=0.5]{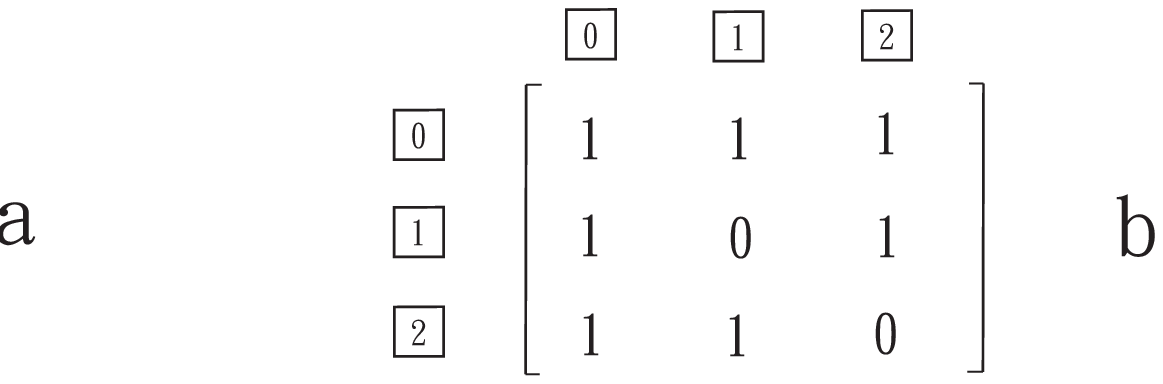}
\end{equation*}
Therefore, as in Theorem 1.2,

\begin{equation*}
h(\mathbb{X}_{2}(3,\mathcal{C}))=4\hspace{0.2cm}\underset{k=1}{\overset{\infty}{\sum}}\frac{1}{3^{k+1}}\log a_{k}(2;3,\mathcal{C}),
\end{equation*}
where $a_{1}(2;3,\mathcal{C})=3$, $a_{k}(2;3,\mathcal{C})=\left|A(2;3,\mathcal{C})^{k-1}\right|$ for all $k\geq 2$.

\end{example}
\section{One-dimensional coupled systems}
\hspace{0.4cm}This section investigates the one-dimensional coupled system which is an intersection of the multiplicative integer system $\mathbb{X}_{Q}^{0}$ with an additive proper shift of finite type $\Sigma_{A}$, i.e.,

\begin{equation}\label{eqn:3.1}
\mathbb{X}_{Q}^{A}\equiv \mathbb{X}_{Q}^{0} \cap \Sigma_{A}.
\end{equation}

A simple system is considered first; the findings are then extended to general systems. Consider

\begin{equation}\label{eqn:3.2}
\begin{array}{rl}
\mathbb{X}^{A}_{2}\equiv & \mathbb{X}_{2}^{0} \cap \Sigma_{A} \\
& \\
= & \left\{
(x_{1},x_{2},x_{3},\cdots )\in\{0,1\}^{\mathbb{N}} \mid x_{k}x_{2k}=0 \text{ for all } k\geq 1 \text{ and } (x_{1},x_{2},x_{3},\cdots )\in\Sigma_{A}
 \right\}.
 \end{array}
\end{equation}

To incorporate the effect of $\Sigma_{A}$, Table 1.1 is replaced by the following figure.

\begin{equation*}
\includegraphics[scale=0.6]{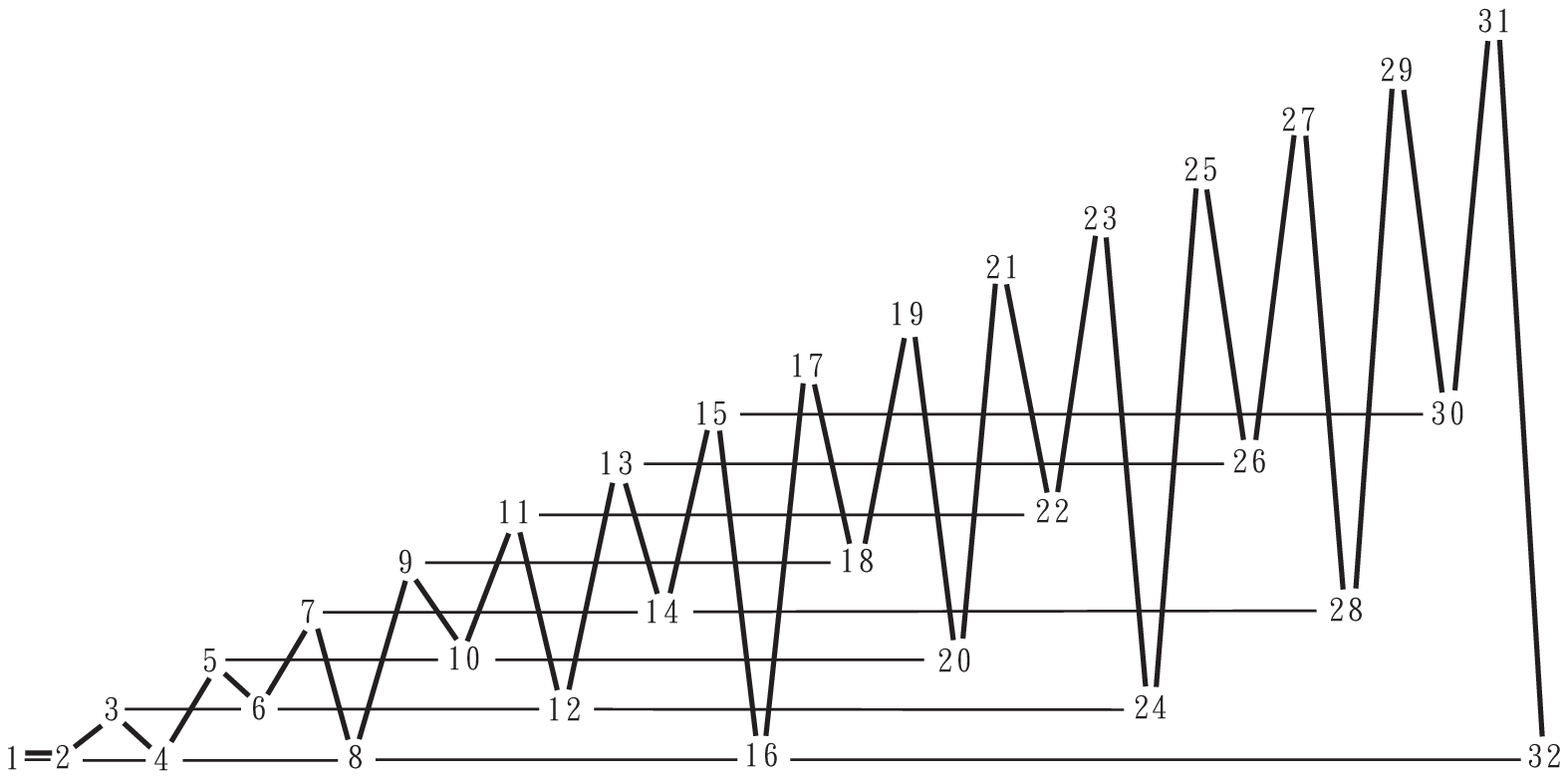}
\end{equation*}
\begin{equation*}
\text{Figure 3.1.}
\end{equation*}

As in Table 1.1, the horizontal lines in Fig. 3.1 connect the integers in $i\mathbb{M}_{2}$ for each $i\in\mathcal{I}_{2}$, the effect comes from $\mathbb{X}_{2}^{0}$. On the other hand, the bold zigzag line in Fig. 3.1 connects all natural integers comes from $\Sigma_{A}$. Therefore, for any $i\neq j$ in $\mathcal{I}_{2}$, $i\mathbb{M}_{2}$ and $j\mathbb{M}_{2}$ are no longer mutually independent. In fact, they are all coupled through the relation set $\mathbb{M}_{2}$. Therefore, (\ref{eqn:3.2}) is regarded as a coupled system.

Before the system $\mathbb{X}_{2}^{A}$ is decoupled, the following definition is needed.

\begin{definition}
\label{definition:3.0}
Two sets of integers of $M$ and $M'$ are mutually independent in $\mathbb{X}_{2}^{A}$ if
\begin{equation}\label{eqn:3.3-0}
M\cap M'=\emptyset
\end{equation}
and any numbers $m$ in $M$ and $m'$ in $M'$ are not consecutive and also not consecutive in $2$-power, i.e., if $m=2^{n}$ for some $n$ then $m'\neq 2^{n+1}$ and $2^{n-1}$.

\end{definition}

Then, the following lemma can be obtained.

\begin{lemma}
\label{lemma:3.0-1}
Suppose $M$ and $M'$ are mutually independent in $\mathbb{X}_{2}^{A}$. Then
\begin{equation}\label{eqn:3.3-1}
|\Sigma(M\cup M')|=|\Sigma(M)||\Sigma(M')|,
\end{equation}
where $\Sigma(M)$ is the set of all admissible patterns on lattice $M$, and $\Sigma(M')$ and $\Sigma(M\cup M')$ are defined analogously.
\end{lemma}

\begin{proof}
Since $M$ and $M'$ are decoupled in $\mathbb{X}_{2}^{A}$, the patterns in $\Sigma(M)$ are independent of the patterns in $\Sigma(M')$. Therefore, the result holds.
\end{proof}

The strategy for studying (\ref{eqn:3.2}) is to decouple the whole system into disjoint pieces which are located in some proper subset $\tilde{\mathbb{X}}_{2}^{A}$ of $\mathbb{X}_{2}^{A}$.

%
From the reduced system $\tilde{\mathbb{X}}_{2}^{A}$, a sequence $\left\{\mathbb{X}_{2}^{A}(m)\right\}_{m=1}^{\infty}$ of independent decoupled subsystems are chosen. Then, the entropy of the decoupled independent system $\mathbb{X}_{2}^{A}(m)$ can be computed easily. An appropriate choice of $\mathbb{X}_{2}^{A}(m)$ is demonstrated to enable the recovery of the entropy of $\mathbb{X}_{2}^{A}$, i.e.,

 \begin{equation}\label{eqn:3.5}
\underset{m\rightarrow\infty}{\lim}h(\mathbb{X}_{2}^{A}(m))=h(\mathbb{X}_{2}^{A}).
\end{equation}

As in decoupled system $\mathbb{X}_{2}^{0}$, the admissible numbered lattice $M_{k}$ in $\tilde{\mathbb{X}}_{2}^{A}$ is firstly picked up.
Indeed, in Fig. 3.2, some $M_{k}(l)$ are drawn for $1\leq k\leq 4$.

\begin{equation*}
\begin{array}{cccc}
\psfrag{a}{{\scriptsize(a) $M_{1}(3)$}}
\includegraphics[scale=0.5]{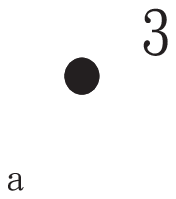} & \hspace{1.0cm}
\psfrag{a}{{\scriptsize(b) $M_{1}(l)$}}
\psfrag{l}{$l$}
\includegraphics[scale=0.5]{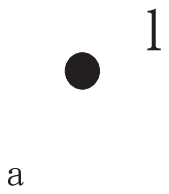} & \hspace{1.0cm}
\psfrag{a}{{\scriptsize (c) $M_{2}(3)$}}
\includegraphics[scale=0.5]{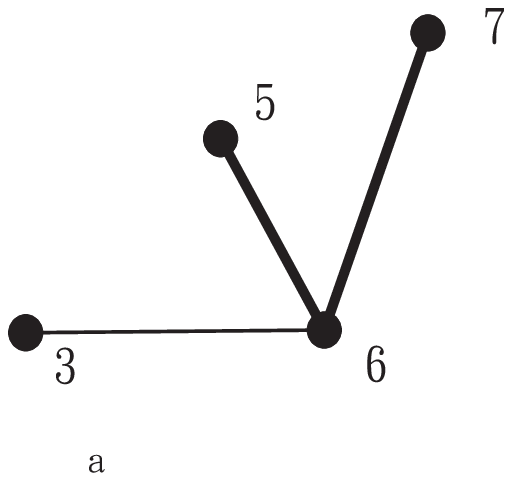}  & \hspace{1.0cm}
\psfrag{a}{{\scriptsize (d) $M_{2}(l)$}}
\psfrag{b}{{\scriptsize$l$}}
\psfrag{c}{{\scriptsize$2l$}}
\psfrag{d}{{\scriptsize$2l-1$}}
\psfrag{e}{\scriptsize{$2l+1$}}
\includegraphics[scale=0.5]{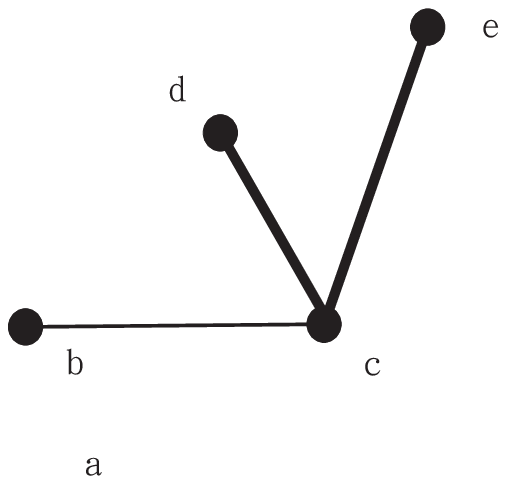}
\end{array}
\end{equation*}

\begin{equation*}
\begin{array}{cc}
\psfrag{a}{{\scriptsize (e) $M_{3}(3)$}}
\includegraphics[scale=0.3]{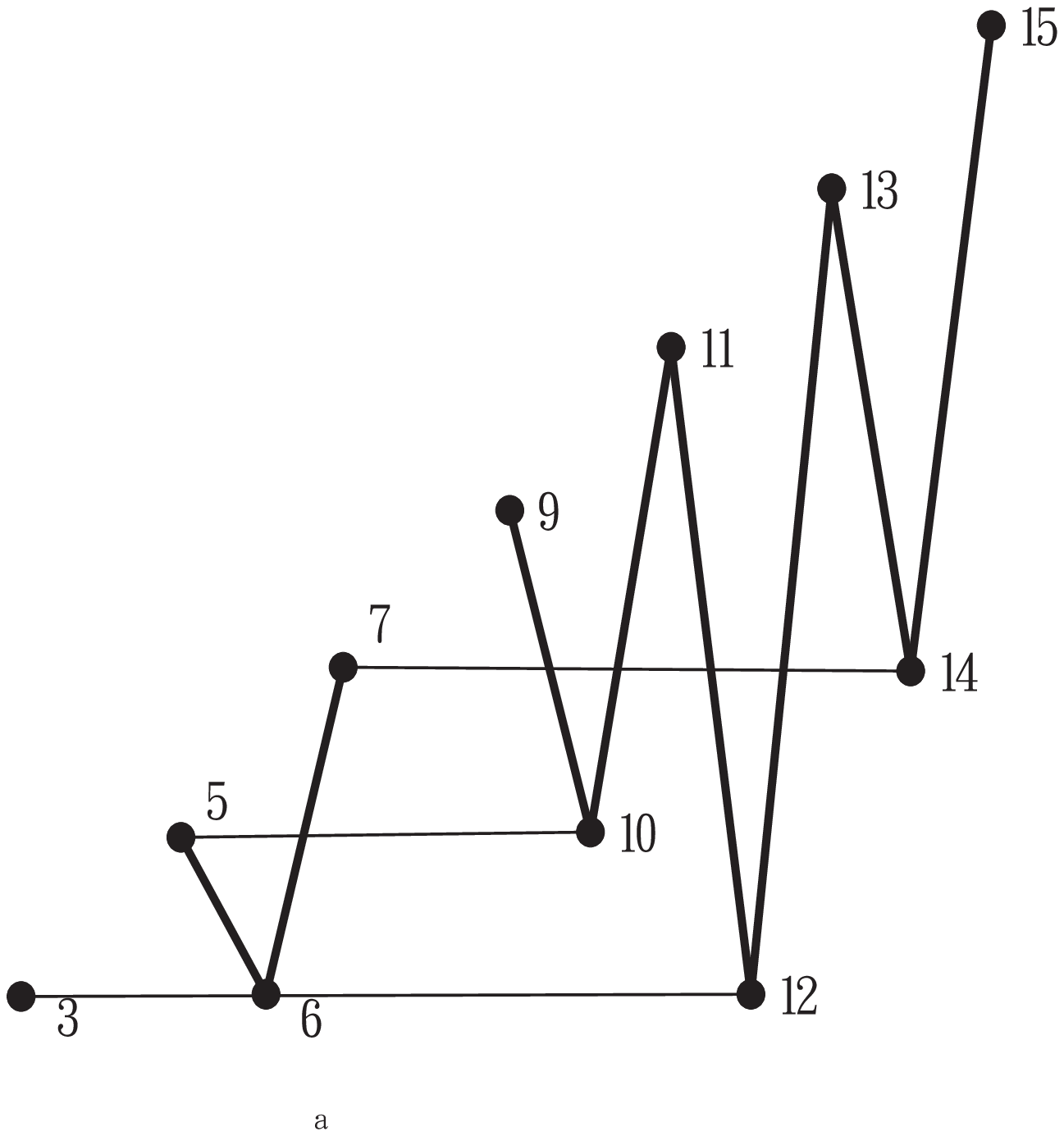} & \hspace{1.0cm}
\psfrag{a}{{\tiny (f) $M_{3}(l)$}}
\psfrag{b}{{\tiny$l$}}
\psfrag{c}{{\tiny$2l$}}
\psfrag{d}{{\tiny$4l$}}
\psfrag{e}{{\tiny$2l-1$}}
\psfrag{f}{{\tiny$4l-2$}}
\psfrag{g}{{\tiny$2l+1$}}
\psfrag{h}{{\tiny$4l+2$}}
\psfrag{j}{{\tiny$4l-3$}}
\psfrag{k}{{\tiny$4l-1$}}
\psfrag{l}{{\tiny$4l+1$}}
\psfrag{m}{{\tiny$4l+3$}}
\includegraphics[scale=0.3]{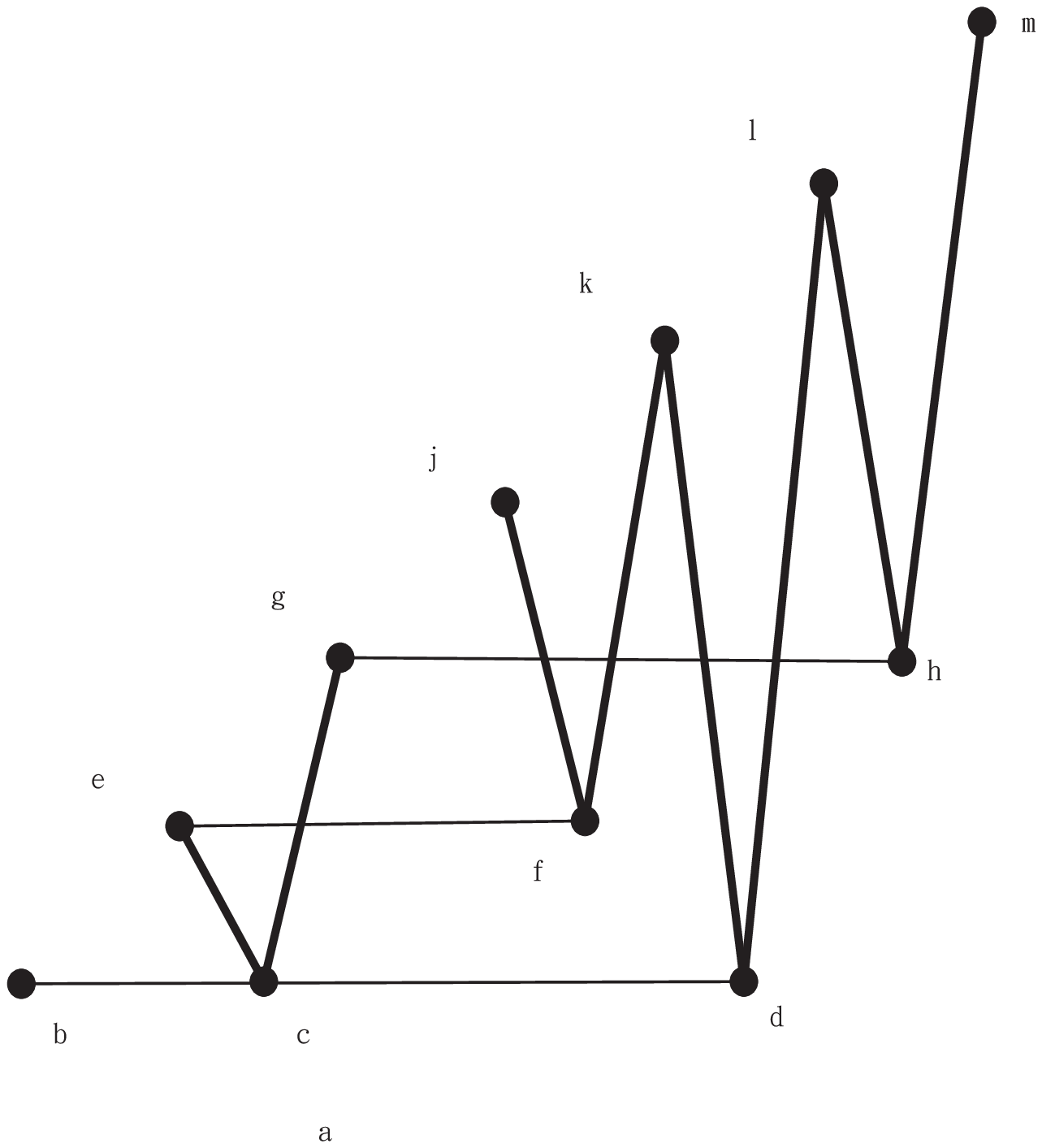}
\end{array}
\end{equation*}

\begin{equation*}
\psfrag{a}{{\scriptsize (g) $M_{4}(3)$}}
\includegraphics[scale=0.3]{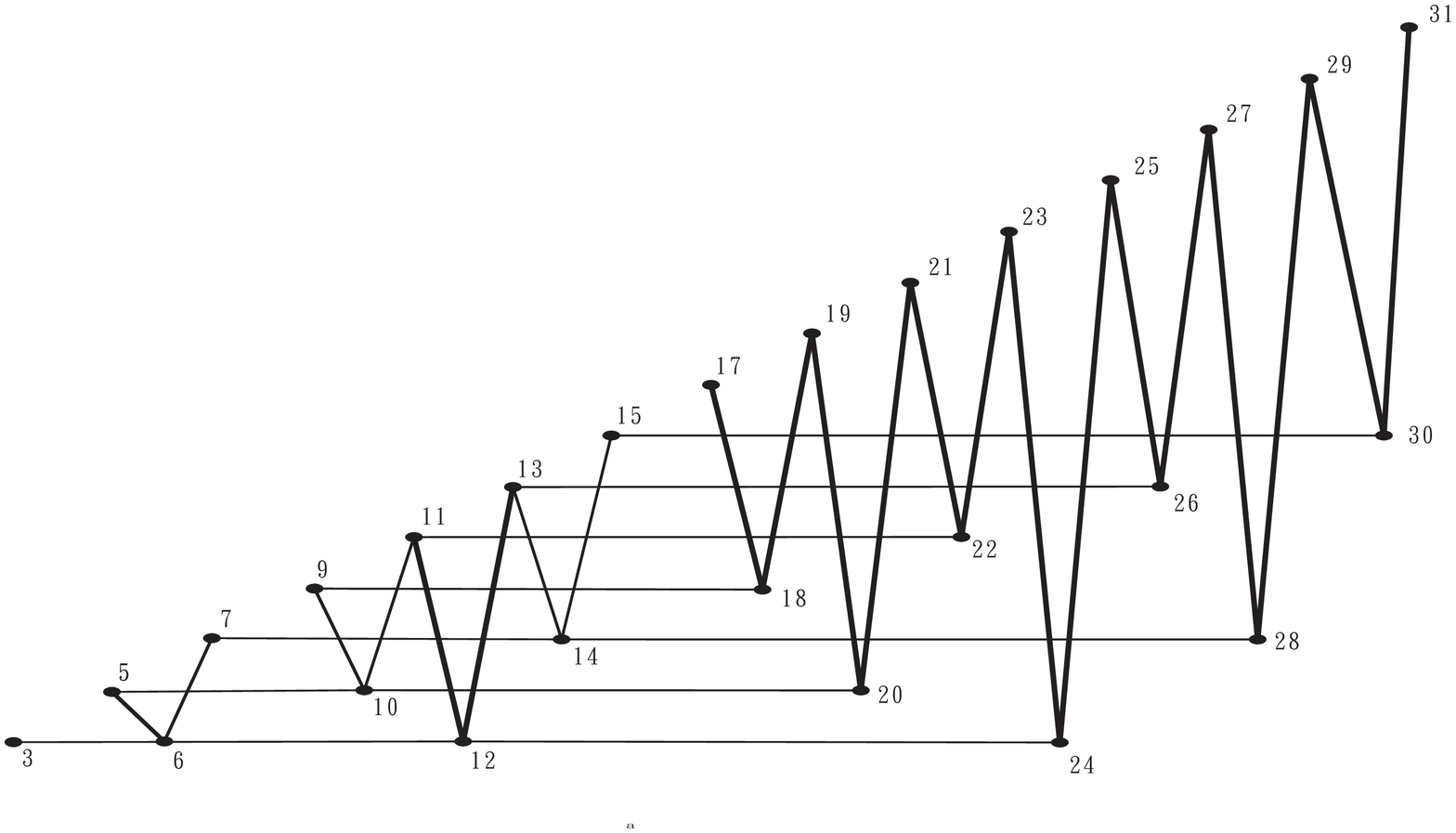}
\end{equation*}

\begin{equation*}
\text{Figure 3.2.}
\end{equation*}

The choice of $M_{k}$ is recursive and robotic. The basic idea is that any number can produce the next generation through $\mathbb{X}_{2}$ or $\Sigma_{A}$. More precisely, for each number $n$, if $n\in\mathcal{I}_{2}$, then $n$ can produce the next generation $2n\in n \mathbb{M}_{2}$. If $n \notin \mathcal{I}_{2}$ with $n=i 2^{m}$, $m\geq 1$, then $n$ can produce $2n=i 2^{m+1}\in i\mathbb{M}_{2}$ through $\mathbb{X}_{2}$ and
$n\pm 1=i 2^{m}\pm 1 \in \mathcal{I}_{2}$ through $\Sigma_{A}$. In summary, a complete production cycle is as follows. If $n\in\mathcal{I}_{2}$, then $n$ produces $2n$ and then $2n\pm 1$. If $n=i 2^{m}$, $m\geq 1$, then $n$ produces $i2^{m+1}$ and then $i2^{m+1}\pm 1$.

For example, $M_{1}(3)$ has one cell, and number $3$ is regarded as the number of first generation. $M_{2}(3)$ is constructed from $M_{1}(3)$ by producing number $6$ from $3$ through $3\mathbb{M}_{2}$. Immediately, $6$ creates numbers $5$ and $7$ as the descendanta through $\Sigma_{A}$. $M_{2}(3)$ is of degree 2 since there are two numbers $\{3,6\}$ on the horizontal line.

The construction of $M_{3}$ from $M_{2}$ is performed similarly: the number $6$ yields number 12 in $3\mathbb{M}_{2}$. On the same time, numbers $5$ and $7$ yield the numbers $10$ and 14 in $5\mathbb{M}_{2}$ and $7\mathbb{M}_{2}$, respectively. Next, numbers 10, 14 and 12 yield their descendants 9, 11, 13, 15 and 11, 13 in $\mathcal{I}_{2}$ through $\Sigma_{A}$, as presented in Fig. 3.2 (e).
Now, the three numbers $3$, $6$ and $12$ are in the lowest horizontal direction, and $M_{3}(3)$ is therefore of degree 3. On $M_{k}(i)$, the maximal number of numbers of cells in the horizontal direction is $k$, and $M_{k}(i)$ is of degree $k$.

Now, $M_{k}(l)$ can be defined formally as follows.
%
%

\begin{definition}
\label{Definition:3.1}

For each $l\in \mathcal{I}_{2}$, define $V_{1}(l)=\{l\}$. For $k\geq 1$, define

\begin{equation*}
V_{k+1}(l)=\left\{ 2i \text{ or }2i\pm 1 \hspace{0.2cm}\mid \hspace{0.2cm} i\in V_{k}(l) \right\}.
\end{equation*}
Then, define $M_{1}(l)=V_{1}(l)$ and for $k\geq 2$,

\begin{equation*}
M_{k+1}(l)\equiv M_{k}(l) \cup V_{k+1}(l).
\end{equation*}

\end{definition}
See Fig. 3.2 for $M_{k}(3)$, $1\leq k\leq 4$.

Notably, $M_{m}(l)$ and $M_{m}(l')$ are mutually independent when $l,l'\in\mathcal{I}_{2}$ and $|l-l'|\geq 2^{m+1}$.

After the lattices $M_{k}$ and $L_{k}$ are identified, in a given range $\mathcal{N}(2^{n})$, (II) is then to be carried out, i.e., the number of disjoint copies of $M_{k}(l)\subset \tilde{\mathbb{X}}_{2}^{A}$ with $l\in\mathcal{I}_{2}$ is computed. For example, in Fig. 3.3, $[1,32]\cap\mathbb{N}$ can be decoupled by $M_{2}(3)$, $M_{2}(9)$, $M_{2}(11)$, $M_{2}(13)$, $M_{2}(15)$ and the numbers in $\{1,2,4,8,10,12,14,16,20,24,28,32\}$ are not used. There are one copy $M_{2}(3)$ in $(2,2^{3})$ and four copies $M_{2}(9)$, $M_{2}(11)$, $M_{2}(13)$, $M_{2}(15)$ in $(2^{3},2^{5})$; see Fig. 3.3.

\begin{equation*}
\includegraphics[scale=0.6]{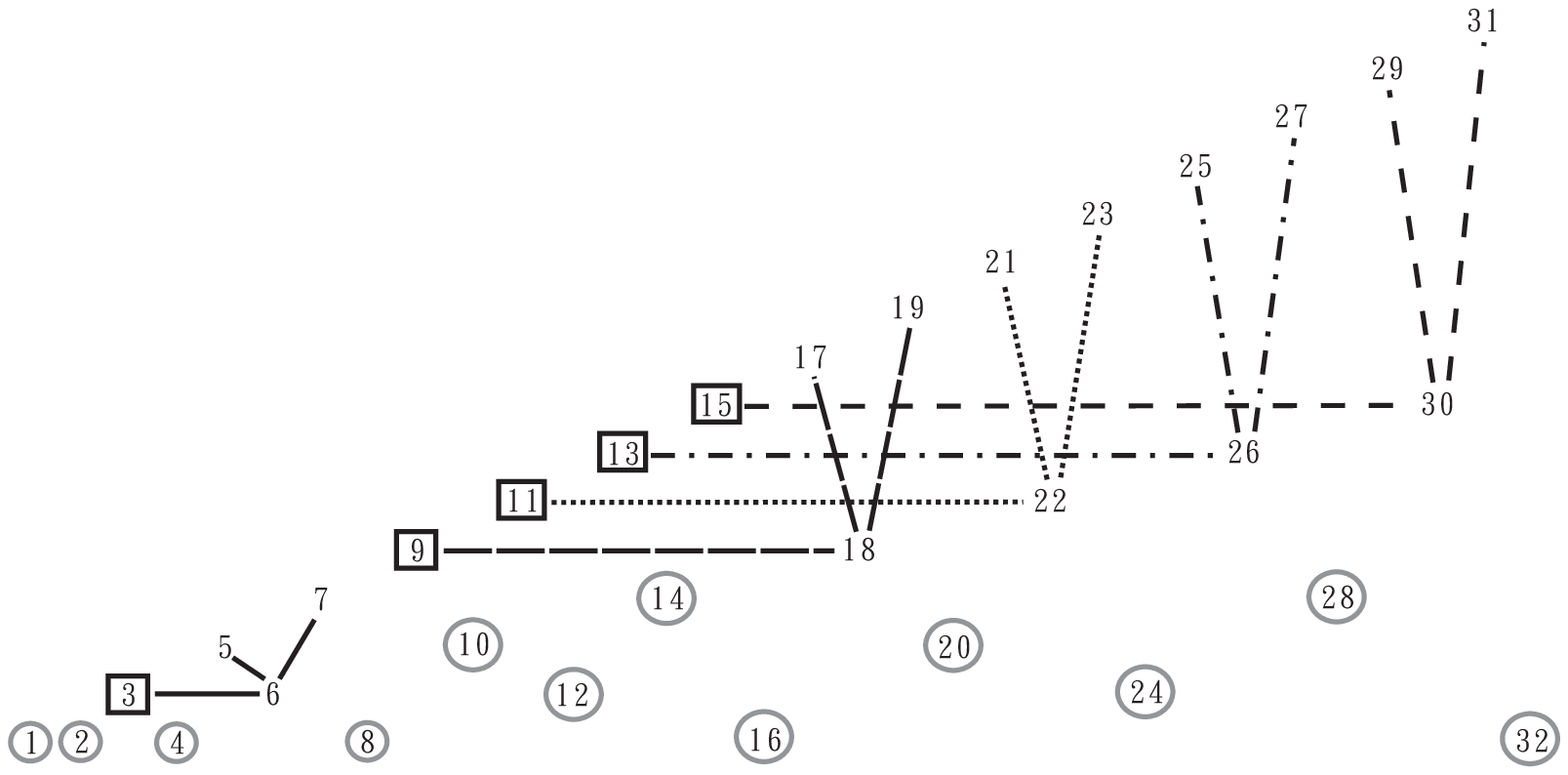}
\end{equation*}
\begin{equation*}
\text{Figure 3.3.}
\end{equation*}

The results for disjoint copies in $\left(1,2^{n}\right)$ can be proven as follows.

\begin{lemma}
\label{lemma:3.1}
Given $k\geq 1$ and $n\geq k+1$, define

\begin{equation}\label{eqn:3.5-50}
m^{*}= \left\lfloor\frac{n-1}{k}\right\rfloor -1.
\end{equation}
Then, within $\left(1,2^{n}\right)$, there are

 \begin{equation}\label{eqn:3.6}
\alpha_{k}(n)\equiv 2^{n-k-1}\left(1+\frac{1}{2^{k}}+\frac{1}{2^{2k}}+\cdots +\frac{1}{2^{m^{*}k}}\right)
 \end{equation}
mutually independent copies of $M_{k}(l)$ with $l\in\mathcal{I}_{2}$.

\end{lemma}

\begin{proof}


 We begin with $k=2$ and $k=3$. It is easy to show that in $\left(2^{n-2},2^{n}\right)$, each odd integer $l\in\left(2^{n-2},2^{n-1}\right)\cap\mathcal{I}_{2}$ can produce a $M_{2}(l)$ that lies in $\left(2^{n-2},2^{n}\right)$ and they are all disjoint; see Fig. 3.4. Therefore, there are totally $2^{n-3}$ copies of $M_{2}$. Similarly, as in Fig. 3.5,  between $\left(2^{n-3},2^{n}\right)$, each $l\in\left(2^{n-3},2^{n-2}\right)\cap\mathcal{I}_{2}$ produces a $M_{3}(l)$ that lies in $\left(2^{n-3},2^{n}\right)$. They are all disjoint too. The total number of copies of $M_{3}(l)$ in $\left(2^{n-3},2^{n}\right)$ is $2^{n-4}$.

\begin{equation*}
\begin{array}{cc}
\psfrag{a}{{\scriptsize$2^{n-2}$}}
\psfrag{b}{{\scriptsize$2^{n-1}$}}
\psfrag{c}{{\scriptsize$2^{n}$}}
\psfrag{f}{{\scriptsize$2^{n-4}$}}
\psfrag{e}{{\scriptsize$2^{n-3}$}}
\psfrag{d}{{\scriptsize$\mathcal{I}_{2}$}}
\includegraphics[scale=0.8]{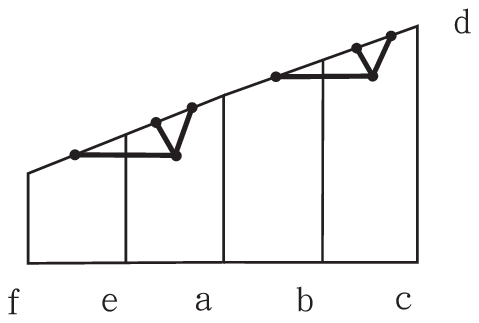}  & \hspace{1.3cm}\psfrag{a}{{\scriptsize$2^{n-2}$}}
\psfrag{b}{{\scriptsize$2^{n-1}$}}
\psfrag{c}{{\scriptsize$2^{n}$}}
\psfrag{e}{{\scriptsize$2^{n-3}$}}
\psfrag{f}{{\scriptsize$2^{n-4}$}}
\psfrag{g}{{\scriptsize$2^{n-5}$}}
\psfrag{h}{{\scriptsize$2^{n-6}$}}
\psfrag{d}{{\scriptsize$\mathcal{I}_{2}$}}
\includegraphics[scale=0.8]{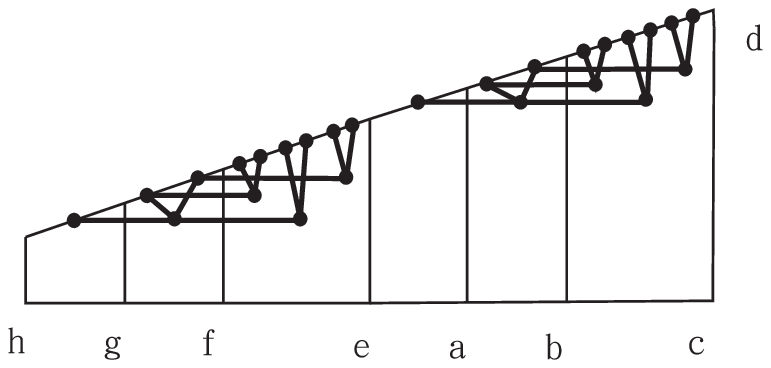}\\
& \\
\text{Figure 3.4.} & \hspace{1.3cm}\text{Figure 3.5.}
\end{array}
\end{equation*}
By using the similar argument, it can be verified that for any $k\geq 2$ and $0\leq l\leq m^{*}$, within $\left(2^{n-k(l+1)},2^{n-kl}\right)$, the number of disjoint copies of $M_{k}$ is

 \begin{equation}\label{eqn:3.6-10}
2^{n-k(l+1)-1}.
 \end{equation}
Therefore,
\begin{equation*}
\begin{array}{rl}
\alpha_{k}(n)= & \underset{l=0}{\overset{m^{*}}{\sum}}2^{n-k(l+1)-1} \\
& \\
= & 2^{n-k-1}\left(1+\frac{1}{2^{k}}+\frac{1}{2^{2k}}+\cdots +\frac{1}{2^{m^{*}k}}\right).
\end{array}
\end{equation*}
The proof is complete.
\end{proof}

Finally, for step (III), denote by $\Sigma_{k}$ the admissible patterns on $L_{k}$. Then, Lemma 3.4 yields the lower bound of the entropy $h(\mathbb{X}_{2}^{A})$.

\begin{lemma}
\label{lemma:3.2}
For any $k\geq 1$,

 \begin{equation}\label{eqn:3.7}
\frac{1}{2\left(2^{k}-1\right)}log|\Sigma_{k}|\leq h(\mathbb{X}_{2}^{A}).
 \end{equation}
\end{lemma}

\begin{proof}
Within $[1,2^{n}]$, we put symbol 0 on the cells that are not used in deriving (\ref{eqn:3.6}).
Then,
\begin{equation*}
|X_{2^{n}}|\geq |\Sigma_{k}|^{\alpha_{k}(n)}.
\end{equation*}
Now, it is easy to see
\begin{equation*}
\underset{n\rightarrow \infty}{\lim}\frac{\alpha_{k}(n)}{2^{n}}=\frac{1}{2\left(2^{k}-1\right)}.
\end{equation*}

Therefore,
\begin{equation*}
\begin{array}{rl}
h(\mathbb{X}_{2}^{A})= & \underset{n\rightarrow \infty}{\lim}\frac{1}{2^{n}}\log \left|X_{2^{n}}\right| \\
& \\
\geq & \left(\underset{n\rightarrow \infty}{\lim}\frac{\alpha_{k}(n)}{2^{n}} \right)\log \left|\Sigma_{k}\right| \\
& \\
= & \frac{1}{2\left(2^{k}-1\right)}log|\Sigma_{k}|
\end{array}
\end{equation*}
The proof is complete.
\end{proof}

Now, a good upper estimate of the entropy $h(\mathbb{X}_{2}^{A})$ remains to be found. From the derivation of Lemma 3.4, we need to estimate the contribution to entropy of the vertices that are not used in deriving (\ref{eqn:3.6}). The following lemma is obtained.

\begin{lemma}
\label{lemma:3.3}
For each $k\geq 1$, the number of the vertices of $L_{k}$ satisfies

 \begin{equation}\label{eqn:3.8}
 \begin{array}{rl}
|L_{k}|= & k+2|L_{k-1}| \\
& \\
=&2\left(2^{k}-1\right)-k.
\end{array}
 \end{equation}
Within $[1,2^{n}]$, the number of the vertices that are not used in deriving (\ref{eqn:3.6}) is

\begin{equation}\label{eqn:3.10}
 \frac{k}{2\left(2^{k}-1\right)}\left(2^{n}-2^{n-k(m^{*}+1)}\right)+2^{n-k(m^{*}+1)}.
\end{equation}
Moreover,

\begin{equation}\label{eqn:3.11}
h(\mathbb{X}_{2}^{A})\leq \frac{1}{2\left(2^{k}-1\right)}\log|\Sigma_{k}|+\frac{k}{2\left(2^{k}-1\right)}\log 2,
\end{equation}
where $|\Sigma_{k}|$ is the number of all admissible patterns on $L_{k}$.

\end{lemma}

\begin{proof}
(\ref{eqn:3.8}) is easily proved by induction. Then, for $0\leq l\leq m^{*}$, within $\left(2^{n-k(l+1)},2^{n-kl}\right]$, by (\ref{eqn:3.6-10}) and (\ref{eqn:3.8}), it can be verified that there are
\begin{equation*}
\begin{array}{rl}
& \left( 2^{n-kl}-2^{n-k(l+1)}\right)-\left[2^{n-k(l+1)-1}\left(2(2^{k}-1)-k\right)\right]  \\
 & \\
= & k \left(2^{n-k(l+1)-1}\right)
\end{array}
\end{equation*}
vertices that are not used in deriving (\ref{eqn:3.6}). Since
\begin{equation*}
[1,2^{n}]=\left(\underset{l=0}{\overset{m^{*}}{\bigcup}}\left(2^{n-k(l+1)},2^{n-kl}\right] \right)\bigcup \left[ 1,2^{n-k(m^{*}+1)}\right].
\end{equation*}
It is easy to see that there are

\begin{equation*}
\begin{array}{rl}
& \left(\underset{l=0}{\overset{m^{*}}{\sum}} \hspace{0.1cm}k \left(2^{n-k(l+1)-1}\right)\right)+2^{n-k(m^{*}+1)} \\
& \\
= & \frac{k}{2\left(2^{k}-1\right)}\left(2^{n}-2^{n-k(m^{*}+1)}\right)+2^{n-k(m^{*}+1)}
\end{array}
\end{equation*}
vertices in $[1,2^{n}]$ that are not used in deriving (\ref{eqn:3.6}).

Since two symbols 0 and 1 may appear on vertices in (\ref{eqn:3.10}), the contribution to the entropy of these unused vertices is at most

\begin{equation*}
\begin{array}{rl}
& \underset{n\rightarrow\infty}{\lim}\left[\left( \frac{k}{2\left(2^{k}-1\right)}\left(2^{n}-2^{n-k(m^{*}+1)}\right)+2^{n-k(m^{*}+1)} \right)\log2\right]/2^{n} \\ & \\= & \frac{k}{2(2^{k}-1)}\log 2.
\end{array}
\end{equation*}
The upper estimate (\ref{eqn:3.11}) of entropy follows.

\end{proof}

Lemmas 3.5 and 3.6 yield the following result.

\begin{theorem}
\label{theorem:3.4}
The entropy $h(\mathbb{X}_{2}^{A})$ is given by

\begin{equation}\label{eqn:3.12}
h(\mathbb{X}_{2}^{A})= \underset{k\rightarrow\infty}{\lim} \frac{1}{2\left(2^{k}-1\right)}\log|\Sigma_{k}|.
\end{equation}
Furthermore,

\begin{equation}\label{eqn:3.13}
\frac{1}{2\left(2^{k}-1\right)}\log|\Sigma_{k}|\leq h(\mathbb{X}_{2}^{A})\leq \frac{1}{2\left(2^{k}-1\right)}\log|\Sigma_{k}|+\frac{k}{2\left(2^{k}-1\right)}\log 2.
\end{equation}

\end{theorem}

\begin{proof}
Lemmas 3.5 and 3.6 imply (\ref{eqn:3.13}), and then (\ref{eqn:3.12}) follows from (\ref{eqn:3.13}) immediately.
\end{proof}


\begin{example}
\label{example:3.4-1}
Consider the one-dimensional couple shifts $\mathbb{X}_{2}^{G}\equiv \mathbb{X}_{2}\bigcap \Sigma_{G}$ where

\begin{equation*}
G=\left[
\begin{array}{cc}
1 & 1 \\
1 & 0
\end{array}
\right]
\end{equation*}
and $\Sigma_{G}$ is the golden mean shift.

Table 3.1 presents a numerical approximation of (\ref{eqn:3.13}). For $n\geq 1$, let

\begin{equation*}
\begin{array}{ccc}
h^{(n)}(\mathbb{X}_{2}^{G})=\frac{1}{2\left(2^{n}-1\right)}\log|\Sigma_{n}| & \text{and} & \bar{h}^{(n)}(\mathbb{X}_{2}^{G})=h^{(n)}(\mathbb{X}_{2}^{G})+\frac{n}{2\left(2^{n}-1\right)}\log 2.
\end{array}
\end{equation*}

\begin{equation*}
\begin{tabular}{|c|c|c|c|}
\hline
 $n$  &2 & 3& 4
\\
\hline
 $|\Sigma_{n}|$   & 9 & 237 & 213624\\
\hline
$h^{(n)}(\mathbb{X}_{2}^{G})$   & 0.366204 & 0.390576  & 0.409066   \\
\hline
$\bar{h}^{(n)}(\mathbb{X}_{2}^{G})$   & 0.597253& 0.539107& 0.501485 \\
\hline
\end{tabular}
\end{equation*}

\begin{equation*}
\text{Table 3.1.}
\end{equation*}

\end{example}

\begin{remark}
\label{remark:3.5}
Whether or not $h(\mathbb{X}_{2}^{A})$ can be expressed in explicit form, as in (\ref{eqn:1.2}) for $h(\mathbb{X}_{2}^{0})$ is not clear.
\end{remark}

Now, the result for $\mathbb{X}_{2}^{A}$ is ready to extend to

\begin{equation}\label{eqn:3.14}
\mathbb{X}_{Q}^{A}=\left\{
(x_{1},x_{2},x_{3},\cdots )\in\{0,1\}^{\mathbb{N}} \mid x_{k}x_{Qk}=0 \text{ for all } k\geq 1 \text{ and }(x_{1},x_{2},\cdots )\in\Sigma_{A}
 \right\}
\end{equation}
for any integer $Q\geq 3$.


From the study of $\mathbb{X}_{2}^{A}$, the main steps for $\mathbb{X}_{Q}^{A}$ are:

\begin{enumerate}

\item[(I)$_{c}$] Identify the lattice $M_{k}$ ( and $L_{k}$), which is the maximal connected graph of degree $k$. All $M_{k}(l)$, $l\in\mathcal{I}_{Q}$, are disjoint.

\item[(II)$_{c}$]  Compute the unused vertices in

\begin{equation}\label{eqn:3.15}
\mathcal{N}(Q^{n})\setminus \underset{l\in\mathcal{I}_{Q}}{\bigcup}M_{k}(l).
\end{equation}

\item[(III)$_{c}$] Compute the number $\Sigma_{k}$ of admissible patterns on $L_{k}$.
\end{enumerate}

Step (I)$_{c}$ gives lower bound of $h(\mathbb{X}_{Q}^{A})$, and (II)$_{c}$ gives upper bound of $h(\mathbb{X}_{Q}^{A})$. Then, $h(\mathbb{X}_{Q}^{A})$ follows if the error term in (II)$_{c}$ approaches zero as $n$ tends to infinity.
To minimize the error in (II)$_{c}$, the amount of unused lattices should be as small as possible. Therefore, the choice of $M_{k}$ or the graph $L_{k}$ in (I)$_{c}$ should be as large as possible as far as they are decoupled.

$\mathbb{X}_{3}^{A}$ can be used to illustrate the procedures.
From

\begin{equation}\label{eqn:3.16}
\mathbb{N}=\underset{i\in\mathcal{I}_{3}}{\bigcup}i \mathbb{M}_{3},
\end{equation}
the following Fig. 3.6 is drawn, which is corresponding to Fig. 3.1 for $\mathbb{X}_{2}^{A}$.

\begin{equation*}
\includegraphics[scale=0.5]{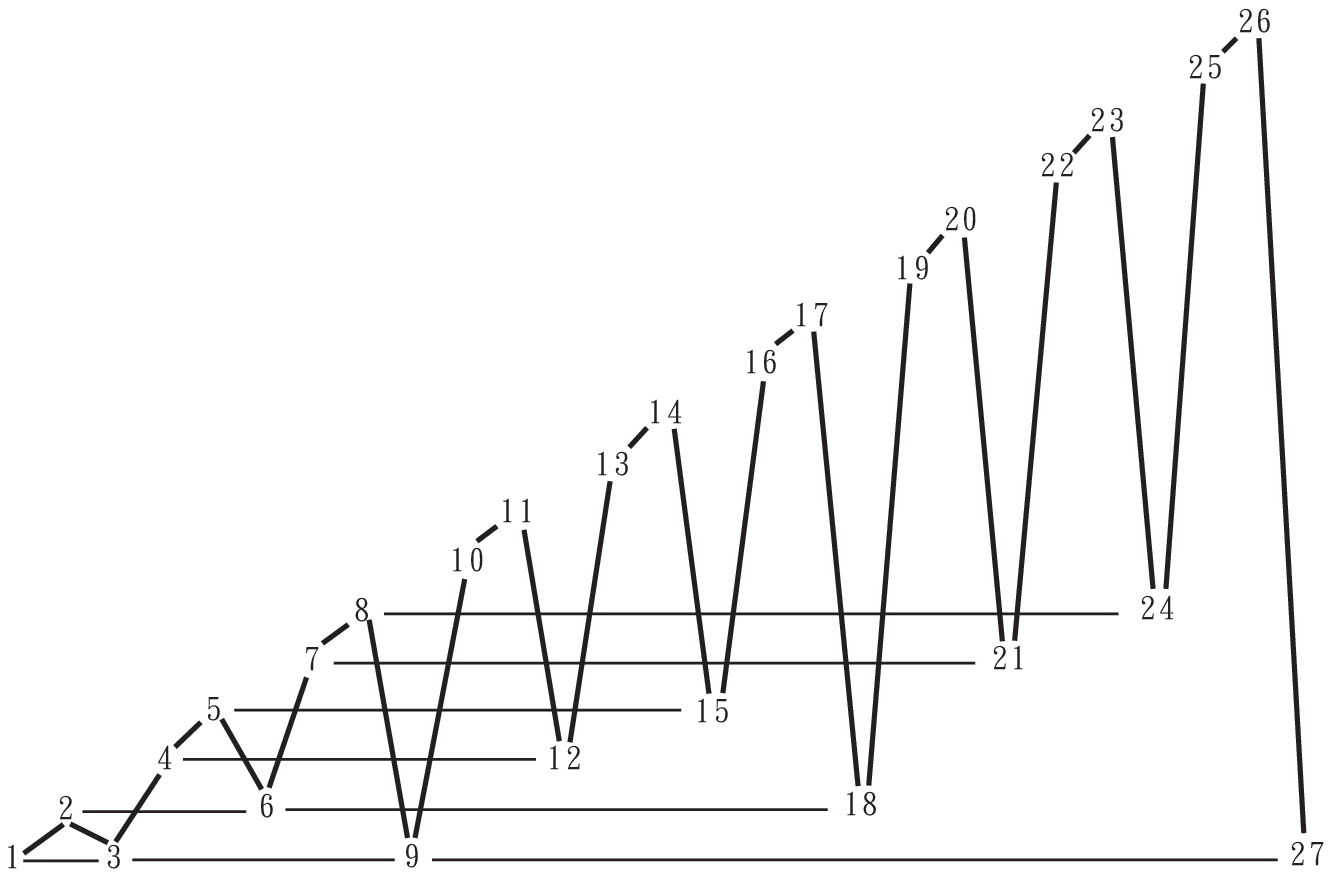}
\end{equation*}
\begin{equation*}
\text{Figure 3.6.}
\end{equation*}

Take $M_{3;1}$ as being on one vertex. In Fig. 3.7, a maximal connected graph $M_{3;2}$ with two horizontal vertices in $\underset{1<i\in\mathcal{I}_{3}}{\bigcup}i\mathbb{M}_{3}$ can be identified as follows. Notably, $M_{3;2}(4)$ and $M_{3;2}(7)$ are mutually independent.

\begin{equation*}
\begin{array}{ccccc}
\psfrag{a}{{\footnotesize $M_{3;2}(4)$}}
\includegraphics[scale=0.5]{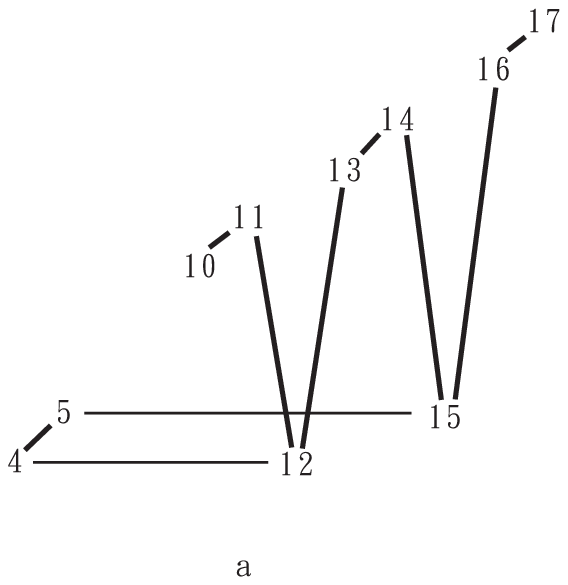} & &  &&
\psfrag{a}{{\footnotesize $M_{3;2}(7)$}}
\includegraphics[scale=0.5]{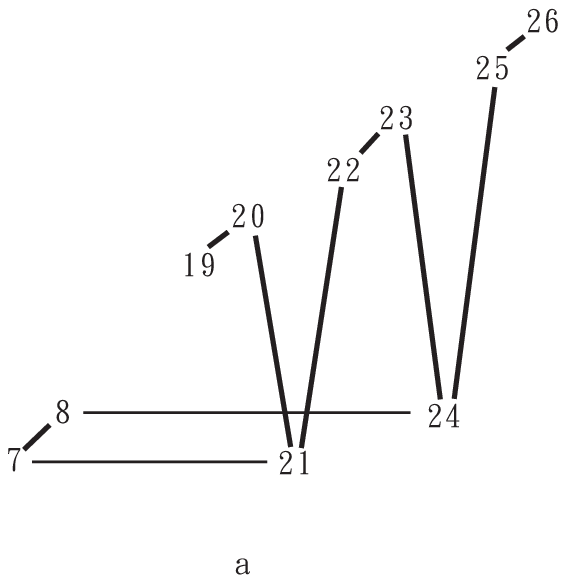}
\end{array}
\end{equation*}
\begin{equation*}
\text{Figure 3.7.}
\end{equation*}

Furthermore, $M_{3;3}(4)$ can be constructed from $M_{3;2}(4)$ as follows.

\begin{equation*}
\psfrag{a}{{\footnotesize $M_{3;3}(4)$}}
\includegraphics[scale=0.5]{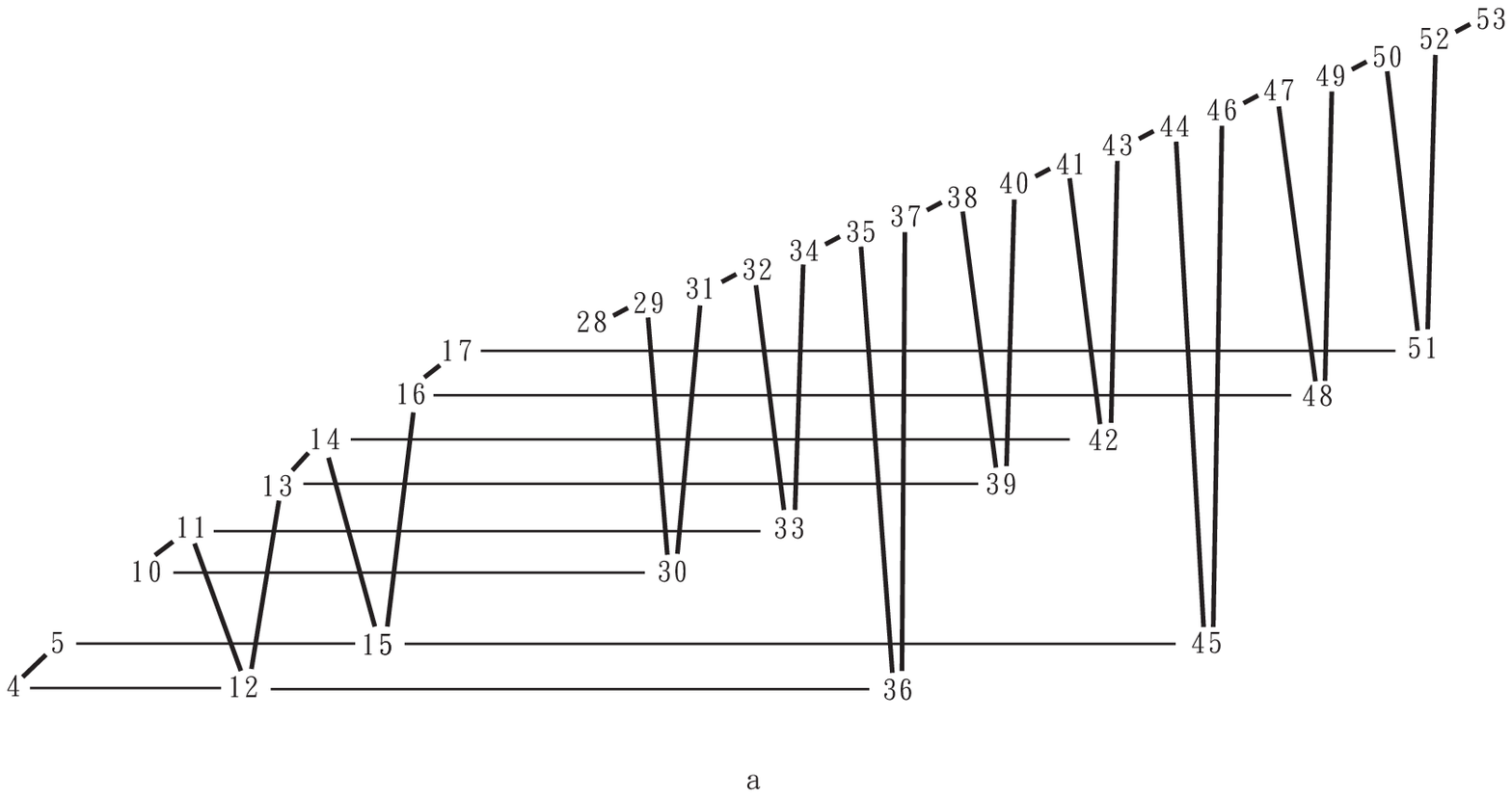}
\end{equation*}
\begin{equation*}
\text{Figure 3.8.}
\end{equation*}

Similar to Definition 3.3 for $Q=2$, for any $Q\geq 3$, $M_{Q;k}(l)$ can be defined as follows.

\begin{definition}
\label{Definition:3.6-10}
For any $Q\geq 3$ and $l\in \mathcal{I}_{Q}$, let $V_{Q;1}(l)=\{l,l+1,\cdots, l+Q-2\}$, and for $k\geq 1$, define

\begin{equation*}
V_{Q;k+1}(l)=\left\{ Qi, Qi\pm 1, Qi\pm 2, \cdots,Qi\pm (Q-1)   \hspace{0.2cm}\mid \hspace{0.2cm} i\in V_{Q;k}(l) \right\}.
\end{equation*}
Then, let $M_{Q;1}(l)=V_{Q;1}(l)$ and for $k\geq 2$, define

\begin{equation*}
M_{Q;k+1}(l)\equiv M_{Q;k}(l) \cup V_{Q;k+1}(l).
\end{equation*}

\end{definition}

For any $Q\geq 3$ and $k\geq 1$, denote $L_{Q;k}$ be the degree $k$ blank lattice. The following lemma gives the number of the vertices of $L_{Q;k}$, an extension of Lemma 3.6.

\begin{lemma}
\label{lemma:3.6}
For any $Q\geq 3$ and $k\geq 2$, the number $|L_{Q;k}|$ of the vertices of $L_{Q;k}$ is

\begin{equation}\label{eqn:3.17}
|L_{Q;k}|=\frac{Q\left(Q^{k}-1\right)}{Q-1}-k.
\end{equation}
\end{lemma}

\begin{proof}
First, (\ref{eqn:3.17}) is proven for the case $Q=3$. The other cases can be treated analogously.

For $Q=3$, let

\begin{equation*}
a_{3,n}=2\cdot 3^{n}.
\end{equation*}

The blank lattice $L_{3;2}$ can be obtained from $M_{3;2}(4)$ in Fig. 3.7, and

\begin{equation*}
|L_{3;2}|=a_{3,1}+(a_{3,1}+a_{3,2})=10.
\end{equation*}
Now, $L_{3;3}$ is obtained from $M_{3;3}(4)$ in Fig. 3.8 and can be grouped as follows.

\begin{equation*}
\includegraphics[scale=0.3]{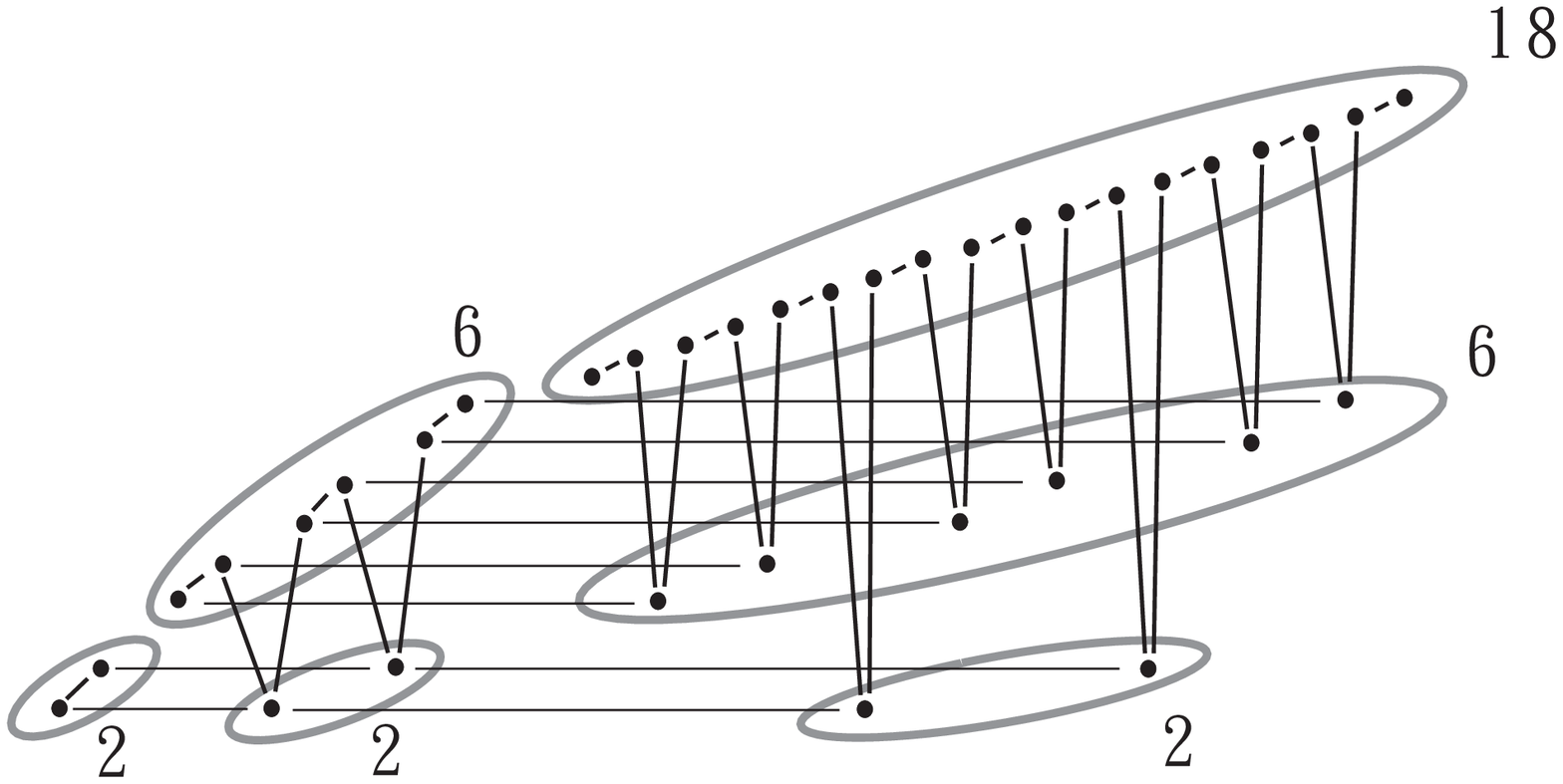}
\end{equation*}
\begin{equation*}
\text{Figure 3.9.}
\end{equation*}
According to Fig. 3.9, it is easy to see that

\begin{equation*}
\begin{array}{rl}
|L_{3;3}|= & \underset{l=1}{\overset{3}{\sum}}\underset{j=1}{\overset{l}{\sum}}a_{3,j} \\
 =& a_{3,1}+(a_{3,1}+a_{3,2})+(a_{3,1}+a_{3,2}+a_{3,3})\\
 & \\
 = & |L_{2;2}|+26.
\end{array}
\end{equation*}
By induction, it can be proven

\begin{equation}\label{eqn:3.18}
|L_{3;m}|=|L_{3;m-1}|+3^{m}-1.
\end{equation}
Therefore,
(\ref{eqn:3.17}) follows for $Q=3$.

For $Q=4$, $M_{4;2}(5)$ is as in Fig. 3.10.

\begin{equation*}
\psfrag{a}{{\footnotesize $M_{4;2}(5)$}}
\includegraphics[scale=0.5]{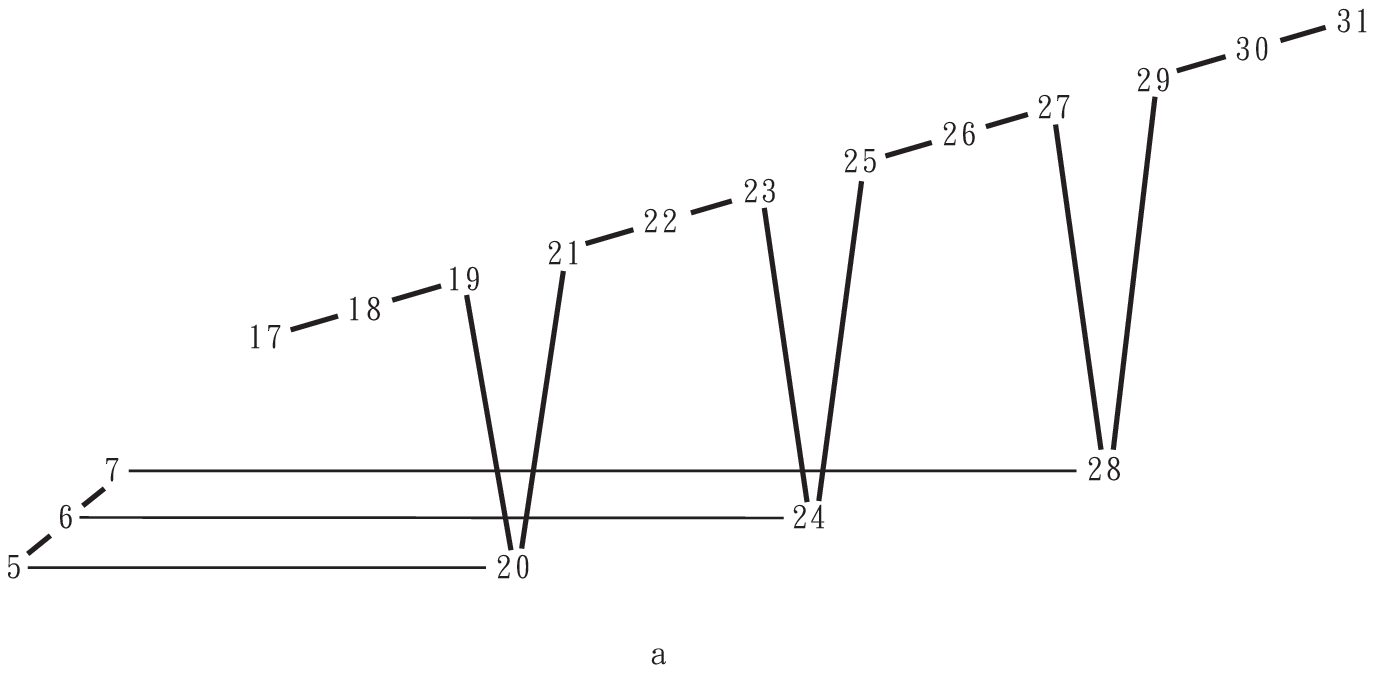}
\end{equation*}
\begin{equation*}
\text{Figure 3.10.}
\end{equation*}

Similarly, for any $Q\geq 4$, $M_{Q;2}(l)$ is given in Fig. 3.11.

\begin{equation*}
\psfrag{a}{}
\psfrag{1}{{\tiny $l$}}
\psfrag{2}{{\tiny $l+1$}}
\psfrag{3}{{\tiny $l+2$}}
\psfrag{4}{{\tiny $l+Q-2$}}
\psfrag{5}{{\tiny $Ql$}}
\psfrag{6}{{\tiny $Q(l+1)$}}
\psfrag{7}{{\tiny $Q(l+2)$}}
\psfrag{8}{{\tiny $Q(l+Q-2)$}}
\psfrag{b}{{\tiny $c_{1}$}}
\psfrag{c}{{\tiny $c_{2}$}}
\psfrag{d}{{\tiny $c_{3}$}}
\psfrag{e}{{\tiny $c_{4}$}}
\psfrag{f}{{\tiny $c_{5}$}}
\psfrag{g}{{\tiny $c_{6}$}}
\psfrag{h}{{\tiny $c_{7}$}}
\psfrag{j}{{\tiny $c_{8}$}}
\psfrag{k}{{\tiny $c_{9}$}}
\psfrag{l}{{\tiny $c_{10}$}}
\includegraphics[scale=0.7]{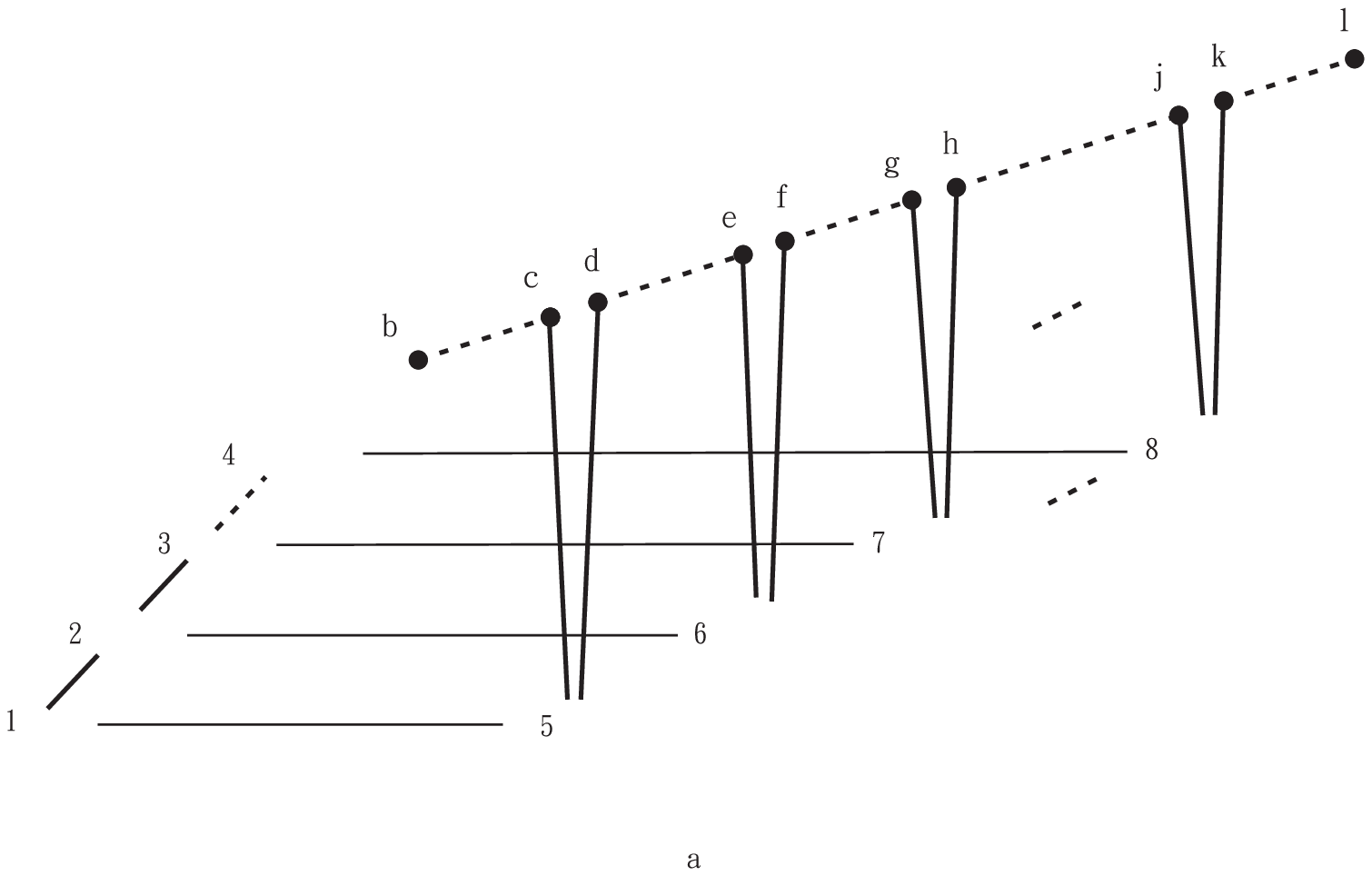}
\end{equation*}
{\footnotesize
\begin{equation*}
\begin{array}{lllll}
c_{1}= Q(l-1)+1,&  c_{2}=Ql-1, &  c_{3}= Ql+1, &  c_{4}= Q(l+1)-1, & c_{5}= Q(l+1)+1, \\
c_{6}= Q(l+2)-1, &  c_{7}= Q(l+2)+1, &  c_{8}= Q(l+Q-2)-1, &  c_{9}= Q(l+Q-2)+1, & c_{10}= Q(l+Q-1)-1 \\
\end{array}
\end{equation*}
}
\begin{equation*}
M_{Q;2}(l)
\end{equation*}
\begin{equation*}
\text{Figure 3.11. }
\end{equation*}

Define

\begin{equation*}
a_{Q,n}=(Q-1)Q^{n-1}.
\end{equation*}
Careful mathematical induction proves (\ref{eqn:3.17}). The details are omitted here.
\end{proof}

Now, we can extend Theorem 3.7 to all $Q\geq 3$, as follows.

\begin{theorem}
\label{theorem:3.7}
For any $Q\geq 3$ and $k\geq 2$,

\begin{equation}\label{eqn:3.19}
\frac{Q-1}{Q\left(Q^{k}-1\right)}\log|\Sigma_{Q;k}|\leq h(\mathbb{X}_{Q}^{A})\leq\frac{Q-1}{Q\left(Q^{k}-1\right)}\left(\log|\Sigma_{Q;k}|+k\log 2\right),
\end{equation}
and

\begin{equation}\label{eqn:3.20}
h(\mathbb{X}_{Q}^{A}) = \underset{k\rightarrow \infty}{\lim}\frac{Q-1}{Q\left(Q^{k}-1\right)}\log|\Sigma_{Q;k}|,
\end{equation}
where $\Sigma_{Q;k}$ is the set of all admissible patterns on $L_{Q;k}$.
\end{theorem}

\begin{proof}
 Given $k\geq 2$ and $n\geq 1$, $m^{*}$ is defined in (\ref{eqn:3.5-50}). As the proof of Lemma 3.4, it can be verified that for $0\leq l\leq m^{*}$, within $\left(Q^{n-k(l+1)},Q^{n-kl}\right)$, the number of mutually independent copies of $M_{Q;k}$ is $(Q-1)Q^{n-k(l+1)-1}$.

Hence, within $\left(1,Q^{n}\right)$, there are

 \begin{equation}\label{eqn:3.20-10}
\alpha_{Q;k}(n)\equiv (Q-1)Q^{n-k-1}\left(1+\frac{1}{Q^{k}}+\frac{1}{Q^{2k}}+\cdots +\frac{1}{Q^{m^{*}k}}\right)
 \end{equation}
disjoint copies of $M_{Q;k}(l)$ with $l\in\mathcal{I}_{Q}$.

As the proof of Lemma 3.6, by (\ref{eqn:3.17}), the number of the remaining vertices that are not used in $[1,Q^{n}]$ is

\begin{equation*}
\frac{(Q-1)k}{Q\left(Q^{k}-1\right)}\left(Q^{n}-Q^{n-k(m^{*}+1)}\right)+Q^{n-k(m^{*}+1)}.
\end{equation*}
Therefore, the results follow immediately.
\end{proof}

Like Theorem 2.10, Theorem 3.12 can be generalized to any number of symbols, any constraints $\mathcal{C}$ and any additive shifts of finite type $\Sigma_{A}$. Indeed, let

\begin{equation}\label{eqn:3.22}
\begin{array}{rl}
 & \mathbb{X}_{Q}^{A}(N,\mathcal{C}) \\ & \\
 = &\left\{
(x_{1},x_{2},x_{3},\cdots )\in\{0,1,\cdots,N-1\}^{\mathbb{N}} \mid x_{k}x_{Qk}\in \mathcal{C} \text{ for all } k\geq 1 \text{ and } (x_{1},x_{2},x_{3},\cdots )\in\Sigma_{A}
 \right\},
 \end{array}
\end{equation}
where $\mathcal{C}\subseteq \left\{0,1,\cdots,(N-1)^{d}\right\}$ is a constraint set and $A$ is an $m\times m$ \hspace{0.1cm}$0-1$ matrix.

Then, the following theorem can easily be obtained. The details of the proof are omitted.
\begin{theorem}
\label{theorem:3.7-10}
For any $Q\geq 2$, $\mathcal{C}\subseteq \left\{0,1,\cdots,(N-1)^{d}\right\}$, $N\geq 2$, $d\geq 1$ and $k\geq 2$,

\begin{equation}\label{eqn:3.23}
\frac{Q-1}{Q\left(Q^{k}-1\right)}\log|\Sigma_{k}(Q;A;N,\mathcal{C})|\leq h\left(\mathbb{X}_{Q}^{A}(N,\mathcal{C})\right)\leq\frac{Q-1}{Q\left(Q^{k}-1\right)}\left(\log|\Sigma_{k}(Q;A;N,\mathcal{C})|+k\log N\right),
\end{equation}
and

\begin{equation}\label{eqn:3.24}
h\left(\mathbb{X}_{Q}^{A}(N,\mathcal{C})\right) = \underset{k\rightarrow \infty}{\lim}\frac{Q-1}{Q\left(Q^{k}-1\right)}\log|\Sigma_{k}(Q;A;N,\mathcal{C})|,
\end{equation}
where $\Sigma_{k}(Q;A;N,\mathcal{C})$ is the set of all admissible patterns on $L_{Q;k}$, the constraint of the vertices on the bold lines in $L_{Q;k}$ is given by $A$ and  the constraint of the vertices on thin lines in $L_{Q;k}$ is given by $N$ and $\mathcal{C}$.
\end{theorem}

\section{Multi-dimensional coupled systems}

\hspace{0.4cm}This section discusses the multi-dimensional coupled systems and points out the difficulties when applying the method that works well for one-dimensional coupled
systems.

The multi-dimensional coupled system is
\begin{equation}\label{eqn:4.1}
\mathbb{X}_{\Gamma}^{A}=  \mathbb{X}_{\Gamma}^{0}\bigcap \Sigma_{A},
\end{equation}
where $\mathbb{X}_{\Gamma}^{0}$ is a multiplicative integer system and $\Sigma_{A}$ is a shift of finite type. For clarity, only $\mathbb{X}_{2,3}^{A}$ is considered.

Recall the strategy for studying one-dimensional coupled system.
\begin{enumerate}
\item[(I)] For $k\geq 1$, choose a suitable admissible numbered lattice $M_{k}(l)$ as basic elements for $\mathbb{X}_{Q}^{A}$; see Fig. 3.2 for $\mathbb{X}_{2}^{A}$.

\item[(II)] For each $k\geq 1$, split the natural numbers $\mathbb{N}$ into two parts:

\begin{equation}\label{eqn:4.2}
\mathbb{N}=U_{k}\cup W_{k},
\end{equation}
where $U_{k}$ is used to select the mutually independent admissible numbered lattice $M_{k}(l)$ for approximating the entropy, and $W_{k}$ is the set of the cells that is removed from $\mathbb{N}$ to
achieve independence of $M_{k}(l)$ in $U_{k}$. Good splittings require

\begin{equation}\label{eqn:4.3}
\underset{k\rightarrow\infty}{\lim}\underset{n\rightarrow\infty}{\lim}\frac{\left|W_{k}\cap [1,n]\right|}{\left|U_{k}\cap [1,n]\right|}=0;
\end{equation}
see Lemmas 3.4 and 3.6 and Theorems 3.7 and 3.12.

\item[(III)] Finally, compute the number of members of the set

\begin{equation}\label{eqn:4.4}
\mathcal{M}_{k}=\left\{ M_{k}(l)\subset U_{k} \hspace{0.1cm} \mid \hspace{0.1cm}  l \in \mathcal{I}_{\Gamma} \right\}.
\end{equation}
Let

\begin{equation*}
\alpha_{k}(n)=\left|\mathcal{M}_{k} \cap [1,n] \right|,
\end{equation*}
and

\begin{equation*}
\alpha_{k}^{*}=\underset{n\rightarrow \infty }{\lim}\frac{\alpha_{k}(n)}{n}.
\end{equation*}
Then, the number $\alpha_{k}^{*} \log |\Sigma_{k}|$ is an approximation to the entropy of $\mathbb{X}_{Q}^{A}$, where $\Sigma_{k}$ is the set of all admissible patterns on $M_{k}$. For $\mathbb{X}_{2}^{A}$, see Lemma 3.4.

\end{enumerate}

For multi-dimensional cases, firstly observe topological effect of dimensionality. Consider the model system

\begin{equation}\label{eqn:4.5}
\mathbb{X}_{2,3}^{A}=\mathbb{X}_{2,3}^{0}\cap \Sigma_{A},
\end{equation}
where $\mathbb{X}_{2,3}^{0}$ is given by (\ref{eqn:2.1}) and $\Sigma_{A}$ is a shift of finite type. $M_{k}$ is the L-shaped $k$-cell numbered lattices given in Table 2.1 and Fig. 2.1. $\mathcal{I}_{2,3}=\left\{6k+1,6k+5\right\}_{k=0}^{\infty}$ is given in (\ref{eqn:2.4}), and then $\mathbb{N}= \underset{i\in \mathcal{I}_{2,3}}{\bigcup}i\mathbb{M}_{2,3}$ given in (\ref{eqn:2.5}) is obtained.

Recall $iM_{k}$, $k\in\{1,2,3,4\}$, as follows.

\begin{equation*}
\begin{array}{c}
\psfrag{a}{$iM_{1}$ }
\psfrag{b}{$iM_{2}$ }
\psfrag{c}{$iM_{3}$ }
\psfrag{d}{$iM_{4}$ }
\psfrag{e}{$i$ }
\psfrag{f}{$2i$ }
\psfrag{g}{$3i$ }
\psfrag{h}{$4i$ }
\includegraphics[scale=1.0]{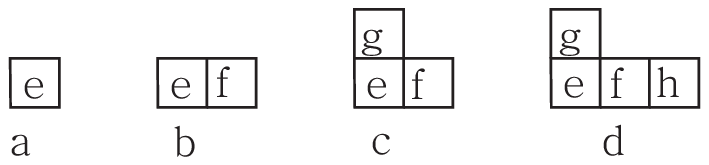}
\end{array}
\end{equation*}
\begin{equation*}
\text{Figure 4.1.}
\end{equation*}

Two sets $M$ and $M'$ of natural integers are called consecutive or connected if there exist $m\in M$ and $m'\in M'$ such that

\begin{equation}\label{eqn:4.6}
|m-m'|=1.
\end{equation}
First, we prove the consecutive results for $iM_{k}$, $i\in\mathcal{I}_{2,3}$ and $1\leq k\leq4$.

\begin{lemma}
\label{Lemma:4.1}
For any $l\geq 0$, the following pairs are consecutive:

\begin{equation*}
\begin{array}{l}
(1) \hspace{0.2cm} \text{(i)} \hspace{0.1cm}(6l+1)M_{2}  \text{ and }  (12l+1)M_{1}, \hspace{0.3cm} \text{(ii)} \hspace{0.1cm}(6l+5)M_{2}  \text{ and }  (12l+11)M_{1}, \\
\\
(2) \hspace{0.2cm} \text{(i)} \hspace{0.1cm}(24l+5)M_{3} \text{ and }  (36l+7)M_{2}, \hspace{0.3cm} \text{(ii)} \hspace{0.1cm}(24l+11)M_{3}  \text{ and }  (36l+17)M_{2} \\
 \\
 \hspace{0.5cm} \text{(iii)} \hspace{0.1cm}(24l+13)M_{3}  \text{ and }  (36l+19)M_{2},  \hspace{0.3cm} \text{(iv)} \hspace{0.1cm}(24l+19)M_{3}  \text{ and }  (36l+29)M_{2},\\
 \\
 (3) \hspace{0.2cm}\text{(i)} \hspace{0.1cm}(18l+1)M_{4}  \text{ and }   (72l+5)M_{1},  \hspace{0.3cm}\text{(ii)} \hspace{0.1cm}(18l+17)M_{4}  \text{ and }  (72l+67)M_{1}, \\
 \\
 (4) \hspace{0.2cm} \text{(i)} \hspace{0.1cm}(54l+23)M_{4}  \text{ and } (72l+31)M_{3},  \hspace{0.3cm}\text{(ii)} \hspace{0.1cm}(54l+31)M_{4}  \text{ and }  (72l+41)M_{3}.\\
\end{array}
\end{equation*}
%
%
%
%
%
%

\end{lemma}

\begin{proof}
That all pairs $iM_{k}$ and $jM_{l}$ satisfy

\begin{equation}\label{eqn:4.7}
|ik-jl|=1
\end{equation}
can be straightforwardly verified; consecutiveness follows.
\end{proof}

The consequence of Lemma 4.1 is that $\left\{iM_{4}\right\}_{i\in \mathcal{I}_{2,3}}$ are tied closely. The following proposition is asserted.

\begin{proposition}
\label{Proposition:4.2}
For each $i\in\mathcal{I}_{2,3}$, a decreasing finite sequence $\{i_{k}\}_{k=0}^{N}\subset \mathcal{I}_{2,3}$ with $i_{0}=i$ and $i_{N}=1$ exists such that $i_{k}M_{4}$ and $i_{k+1}M_{4}$ are consecutive
for each $0\leq k\leq N$.

\end{proposition}

\begin{proof}
It is easy to verified that
\begin{equation}\label{eqn:4.8}
\begin{array}{rl}
\mathcal{I}_{2,3}= & \left\{6k+1,6k+5\right\}_{k=0}^{\infty} \\
& \\
=& \left\{12m+1,12m+5,12m+7,12m+11\right\}_{m=0}^{\infty}.
\end{array}
\end{equation}
Now,

\begin{equation}\label{eqn:4.9}
\begin{array}{rl}
 &\left\{12m+5\right\}_{m=0}^{\infty}  \\
& \\
=& \left\{72l+5,72l+17,72l+29,72l+41,72l+53,72l+65\right\}_{l=0}^{\infty}.
\end{array}
\end{equation}
and
\begin{equation}\label{eqn:4.10}
\begin{array}{rl}
 &\left\{12m+7\right\}_{m=0}^{\infty}  \\
& \\
=& \left\{72l+7,72l+19,72l+31,72l+43,72l+55,72l+67\right\}_{l=0}^{\infty}.
\end{array}
\end{equation}
Firstly, we show that for any $i\in\mathcal{I}_{2,3}$, $iM_{4}$ is consecutive to $i_{1}M_{_{4}}$ for some $i_{1}\in\mathcal{I}_{2,3}$ with $i_{1}< i$ by taking one of the pairs which appear in (1)$\sim$(4) in Lemma 4.1.

Indeed, if $i=12m+1$ or $i=12m+11$, then $iM_{1}$ is consecutive to $jM_{2}$ by choosing $j=6m+1$ or $j=6m+5$ in (1) (i) and (ii), respectively. If $i=12m+5$, then (4.9) and Lemma 4.1 implies $(72l+5)M_{1}$, $(72l+17)M_{2}$, $(72l+53)M_{2}$, $(72l+29)M_{2}$, $(72l+65)M_{2}$ and $(72l+41)M_{3}$ can be consecutive to $jM_{q}$ which appears in (2), (3) and (4) of Lemma 4.1. Similar results hold for $i=12m+7$ by using (4.10) and the results in (2), (3) and (4) of Lemma 4.1.

By induction, there is a finite sequence $i_{2}> i_{3}> \cdots > i_{N}=1$ such tat $i_{k}\in \mathcal{I}_{2,3}$ with $i_{k}M_{4}$ and $i_{k+1}M_{4}$ are consecutive for each $4\leq k\leq N$. The proof is complete.

\end{proof}

The consecutive diagram of $\{iM_{4}\}$ for $i\in [1,19]\cap \mathcal{I}_{2,3}$ is drawn in Fig. 4.2 (i). Furthermore, the projection of consecutive diagram of $\{iM_{4}\}$ in $\mathcal{I}_{2,3}$  for $i\in [1,41]\cap\mathcal{I}_{2,3}$ is drawn in Fig. 4.2 (ii).

\begin{equation*}
\begin{array}{c}
\includegraphics[scale=0.6]{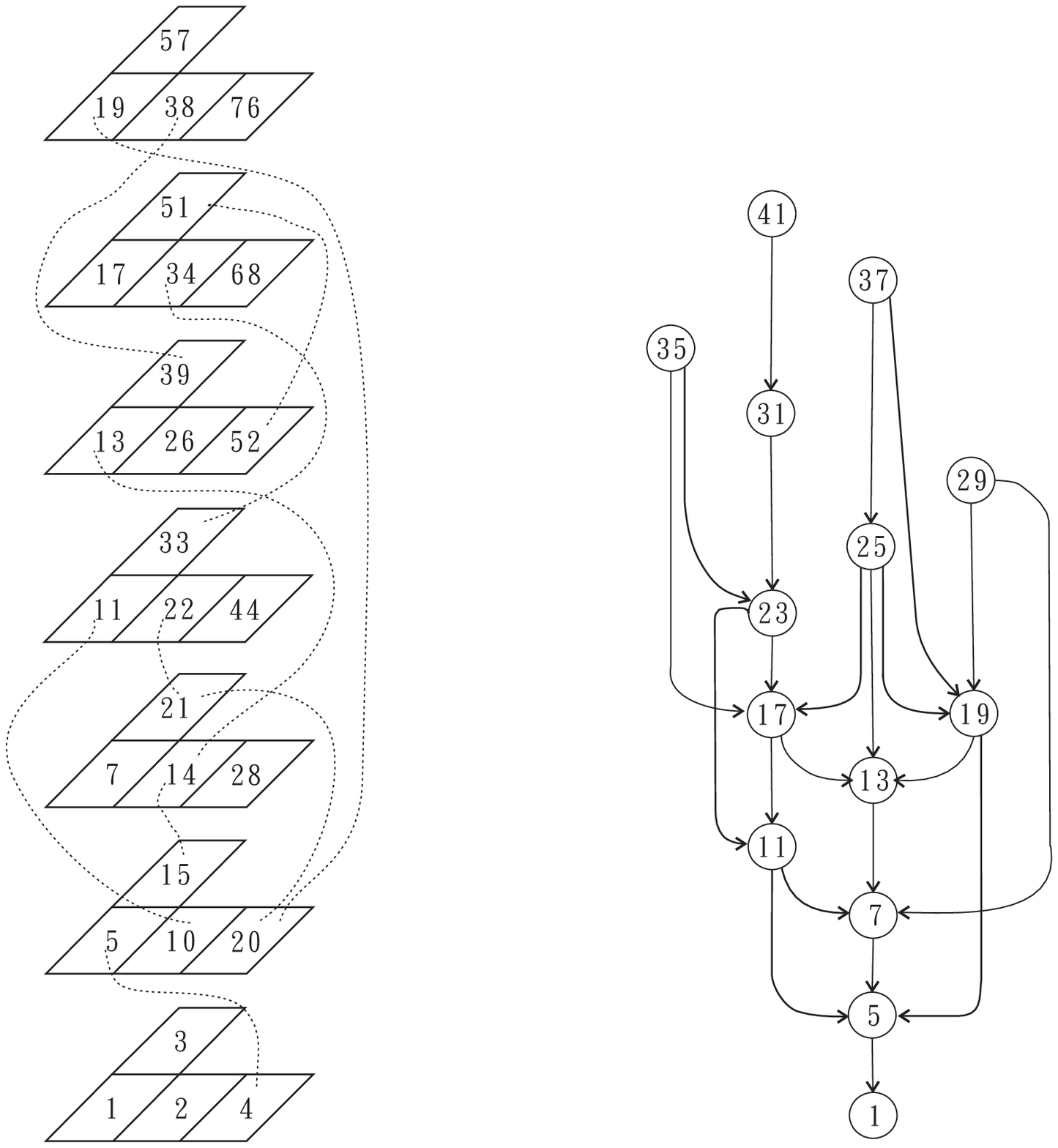}
\end{array}
\end{equation*}
\begin{equation*}
\hspace{-0.5cm}\text{Figure 4.2 (i).} \hspace{4.0cm}\text{Figure 4.2 (ii).}
\end{equation*}

In the following, the consecutiveness of $iM_{4}$, $i\in\mathcal{I}_{2,3}$, is considered further. Indeed, $iM_{4}$ looks like an octopus with eight arms.

\begin{equation*}
\begin{array}{c}
\psfrag{a}{$i$ }
\psfrag{b}{$2i$ }
\psfrag{c}{$4i$ }
\psfrag{d}{$3i$ }
\psfrag{e}{{\footnotesize$i-1$ }}
\psfrag{f}{{\footnotesize$i+1$ }}
\psfrag{g}{{\scriptsize$3i-1$ }}
\psfrag{h}{{\scriptsize$3i+1$ }}
\psfrag{k}{{\scriptsize$2i-1$ }}
\psfrag{l}{{\scriptsize$2i+1$ }}
\psfrag{m}{{\scriptsize$4i+1$ }}
\psfrag{n}{{\scriptsize$4i-1$ }}
\includegraphics[scale=0.8]{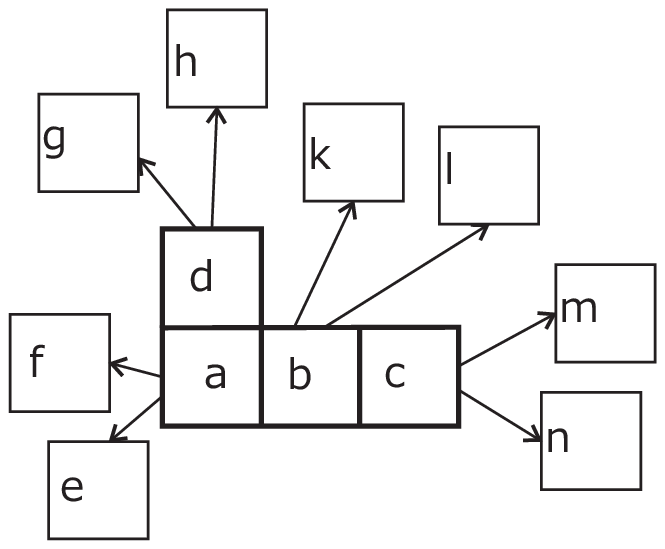}
\end{array}
\end{equation*}
\begin{equation*}
\text{Figure 4.3.}
\end{equation*}

Whether these arms can reach some $jM_{4}$ is of interest, as a determination would reveal which elements in $\{i\pm 1,2i\pm 1,3i\pm 1,4i\pm 1,\}$ belong to $l\mathcal{I}_{2,3}$ with $l\in \{1,2,3,4\}$.
The following proposition provides the answers.

\begin{proposition}
\label{Proposition:4.3}
For any $k\geq 0$,

\item[(A)] $i=6k+1$,
\begin{equation*}
\begin{array}{ll}
\text{(I)}\hspace{0.2cm} \text{(i)} \hspace{0.2cm} \text{if }k=2l, i+1\in 2\mathcal{I}_{2,3} & \text{(ii)} \hspace{0.2cm} \text{if }k=4l+3, i+1\in 4\mathcal{I}_{2,3}\\
& \\
\text{(II)}\hspace{0.2cm} \text{(i)}\hspace{0.2cm} \text{(a)}   \hspace{0.2cm} \text{if }k=3l, 2i+1\in 3\mathcal{I}_{2,3} & \text{(b)} \hspace{0.2cm} \text{if }k=3l+1, 2i+1\in 3\mathcal{I}_{2,3} \\
\hspace{0.7cm} \text{(ii)}\hspace{0.2cm} \text{(a)}   \hspace{0.2cm} \text{if }k=3l, 2i-1\in \mathcal{I}_{2,3} & \text{(b)} \hspace{0.2cm} \text{if }k=3l+1, 2i-1\in \mathcal{I}_{2,3} \\
& \\
\text{(III)}\hspace{0.2cm} \text{(i)} \hspace{0.2cm} \text{if }k=2l+1, 3i+1\in 2\mathcal{I}_{2,3} & \\
\hspace{0.9cm} \text{(ii)}\hspace{0.2cm} \text{(a)}   \hspace{0.2cm} \text{if }k=2l, 3i-1\in 2\mathcal{I}_{2,3} & \text{(b)} \hspace{0.2cm} \text{if }k=4l+1, 3i-1\in 4\mathcal{I}_{2,3} \\
& \\
\text{(IV)}\hspace{0.2cm} \text{(i)} \hspace{0.2cm} 4i+1\in \mathcal{I}_{2,3} & \\
\hspace{0.8cm} \text{(ii)}\hspace{0.2cm} \text{(a)}   \hspace{0.2cm} \text{if }k=3l, 4i-1\in 3\mathcal{I}_{2,3} & \text{(b)} \hspace{0.2cm} \text{if }k=3l+2, 4i-1\in 3\mathcal{I}_{2,3} \\
\end{array}
\end{equation*}

\item[(B)] $i=6k+5$,
\begin{equation*}
\begin{array}{ll}
\text{(I)}\hspace{0.2cm} \text{(i)} \hspace{0.2cm} \text{if }k=2l+1, i-1\in 2\mathcal{I}_{2,3} & \text{(ii)} \hspace{0.2cm} \text{if }k=4l, i-1\in 4\mathcal{I}_{2,3}\\
& \\
\text{(II)}\hspace{0.2cm} \text{(i)}\hspace{0.2cm}  2i+1\in \mathcal{I}_{2,3} &  \\
\hspace{0.7cm} \text{(ii)}\hspace{0.2cm} \text{(a)}   \hspace{0.2cm} \text{if }k=3l+1, 2i-1\in 3\mathcal{I}_{2,3} & \text{(b)} \hspace{0.2cm} \text{if }k=3l+2, 2i-1\in 3\mathcal{I}_{2,3} \\
& \\
\text{(III)}\hspace{0.2cm} \text{(i)}\hspace{0.2cm} \text{(a)}   \hspace{0.2cm} \text{if }k=4l+2, 3i+1\in 4\mathcal{I}_{2,3} & \text{(b)} \hspace{0.2cm} \text{if }k=2l+1, 3i+1\in 2\mathcal{I}_{2,3}  \\
\hspace{0.8cm} \text{(ii)}\hspace{0.2cm} \text{(a)}   \hspace{0.2cm} \text{if }k=2l, 3i-1\in 2\mathcal{I}_{2,3} & \text{(b)} \hspace{0.2cm} \text{if }k=4l+3, 3i-1\in 4\mathcal{I}_{2,3} \\
& \\
\text{(IV)}\hspace{0.2cm}\text{(i)}\hspace{0.2cm} \text{(a)}   \hspace{0.2cm} \text{if }k=3l, 4i+1\in 3\mathcal{I}_{2,3} & \text{(b)} \hspace{0.2cm} \text{if }k=3l+2, 4i+1\in 3\mathcal{I}_{2,3}\\
\hspace{0.8cm}  \text{(ii)} \hspace{0.2cm} 4i-1\in \mathcal{I}_{2,3}. & \\
\\
\end{array}
\end{equation*}

\end{proposition}

\begin{proof}
The proof is straightforward and omitted.

\end{proof}

From Proposition 4.3, there are 12 cases for each $i=6k+1$ and $i=6k+5$; see the following Table 4.1. Notably, in Tables 4.1, (i) and (ii) are symmetric in $lj\pm1$ of $12m+j$ and  $lj\mp1$ of $12n+(12-j-1)$.

For $i=6k+1$,
\begin{equation*}
\begin{tabular}{|l|c|c|c|c|c|c|c|c|c|c|}
\hline
$k$ &  $4i+1$ & $4i-1$ &  $3i+1$ & $3i-1$ &  $2i+1$ & $2i-1$ &  $i+1$ & $i-1$  & \text{up} & {\footnotesize\text{down}}  \\ \hline
$12m$ & $ \mathcal{I}_{2,3}$    & $3 \mathcal{I}_{2,3}$  &$4 \mathcal{I}_{2,3}$    &  $2 \mathcal{I}_{2,3}$  &  $3 \mathcal{I}_{2,3}$ & $ \mathcal{I}_{2,3}$  & $2 \mathcal{I}_{2,3}$&   & 4 & 3\\ \hline
$12m+1$ & $ \mathcal{I}_{2,3}$  &                         & $2 \mathcal{I}_{2,3}$  &  $4 \mathcal{I}_{2,3}$  &   $3 \mathcal{I}_{2,3}$    & $\mathcal{I}_{2,3}$    &                       & & 3 & 2 \\ \hline
$12m+2$ &$ \mathcal{I}_{2,3}$   & $3 \mathcal{I}_{2,3}$    &                        &  $2 \mathcal{I}_{2,3}$    &                        & $\mathcal{I}_{2,3}$    & $ 2 \mathcal{I}_{2,3}$                     &  & 4 & 1\\ \hline
$12m+3$ &$ \mathcal{I}_{2,3}$   & $3 \mathcal{I}_{2,3}$    & $2 \mathcal{I}_{2,3}$ &                         &    $3 \mathcal{I}_{2,3}$    & $\mathcal{I}_{2,3}$     & $ 4 \mathcal{I}_{2,3}$    &  & 4 &2 \\ \hline
$12m+4$ & $ \mathcal{I}_{2,3}$  &                           & $4 \mathcal{I}_{2,3}$   &  $2 \mathcal{I}_{2,3}$   &  $3 \mathcal{I}_{2,3}$     &$ \mathcal{I}_{2,3}$       & $2 \mathcal{I}_{2,3}$  &  & 3 & 3\\ \hline
$12m+5$ &$ \mathcal{I}_{2,3}$   &$3 \mathcal{I}_{2,3}$    & $2 \mathcal{I}_{2,3}$  &   $4 \mathcal{I}_{2,3}$   &                        &   $ \mathcal{I}_{2,3}$  &                       &  & 4 & 1\\ \hline
$12m+6$ & $ \mathcal{I}_{2,3}$  & $3 \mathcal{I}_{2,3}$   &                        &   $2 \mathcal{I}_{2,3}$   &  $3 \mathcal{I}_{2,3}$    &  $\mathcal{I}_{2,3}$     & $2 \mathcal{I}_{2,3}$    &  & 4 & 2\\ \hline
$12m+7$ & $ \mathcal{I}_{2,3}$  &                          & $ 2\mathcal{I}_{2,3}$   &                         &    $3 \mathcal{I}_{2,3}$  &   $\mathcal{I}_{2,3}$  &   $4 \mathcal{I}_{2,3}$   & & 3 & 2 \\ \hline
$12m+8$ &$ \mathcal{I}_{2,3}$   & $3 \mathcal{I}_{2,3}$   &  $4 \mathcal{I}_{2,3}$   &   $2 \mathcal{I}_{2,3}$   &                        &    $\mathcal{I}_{2,3}$   & $2 \mathcal{I}_{2,3}$  & & 4 & 2 \\ \hline
$12m+9$ &$ \mathcal{I}_{2,3}$   & $3 \mathcal{I}_{2,3}$   &  $2 \mathcal{I}_{2,3}$   &  $4 \mathcal{I}_{2,3}$  &    $3 \mathcal{I}_{2,3}$  &  $\mathcal{I}_{2,3}$  &                       & & 4 & 2\\ \hline
$12m+10$ & $ \mathcal{I}_{2,3}$  &                         &                        &    $2 \mathcal{I}_{2,3}$ &   $3 \mathcal{I}_{2,3}$   & $  \mathcal{I}_{2,3}$  &  $2 \mathcal{I}_{2,3}$  &  & 3 & 2\\ \hline
$12m+11$ &$ \mathcal{I}_{2,3}$ &$3 \mathcal{I}_{2,3}$     &  $2 \mathcal{I}_{2,3}$   &                         &                        &  $   \mathcal{I}_{2,3}$ &    $4 \mathcal{I}_{2,3}$  & & 4 &1 \\ \hline
\end{tabular}%
\end{equation*}

\begin{equation*}
\begin{array}{c}
\text{up: the number of upward arms } \\
\text{down: the number of downward arms}\\
\\
\text{Table 4.1 (i).}
\end{array}
\end{equation*}
and for $i=6k+5$,

\begin{equation*}
\begin{tabular}{|l|c|c|c|c|c|c|c|c|c|c|}
\hline
$k$ &  $4i+1$ & $4i-1$ &  $3i+1$ & $3i-1$ &  $2i+1$ & $2i-1$ &  $i+1$ & $i-1$ & \text{up}  & {\footnotesize\text{down}}  \\ \hline
$12n$ & $ 3\mathcal{I}_{2,3}$     & $ \mathcal{I}_{2,3}$    &                        &  $2 \mathcal{I}_{2,3}$  &  $ \mathcal{I}_{2,3}$ &                          &  & $4 \mathcal{I}_{2,3}$ & 4 & 1\\ \hline
$12n+1$ & $                  $    &  $ \mathcal{I}_{2,3}$   & $2 \mathcal{I}_{2,3}$  &                            &   $ \mathcal{I}_{2,3}$    & $3\mathcal{I}_{2,3}$ &   & $2 \mathcal{I}_{2,3}$  & 3& 2 \\ \hline
$12n+2$ &$ 3\mathcal{I}_{2,3}$    & $ \mathcal{I}_{2,3}$    & $4 \mathcal{I}_{2,3}$  &  $2 \mathcal{I}_{2,3}$    &   $ \mathcal{I}_{2,3}$ & $3\mathcal{I}_{2,3}$    & &                    & 4 &2 \\ \hline
$12n+3$ &$ 3\mathcal{I}_{2,3}$    & $ \mathcal{I}_{2,3}$    & $2 \mathcal{I}_{2,3}$  &  $4 \mathcal{I}_{2,3}$   &    $ \mathcal{I}_{2,3}$    &                       & & $ 2 \mathcal{I}_{2,3}$    & 4 & 2 \\ \hline
$12n+4$ & $                 $     &  $ \mathcal{I}_{2,3}$   &                        &  $2 \mathcal{I}_{2,3}$   &  $ \mathcal{I}_{2,3}$     &$ 3\mathcal{I}_{2,3}$    &   & $4 \mathcal{I}_{2,3}$  & 3 & 2 \\ \hline
$12n+5$ &$ 3\mathcal{I}_{2,3}$    &$ \mathcal{I}_{2,3}$     & $2 \mathcal{I}_{2,3}$  &                        &    $ \mathcal{I}_{2,3}$ &   $ 3\mathcal{I}_{2,3}$  &  & $2 \mathcal{I}_{2,3}$ & 4 &2 \\ \hline
$12n+6$ & $ 3\mathcal{I}_{2,3}$   & $ \mathcal{I}_{2,3}$    &  $4 \mathcal{I}_{2,3}$ &   $2 \mathcal{I}_{2,3}$   &  $ \mathcal{I}_{2,3}$    &                      &  &         & 4 &1 \\ \hline
$12n+7$ & $                   $   &   $ \mathcal{I}_{2,3}$  & $ 2\mathcal{I}_{2,3}$  &$4 \mathcal{I}_{2,3}$      &    $ \mathcal{I}_{2,3}$  &  3 $\mathcal{I}_{2,3}$  & &   $2\mathcal{I}_{2,3}$   & 3 & 3  \\ \hline
$12n+8$ &$ 3\mathcal{I}_{2,3}$    & $ \mathcal{I}_{2,3}$    &                        &   $2 \mathcal{I}_{2,3}$   &    $ \mathcal{I}_{2,3}$ &    3$\mathcal{I}_{2,3}$  & & $4 \mathcal{I}_{2,3}$  & 4 & 2  \\ \hline
$12n+9$ &$ 3\mathcal{I}_{2,3}$    & $ \mathcal{I}_{2,3}$    &  $2 \mathcal{I}_{2,3}$ &                             &    $ \mathcal{I}_{2,3}$  &                        &  & $2 \mathcal{I}_{2,3}$  & 4 & 1  \\ \hline
$12n+10$ & $                 $    &   $ \mathcal{I}_{2,3}$  & $4 \mathcal{I}_{2,3}$  & $2 \mathcal{I}_{2,3}$    &   $\mathcal{I}_{2,3}$   & $ 3 \mathcal{I}_{2,3}$  &  &               & 3 &2   \\\hline
$12n+11$ &$ 3\mathcal{I}_{2,3}$   &$ \mathcal{I}_{2,3}$     &  $2 \mathcal{I}_{2,3}$  &  $4 \mathcal{I}_{2,3}$    &  $ \mathcal{I}_{2,3}$  &  $  3 \mathcal{I}_{2,3}$ & &   $2 \mathcal{I}_{2,3}$ & 4 &3 \\ \hline
\end{tabular}%
\end{equation*}

\begin{equation*}
\text{Table 4.1 (ii).}
\end{equation*}

The following examples are five graphs that have the different numbers of upward arms and downward arms in Table 4.1.

\begin{example}
\label{Example:4.4}

\begin{equation*}
\begin{array}{lll}
& & \\
\begin{array}{c}
\psfrag{a}{{\footnotesize $293$}}
\psfrag{b}{{\footnotesize$145$}}
\psfrag{c}{{\footnotesize$109\cdot 2$}}
\psfrag{d}{{\footnotesize$97 \cdot 3$}}
\psfrag{e}{{\footnotesize $37\cdot 2$}}
\psfrag{f}{{\footnotesize$49\cdot 3$}}
\psfrag{g}{{\footnotesize$55\cdot 4$}}
\includegraphics[scale=0.8]{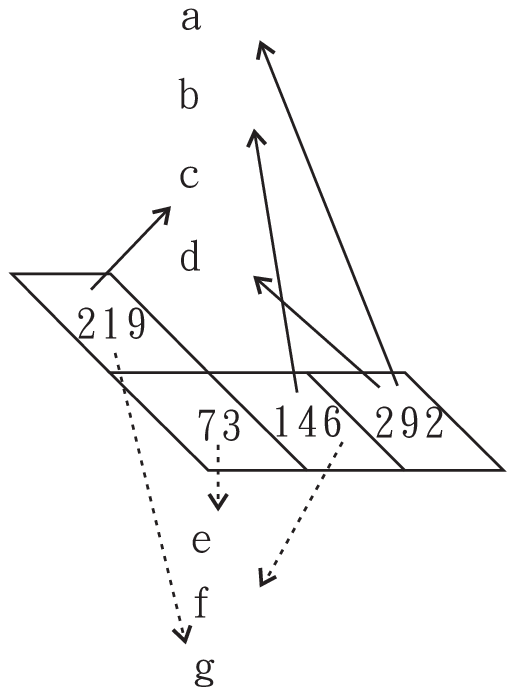}
\end{array}
&&
\begin{array}{c}
\psfrag{a}{{\footnotesize $77$}}
\psfrag{b}{{\footnotesize$37$}}
\psfrag{c}{{\footnotesize$29\cdot 2$}}
\psfrag{d}{{\footnotesize$25 \cdot 3$}}
\psfrag{e}{{\footnotesize $13\cdot 3$}}
\psfrag{f}{{\footnotesize$5\cdot 4$}}
\includegraphics[scale=0.8]{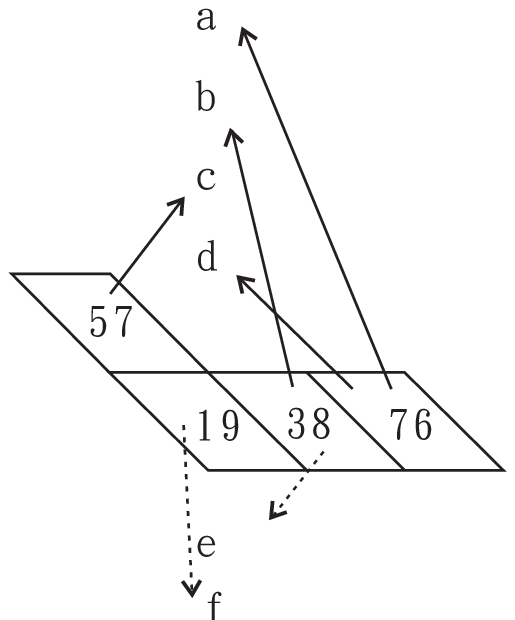}
\end{array} \\
& & \\
\text{(a) 4 upward arms and 3 downward arms} && \text{(b)  4 upward arms and 2 downward arms} \\
& & \\
\begin{array}{c}
\psfrag{a}{{\footnotesize $101$}}
\psfrag{b}{{\footnotesize$49$}}
\psfrag{c}{{\footnotesize$37\cdot 2$}}
\psfrag{d}{{\footnotesize$19 \cdot 4$}}
\psfrag{e}{{\footnotesize $17\cdot 3$}}
\psfrag{f}{{\footnotesize$13\cdot 2$}}
\includegraphics[scale=0.8]{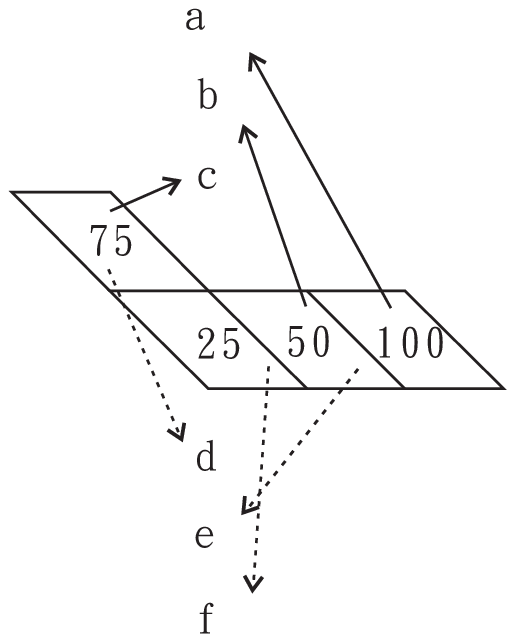}
\end{array}
&&
\begin{array}{c}
\psfrag{a}{{\footnotesize $19$}}
\psfrag{b}{{\footnotesize$11$}}
\psfrag{c}{{\scriptsize$7\cdot 2,7\cdot 3 $}}
\psfrag{d}{{\footnotesize$1 \cdot 4$}}
\includegraphics[scale=0.8]{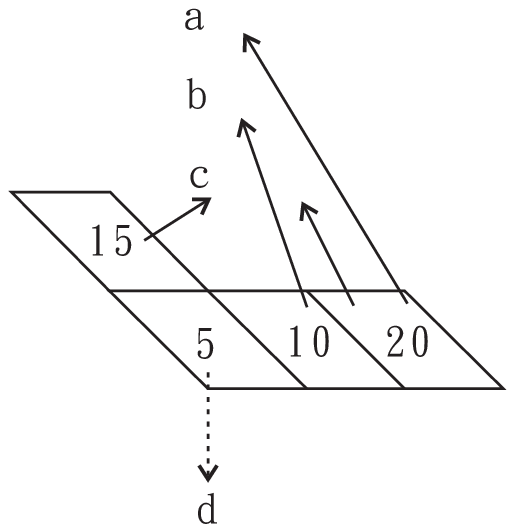}
\end{array} \\
& &\\
\text{(c) 3 upward arms and 3 downward arms} && \text{(d)  4 upward arms and 1 downward arm} \\
& &\\
\begin{array}{c}
\psfrag{a}{{\footnotesize $29$}}
\psfrag{b}{{\footnotesize$13$}}
\psfrag{c}{{\footnotesize$11\cdot 2$}}
\psfrag{d}{{\scriptsize$5 \cdot 4,5\cdot3$}}
\includegraphics[scale=0.8]{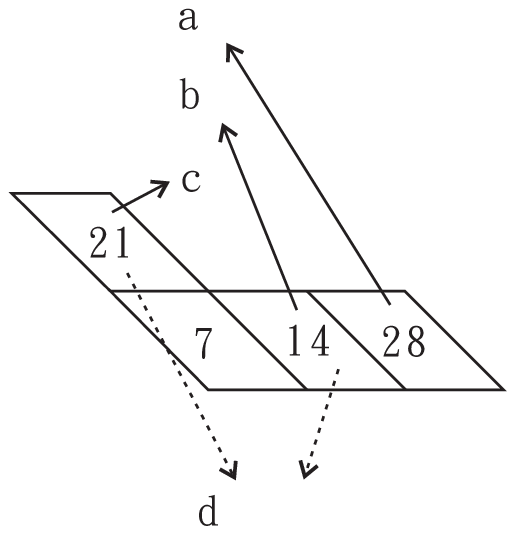}
\end{array}  & & \\
& &\\
\text{(e)  3 upward arms and 2 downward arms} &&  \\
\end{array}
\end{equation*}

\begin{equation*}
\text{Figure 4.4.}
\end{equation*}

\end{example}

Now, consider the splitting procedure (II) again. For the one-dimensional coupled system $\mathbb{X}_{2}^{A}$, the two-dimensional graph in Fig. 3.1 looks like a bamboo-blind. The horizontal lines $i\mathcal{I}_{2}$ are the bamboo sticks and the zigzag lines are the strings that tie the sticks together. For each $k\geq 1$, taking $\{iM_{k}\}_{i\in\mathcal{I}_{2}}$, the blind falls apart into infinitely-many pieces. These disjoint pieces are mutually independent and are used to estimate the lower bound of the entropy.

In the two-dimensional coupled system $\mathbb{X}_{2,3}^{A}$, for each $k\geq4$, the numbers in $G_{k}\equiv \left\{iM_{k}\right\}_{i\in\mathcal{I}_{2,3}}$ are taken as the vertices; the horizontal edges are segments in $iM_{k}$, and the vertical edges are segments that connect the consecutive numbers in $G_{k}$; see Fig. 4.2 (i) for $k=4$ and $1\leq i\leq 19$. Figure 4.2 (ii) is a projection of the graph in $\mathcal{I}_{2,3}$ for $k=4$ and $1\leq i\leq 41$.
From Propositions 4.2 and 4.3, for each $k\geq 4$, $G_{k}$ is a fully connected graph. This paper does not find any means to split $G_{k}$ into two parts:

\begin{equation*}
G_{k}=U_{k} \bigcup W_{k},
\end{equation*}
where $U_{k}$ is the set of mutually independent subgraphs of $G_{k}$, and $W_{k}=G_{k}\setminus U_{k}$, so (4.3) holds.

From Propositions 4.2 and 4.3, for $k\geq 4$, the connections among the vertices in $G_{k}$ are quite complicated. Then, the topology of $G_{k}$ is far away from $L_{k}\times \mathbb{Z}^{1}$, a standard three-dimensional lattice. Previous results concerning pattern generation problems cannot apply successfully \cite{5,6,7}. A better understanding of $G_{k}$ is required before dealing with $\mathbb{X}_{2,3}^{A}$.

%
%
%
%
%

%

\begin{equation*}
\end{equation*}

\textbf{Acknowledgement.} The numerical results of this paper are provided by Mr. Hung-Shiun Chen, a Ph.D. student of Song-Sun Lin at National Chiao Tung University.

\end{document}